\documentclass[a4paper,12pt]{article} 
     
     \usepackage[intlimits]{amsmath}
\usepackage{amsthm,amsfonts}
\usepackage{amssymb}
\usepackage{mathrsfs}
\usepackage[final]{graphicx,epsfig} 
\usepackage{longtable}
\usepackage{indentfirst}
\usepackage[utf8]{inputenc}

\usepackage[english]{babel}
\usepackage[usenames]{color}
\usepackage{subfigure}

\graphicspath{{img/}}

    \usepackage{geometry} 
 \usepackage[math]{easyeqn}
 \usepackage[mathscr]{eucal}

\usepackage{verbatim}

\theoremstyle{definition}

\newtheorem*{remark}{\noindent\normalsize\bf{}Remark}
\newtheorem{theorem}{\noindent\normalsize\bf{}Theorem}
\newtheorem{lemma}{\noindent\normalsize\bf{}Lemma}
\newtheorem*{cons}{\noindent\normalsize\bf{}Consequence}

\newcommand{\diag}{\ensuremath{\mathrm{diag}}}

\newcommand{\Var}{\ensuremath{\mathop{\mathbb{D}}}\nolimits}

\newcommand{\E}{\ensuremath{\mathrm{{\mathbb E}}}}

\newcommand{\Po}{\ensuremath{\mathrm{Po}}}

\usepackage{listings}

\usepackage{color}
\definecolor{gray}{rgb}{0.4,0.4,0.4}
\definecolor{darkblue}{rgb}{0.0,0.0,0.6}
\definecolor{cyan}{rgb}{0.0,0.6,0.6}

 \renewcommand{\(}{$\,}
\renewcommand{\)}{\,$}

\def\ub{u^{\flat}}

\def\gaussv{\gamma}

\def\gpsiidev{g''(\psi_i^{T} \theta)}
\def\rb0{r_0^{\flat}}

\def\nquad{\hspace{-1cm}}
\def\eqdef{\stackrel{\operatorname{def}}{=}}

\def\epsb{\varepsilon^{\flat}}

\def\bX{X^{\flat}}

\def\xiv{\bf{\xi}}

\def\tX{\tilde{X}}

\def\eps{\varepsilon}

\def\vp{\mathrm{v}}

\def\ND{\mathcal{N}}
\def\Ybb{\mathbb{Y}}

\def\dLhConv{\text{TP}_h}

\newcommand{\bb}[1]{\boldsymbol{#1}}

\renewcommand{\tilde}[1]{\widetilde{#1}}

\renewcommand{\Gamma}{\varGamma}
\renewcommand{\Pi}{\varPi}
\renewcommand{\Sigma}{\varSigma}
\renewcommand{\Delta}{\varDelta}
\renewcommand{\Lambda}{\varLambda}
\renewcommand{\Psi}{\varPsi}
\renewcommand{\Phi}{\varPhi}
\renewcommand{\Theta}{\varTheta}
\renewcommand{\Omega}{\varOmega}
\renewcommand{\Xi}{\varXi}
\renewcommand{\Upsilon}{\varUpsilon}

\def\Var{\operatorname{Var}}

\def\argmax{\operatornamewithlimits{argmax}}
\def\argmin{\operatornamewithlimits{argmin}}

\def\tr{\operatorname{tr}}

\def\R{I\!\!R}
\def\E{I\!\!E}
\def\P{I\!\!P}

\def\kappa{\varkappa}

\def\T{\top}
\def\diag{\operatorname{diag}}
\def\bldiag{\operatorname{blockDiag}}

\def\tX{\widetilde{X}}

\def\gtr{g_\triangle}

\def\nablaGLM{S}
\def\GLMLINK{A}
\def\GLMlink{g}

\def\aGLMlink{a_g}

\def\varepsilonv{\bb{\varepsilon}}

\def\xiv{\bb{\xi}}

\def\dLh{T_h}
\def\dLhb{T_h^{\flat}}
\def\dLb12{T_h^{\flat}(\theta_1^{\flat}, \theta_2^{\flat})}

\def\dLhconv{\text{TP}_h(\tau)}
\def\dLhconvb{\text{TP}_h^{\flat}(\tau)}
\def\dLhconve{\text{TP}_h^{\epsilon}(\tau)}

\def\localr{\Theta(\rr)}
\def\gradL{\nabla L}

\def\opttheta{\widehat{\theta}}

\def\rombb{\diamondsuit^{\flat}(\rr,\xx)}
\def\romb{\diamondsuit(\rr,\xx)}
\def\rombe{\diamondsuit^{\epsilon}(\rr,\xx)}

\def\alphab{\alpha^{\flat}}
\def\chb{\chi^{\flat}}
\def\alphab12{\alpha^{\flat}(\theta, \theta_0)}
\def\chib12{\chi^{\flat}(\theta, \theta_0)}
\def\Lbf{L^{\flat}(\theta)}
\def\Lb0{L^{\flat}(\theta_0)}
\def\Lbopt{L^{\flat}(\widehat{\theta})}
\def\Lfopt{L(\widehat{\theta})}
\def\Lf{L(\theta)}
\def\L0{L(\theta_0)}

\def\Eb{\E^{\flat}}
\def\Varb{\Var^{\flat}}

\newcommand{\normp}[1]{
\left\Vert #1 \right\Vert
}

\newcommand{\normop}[1]{
\left\Vert #1 \right\Vert_{\oper}
}

\newcommand{\vertiii}[1]{{\left\vert\kern-0.25ex\left\vert\kern-0.25ex\left\vert #1 
    \right\vert\kern-0.25ex\right\vert\kern-0.25ex\right\vert}}

\def\rddelta{\delta}

\def\Id{I\!\!\!I}
\def\Ind{\operatorname{1}\hspace{-4.3pt}\operatorname{I}}

\def\nsize{{n}}

\def\DP{D}
\def\DPc{\DP_{0}}

\def\gmiid{\mathtt{g}_{1}}

\def\xib{\xiv^{\flat}}
\def\xivb{\xiv_{\rdb}}

\def\ex{\mathrm{e}}

\def\gm{\mathtt{g}}
\def\gmc{\gm_{c}}
\def\gmb{\gm}

\def\xx{\mathtt{x}}
\def\xxc{\xx_{c}}

\def\rups{\rr_{0}}

\def\nunu{\nu_{0}}

\def\dimp{p}

\def\BB{I\!\!B}
\def\vA{\mathtt{v}}

\def\thetas{\theta^{*}}

\def\thetad{\theta^{\circ}}

\def\Thetas{\Theta}

\def\reps{\epsilon}
\def\eps{\epsilon}

\def\VP{V}
\def\VPc{\VP_{0}}

\def\vp{\mathbf{v}}

\def\lambdaB{{\lambda}^{*}}

\def\expzeta{\mathfrak{s}}

\def\rr{\mathtt{r}}

\def\zz{\mathfrak{z}}

\def\dimA{\mathtt{p}}

\def\lambdaB{\lambda_{\BB}}

\def\Po{\operatorname{Po}}

\def\dimB{\mathtt{p}_{\BB}}

\def\BB{B}

\def\vp{\mathrm{v}}

\def\epsb{\varepsilon^{\flat}}

\def\eps{\varepsilon}

\def\txiv{\tilde{\xiv}}

\def\fm{f}

\def\sbt{\circ}

\def\Pb{\P^{\sbt}}
\def\Eb{\E^{\sbt}}

\def\bxiv{\xiv^{\sbt}}

\def\sbt{\hspace{1pt} \flat}

\def\xivb{\xiv^{\sbt}}

\def\Varb{\Var^{\sbt}}

\def\oper{\operatorname{op}}

\def\UV{\mathcal{U}}

\def\zq{z}

\def\sbt{\hspace{1pt} \flat}

\def\xivb{\xiv^{\sbt}}

\def\Varb{\Var^{\sbt}}

\def\dPsi{\delta_{\Psi}}

\def\supA{\lambda}

\def\td{\delta}
\def\dimB{\dimA}
\def\vpB{\vp}
\def\lambdaB{\supA}
\def\zqc{\zq_{c}}

\def\epsb{\eps^{\sbt}}

\usepackage{flowchart}
\usetikzlibrary{arrows}

\usepackage[round,sort]{natbib}
\usepackage{lipsum}
\usepackage{blindtext} 

\title{Bootstrap for change point detection}

\author{Nazar Buzun, Valeriy Avanesov \\ \{buzun,avanesov\}@wias-berlin.de}

\begin{document}
  
  \maketitle
  
  \begin{flushright}
  \end{flushright}
  
  \section{Introduction} 
  The problem of change point detection appears each time one needs to explore a set of random data and make a decision about homogeneity of its structure. In other words, the problem can be stated as two following questions: were there any structural changes in the nature of observed data? At which moments, if so?  The present work mainly focuses on the \textit{sequential} or \textit{online} change point detection. In this case the data is aggregated from running random process. 
Formally a time moment $\tau$ is a \textit{change point}, if stochastic properties of the observed signal $\{ Y_t \}_{t=1}^{n}$  have undergone changes in its distribution: 
\[
\begin{cases}
Y_t \backsim \P_1 & t < \tau, \\ 
Y_t \backsim \P_2 & t \geq \tau. 
\end{cases} 
\]
The goal is to find such structural breaks as soon as possible. Such problem arises across many scientific areas: quality control \cite{lai1995sequential}, cybersecurity \cite{Cyber1}, \cite{Cyber2}, econometrics \cite{SpokoinyCP}, \cite{Econom2}, geodesy e.t.c. 
Article \cite{shiryaevOptimum} describes classical results in change point detection theory. Overview of the state-of-art methods are presented in \cite{ReviewPolun} and \cite{Shiryaev}.

This research considers sequential hypothesis testing, in which each hypothesis ($\P_1 = \P_2 $) monitors the presence of change point  through Likelihood Ratio Test (LRT) using \textit{sliding window}.  At each time step the procedure extracts a data slice, splits it in two parts of equal size and executes LRT on it. High values of LRT indicate  possible distribution difference in the window parts $(\P_1 \neq \P_2)$. 
Procedures with LRT are rather popular in related literature. 
The work \cite{quandt1960tests} proposes application of LRT for detection of breaks in linear regression model. It was further developed by many authors, e.g. \cite{haccou1987likelihood}, \cite{srivastava1986likelihood}. 
Papers \cite{liu2008empirical}, \cite{zou2007empirical} investigate LRT for change point detection for nonparametric case. Nonparametric approaches are  easily adaptable for complex data but  in general they need more information for model building  than their parametric alternatives. 
Introduction of \textit{parametric assumption}: $\P_1, \P_2 \in \{\P(\theta): \theta \in  \R^p\}$ allows to reduce the suffisient number of observations as soon as $\P(\theta)$ has less degrees of freedom than nontapametric model. The state-of-the-art review of parametric models based on LRT and its application to economics and bio-informatics are presented by \cite{ParStatChen}.  The paper \cite{gombay2000sequential} explores how LRT can be used for sequential change point detection in case $\P(\theta)$ is exponential family.  
 
The LRT  statistic requires its quantiles or critical values to be set from the signal data~$\{ Y_t \}_{t=1}^{n}$.   Many works are dedicated to asymptotic behaviour of LRT, e.g. \cite{jandhyala1999capturing} obtains lower and upper bounds for distribution of asymptotic maximum likelihood estimator. The work \cite{kim1994tests} provides a very detailed study of its asymptotic behaviour in linear regression models. Similar results for change in mean of a Gaussian process are given in \cite{fotopoulos2010exact}. In \cite{biau2016} an approach with Wiener process and Donsker--Prohorov Theorem   describes relatively general method for LRT-like statistics distribution approximation.

Instead of asymptotic distribution for LRT one may find a benefit of  resampling and \textit{bootstrap}. This technique is popular, e.g. \cite{multiscaleCP1}, \cite{SpokoinyCP}, since it provides a way to simulate a complex distribution of LRT statistic (for wide family of $\P(\theta)$) through empirical data distribution. 
Using bootstrap one can generate  LRT$^{\flat}$ statistic multiple times in order to obtain quantile distribution of the initial LRT. Both LRT and LRT$^{\flat}$ statistics have  (ref. Sections \ref{lrt_main} and \ref{boot_WF}) approximation with the following forms with high probability
\[\tag{Qf}
\label{lrt_qf}
\text{LRT} \approx \| \xiv + \Delta \|, 
\quad 
\text{LRT}^{\flat} \approx \| \xiv^{\flat} + \Delta^{\flat} \|.
\]
Larger $\Delta$ values correspond to more pronounced hypothesis  rejection  (more  apparent changes in data sequence). Argument $\xiv$ could be treated as a noise component. For LRT critical value calibration one requires data without change points and consequently with $\Delta = 0$. 
Section \ref{sec:procedure} contains description of a modified LRT which enable the calibration even if data contains change points.

The cornerstone of the novel change point detection procedure is the concept of change-point pattern. The geometry of a pattern depends on a type of transition region between two distributions that the data obeys before and after a change respectively. Three examples are presented at the Fig.~\ref{fig:patterns}. The triangle (spades) pattern appears in case of an abrupt transition from $\P(\theta_1)$ to $\P(\theta_2)$. A smooth transition between two distributions  entails trapezium change-point pattern. And a horn pattern appears due to an abrupt change in variance. Processing of a change-point pattern instead of a single LRT-value allows to reduce noise influence $\xi(t)$ and false-alarm rate. The presence of change-point patterns is the corollary of (\ref{lrt_qf}) representation.

\begin{figure}[!h]
    \centering
    \includegraphics[width=0.4\textwidth, height=0.25\textwidth]{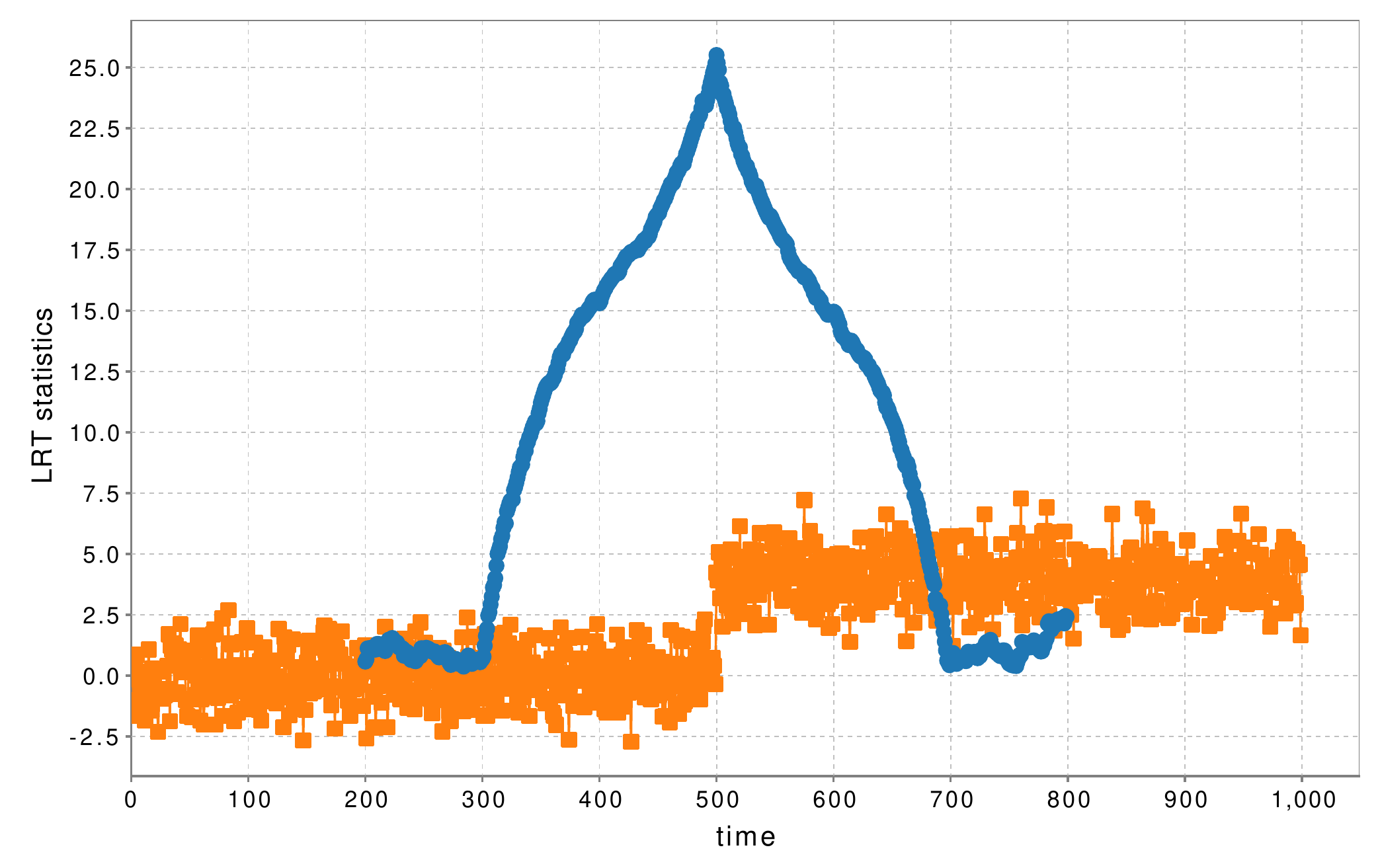}
    \includegraphics[width=0.4\textwidth, height=0.25\textwidth]{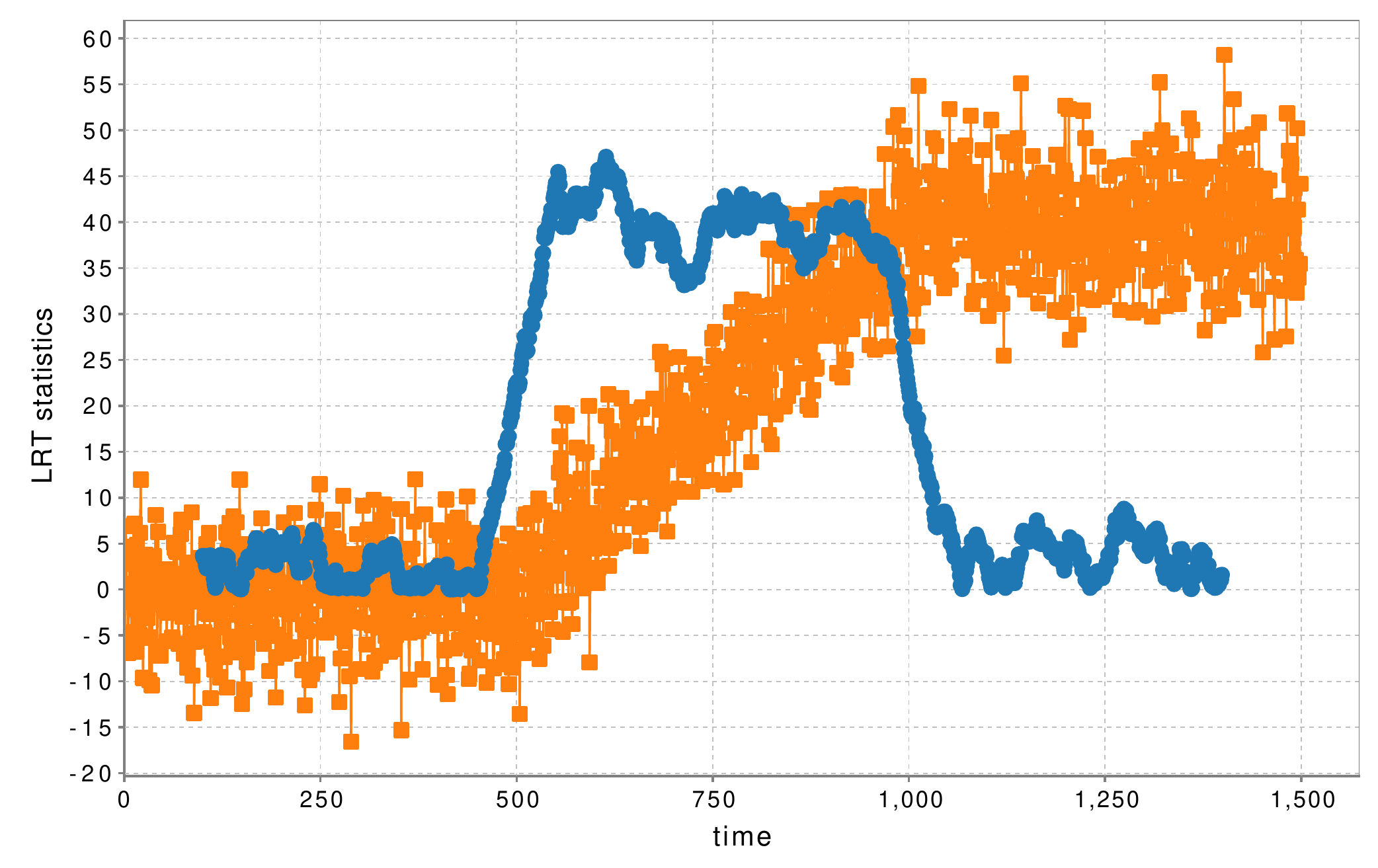} \\
     \includegraphics[width=0.4\textwidth, height=0.25\textwidth]{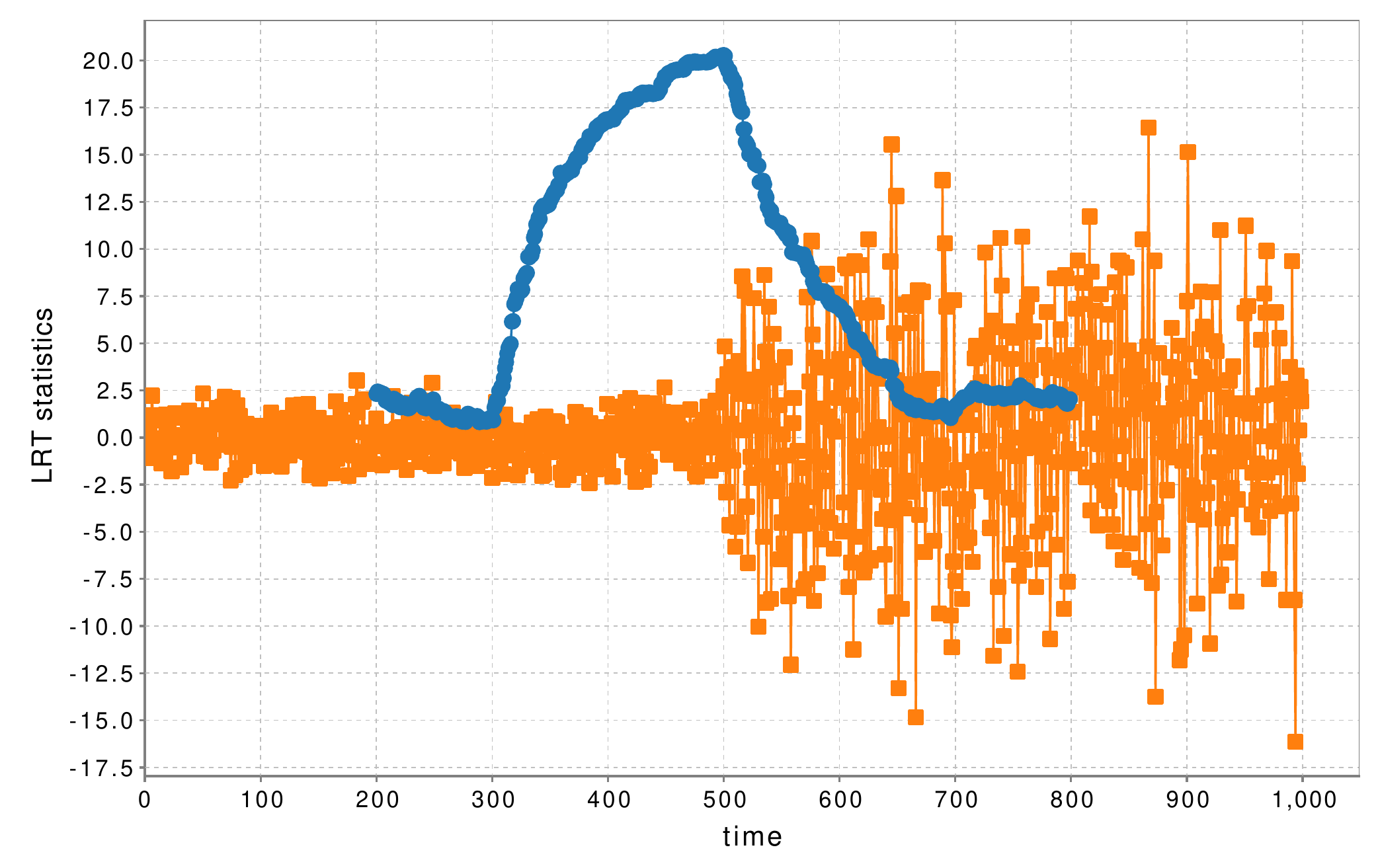}
    \caption{Types of change point and the geometry of change-point patterns: triangle pattern -- abrupt mean transition, trapezium pattern -- smooth mean transition, horn pattern -- abrupt variance transition. }
    \label{fig:patterns}
\end{figure}
In case of a single change point one may  find the pattern position by maximising convolution with a  pattern function $P_\tau(t)$ (ref. Section \ref{sec:procedure} for details): 
\[ \argmax_\tau \sum_{t} P_\tau (t) \|  \xiv(t) + \Delta(t) \|. \]

In order to set critical value correctly quantiles  of the statistic $\max_\tau \sum_{t} P_\tau (t) \|  \xiv(t) \|$ should be close in distribution to quantiles of $\max_\tau \sum_{t} P_\tau (t) \|  \xivb(t) \|$.  Assuming that $\xiv(t) \approx \sum_{i = t -h}^{t+h} \xiv_i$ (independent random vectors sum),  we have made  probability measures comparison using technique from article \cite{GarMaxSum}.  In Section \ref{sm_max_sec} we prove Bootstrap approximation  illustrating on the way useful mathematical concepts  such as Linderberg telescoping sums, anti-concentration of normal vector, Slepian bridge and empirical covariance matrix deviations. Section \ref{complex_max} extends  the statements  for statistics of  type $\max_{t} Q_t(\xiv)$ (in particular  $\max_\tau \sum_{t} P_\tau (t) \|  \xiv(t) \|$, Section~\ref{lrt_boot_prec}).   

The last part of this paper (Section \ref{glm_sec}) contains some specification for aforementioned results with  generalised linear models (GLM).   
  \section{Procedure}
  
\label{sec:procedure}
This section provides  description of the Change Point Detection algorithm which employs Likelihood Ratio Test (LRT). Let $(\P(\theta), ~\theta \in  \R^p, \; L(\theta) = \log (\partial^n\P(\theta)/ \partial Y) )$ be a parametric assumption about the nature of data inside the window $(Y_{t-h},\ldots, Y_{t + h - 1})$ with central point $t$ and size $2h$. Here and further we assume, that the observations $\{Y_i\}_{i=1}^n$ are independent, so
\[\tag{L}
L(\theta, \Ybb) = \sum_i l_i(\theta).
\]
Denote argmax of the Likelihood function and the ``real'' model parameter  value as follows 
\[
\opttheta = \argmax_{\theta} L(\theta, \Ybb),
\quad
\thetas = \argmax_{\theta} \E L(\theta, \Ybb).
\]
The  algorithm sequentially computes  LRT statistic ($\dLh(t)$) for each~$t$ in the sliding window procedure. 
The LRT statistic itself corresponds to the gain from window split into two parts ($\Ybb_l, \Ybb_r$):  
\[\tag{T}\label{lrt_stat}
\dLh(t) = L(\opttheta_l , \Ybb_l) + L(\opttheta_r; \Ybb_r) - L(\opttheta, \Ybb) ,
\]
\[
\Ybb_l = (Y_{t-h},\ldots, Y_{t - 1}),
\quad 
\Ybb_r = (Y_{t},\ldots, Y_{t + h - 1}),
\]
\[
\opttheta_l = \argmax_{\theta} L(\theta, \Ybb_l), 
\quad
\opttheta_r = \argmax_{\theta} L(\theta, \Ybb_r)
\]
According to the Theorem \ref{lrt_th}, encountering change point, statistic $ 2 \dLh(t) \approx \normp{ \xiv(t) + \Delta(t) }^2$ starts growing according to change point pattern type (for example spades, trapezium, horn, ref. the Figure \ref{fig:patterns}). In order to match pattern  positions, the procedure monitors $2h$ values of the LRT simultaneously and convolves them with each of the predefined pattern functions $P_\tau(t)$:
\[\tag{TP}\label{tp_stat}
\dLhConv(\tau) = \sum_{t}^{} P_\tau(t) \sqrt{2 T_{h}(t) }.
\]
High values of $\dLhConv(\tau)$ correspond to a sufficient correlation of $\sqrt{2 T_{h} }$ and $P_\tau$  (similar to the dependence on~$t$).
The algorithm marks a time moment $\tau$ at a scale $h$ as a change point, if the test statistic $\dLhConv(\tau)$ exceeds a calibrated (by bootstrap procedure) critical value $z_h$:
\[
\{\tau \text{  is  a   change   point }\} \Leftrightarrow \{\exists h:  \dLhConv(\tau) > z_h\}.
\]
The greater window size $h$ is chosen,  the more probably the algorithm  will mark $\tau$ as a change point. Again, small windows may mark $\tau$  faster.

\textit{Weighted bootstrap procedure} enables resampling of the statistic $\max_{1 \leq  \tau \leq n}  \dLhConv(\tau) $  and thus calculation of the critical value $z_h$ for the window size $2h$. It generates a sequence of weighted likelihood functions, where each element is a convolution of independent likelihood components  and weight vector $(\ub_1,\ldots,\ub_{n})$:
\begin{equation*}\label{Lb}\tag{Lb}
L^{\flat}(\theta, \Ybb) = \sum_{i} \ub_i l_i(\theta),
\end{equation*}
where $\{\ub_i\}_{i = 1}^n$ are i.i.d. and $\ub_i \in \ND(1,1)$.
At each weights generation one gets a new value of $L^{\flat}(\theta)$ and its optimal parameter $\theta^{\flat}$ and thus bootstrap procedure enables to estimate $L(\opttheta)$ fluctuations.  The corresponding bootstrap LRT statistic is
\begin{align*}\label{Tb}\tag{Tb}
T_{h}^{\flat}(t) &= L^{\flat}(\theta_l^{\flat}, \Ybb_l) + L^{\flat}(\theta_r^{\flat}, \Ybb_r) -\sup_{\theta}\{L^{\flat}(\theta, \Ybb_l) + L^{\flat}(\theta + \opttheta_r - \opttheta_l, \Ybb_r)\},
\end{align*}
\[
\theta^{\flat} = \argmax_{\theta} L^{\flat}(\theta, \Ybb).
\]
Parameter $(\opttheta_r - \opttheta_l)$ is required for  condition  $T_{h}^{\flat} \approx \normp{\xi^{\flat}}$ (ref. Theorem \ref{lrt_boot}). In this  case one can estimate $\max_{1 \leq  \tau \leq n} \dLhConv^{\flat}(\tau) $ quantiles under the null hypothesis  $(\Delta^{\flat}(t) \propto \opttheta_r(t) - \opttheta_l(t))$ instead of the false assumption $(\Delta^{\flat}(t) = 0)$.

\textit{Empirical bootstrap} version generates subsamples of data $\{Y_k\}$ from the complete dataset with random independent indexes of size $n$. In this case 
\[\tag{Le}\label{Le}
L^{\epsilon}(\theta, \Ybb) = \sum_{i }^{} l_{k(i)}(\theta),
\]
where $\{k(i)\}_{i = 1}^n$ are i.i.d. and $k(i) \in \{ 1, \ldots, n \}$. For all window positions $\opttheta_r = \opttheta_l = \opttheta$ and here  bias correction is not required. So the corresponding  LRT statistic is like (\ref{lrt_stat}):
\begin{align*}\label{Tbe}\tag{Te}
T_{h}^{\epsilon}(t) &= L^{\epsilon}(\theta_l^{\epsilon}, \Ybb_l) + L^{\epsilon}(\theta_r^{\epsilon}, \Ybb_r) - L^{\epsilon}(\theta^{\epsilon}, \Ybb),
\end{align*}
\[
\theta^{\epsilon} = \argmax_{\theta} L^{\epsilon}(\theta, \Ybb).
\]
Empirical bootstrap works better in the application but less suitable for theoretical investigations (the distribution is discontinuous).

  \section{Main  results}

Below we present the Theorems that describes difference between probabilistic measures of $\dLhconv$ and   $\dLhconvb$ (precision of the bootstrap calibration) and LRT  sensitivity to parameter $\thetas$ transition at change point. In independent models  
each noise vector $\xiv_{lr}(t) = \xiv(t) \in \R^{p}$ is a sum of independent vectors (ref. Section \ref{lrt_main} for $\xiv_{lr}(t)$ definition)
\[
\xiv_{lr}(t) = \sum_{i = t - h}^{t -1 } \xiv_i - \sum_{i = t }^{t + h -1} \xiv_i,
\quad 
\xiv_i \propto \nabla l_i(\thetas).
\]
Aggregate all $\xiv_i$ into one vector  
\[
\xiv^T = (\xiv_1^T, \ldots , \xiv_n^T).
\]
\begin{theorem} 
\label{boot_approx_main}
Let dataset size be $n$ and the window equal to $h$. Include conditions from lemmas \ref{lrt_th}, \ref{lrt_boot} and \ref{boot_approx}. Then for each fixed $z$
 \[
  \left | \P \left( \max_{1 \leq \tau \leq n} \dLhconv > z  \right) - \Pb \left( \max_{1 \leq \tau \leq n} \dLhconvb  > z  \right)  
\right | \leq \Delta_{TV}  
\]
\[
\Delta_{TV} =   C_1 \mu_Z +  C_2 \|\Var(\xiv_{lr}) - \Varb(\xivb_{lr})  \|^{1/2}_{\infty} + 7  C_A ( \diamondsuit  +  \diamondsuit^{\flat}  ) ,
\]
\end{theorem}
where
\[
\| \Var(\xiv_{lr}) - \Varb(\xivb_{lr})  \|_{\infty} \leq  10 \sqrt{\log(np)} \sqrt{2h}   \| \Var(\xiv) \|_{\infty}  ( 3 + \| b \|) +  \|b \|^2, 
\]
\[
 \| b \|^2 = \max_{t} \sum_{i = t}^{t + 2h} \| \E \xiv_i \|^2_{\infty}	, 
\]
\[
\mu_Z^{3} \leq 2h \E \normp{\xiv}^3_{\infty} ,
\]
\[
C_1 = 5 C_{\mu}^{1/3} C_A, 
\quad 
C_2 = 4 C^{1/2}_{\Sigma} C_{A}.
\]
Constants  $ C_{\mu} \sim \log^2 (n) $, $ C_{\Sigma} \sim \log (n) $ and $C_A \approx  p^{3/2} \log (n)$  are described in Section \ref{lrt_boot_prec}. 

The proof is a direct consequence of Theorems \ref{lrt_th}, \ref{lrt_boot} and \ref{boot_approx}.

\begin{remark}

\begin{enumerate}

\item Parameters   asymptotic 
\[
\mu_Z  \sim  \frac{\log^{1/2}(n)}{(2h)^{1/6}}, 
\quad
\|\Var(\xiv_{lr}) - \Varb(\xivb_{lr})  \|^{1/2}_{\infty} \sim  \frac{\log^{1/4}(n)}{(2h)^{1/4}}, 
\quad
\diamondsuit  +  \diamondsuit^{\flat} \sim \frac{p}{h^{1/2}}.
\]

\item For quantile estimation of the statistic $ \max_{1 \leq \tau \leq n} \dLhconv $ with  quantile of $\max_{1 \leq \tau \leq n} \dLhconvb $ one has to show that 
\[
\left | \P \left( \max_{1 \leq \tau \leq n} \dLhconv > z^{\flat}(\alpha) \right) - \alpha
\right |  \leq \Delta_{TV}, 
\]
for $z^{\flat}(\alpha)$ defined by equation 
\[
\Pb \left( \max_{1 \leq \tau \leq n} \dLhconvb  > z^{\flat}(\alpha)  \right)  = \alpha.
\]
This statement is a consequence of the Theorem (\ref{boot_approx_main}) but not a direct one since the argument $z^{\flat}(\alpha)$ is random and depends on $\max_{1 \leq \tau \leq n} \dLhconv$. Involving sandwich Lemma \ref{sandwich} fulfills this issue.

\end{enumerate}

\end{remark}

The next part of this Section evaluates the smallest parameter $\thetas$ transition that is sufficient for change point detection in a fixed position $\tau$ and window size $2h$. Let $z_h(\alpha)$ be a  quantile of $\sum_t P_\tau (t) \| \xiv_{lr}(t) \|$ such that
\[
\P \left( \sum_t P_\tau (t) \| \xiv_{lr}(t) \| > z_h(\alpha) \right  ) = \alpha  \leq \frac{\Var (\sum_t P_\tau (t) \| \xiv_{lr}(t) \| )}{( z_h(\alpha) - \sum_t P_\tau (t) \E \| \xiv_{lr}(t) \|   )^2}.
\] 
Section \ref{tpvar} provides upper bound for $\alpha$ and is  summarized in following statement.

\begin{theorem} Let $\sum_t P_\tau(t) = 0$ and $\E \| \xiv_{lr}(t) \|  = \sqrt{p}$.
The sufficient condition for abrupt type change point detection of size $\Delta$ with probability $1 - e^{-\xx}$  in position $\tau$ using triangle pattern (\ref{tr_patt})~is
\[
\normp{D_{lr} (\theta^*_r - \theta^*_l)(\tau) } =  \Delta >  5  p^{1/4} (\xx + \log(2h) )^{1/4} e^{\xx/2} + 21 \diamondsuit, 
\]
where matrix $D_{lr}$ is defined in Theorem \ref{lrt_th}.
\end{theorem}

  \section{Experiments}
  In order to substantiate patterns utility we compare procedure from Section \ref{sec:procedure} with the similar one but without pattern (i.e. $P_\tau(t) = \Ind [\tau = t]$). The experiment scenario is follows. The dataset $\{Y_i\}$ consists of $500$ normal random vectors from $\R^{5}$  with one change point at position $\tau^* = 250$. 
\[
Y_{i} \in \ND(0, I_{5}),  
\quad
0 \leq i < 250
\]
\[
Y_{i} \in \ND(0.25, I_{5}),  
\quad
250 \leq i < 500.
\]
The procedure searches for change point location as $\widehat{\tau} =  \argmax_{\tau} \dLhConv(\tau) $. Then the quality of the detection is measured by average error $|\widehat{\tau} - \tau^* |$ (c.p. position error) and fraction of the detected change points  (POWER) (ref. Figure \ref{fig:errpower}). 

\begin{figure}[!h]
    \centering
    \includegraphics[width=0.45\textwidth]{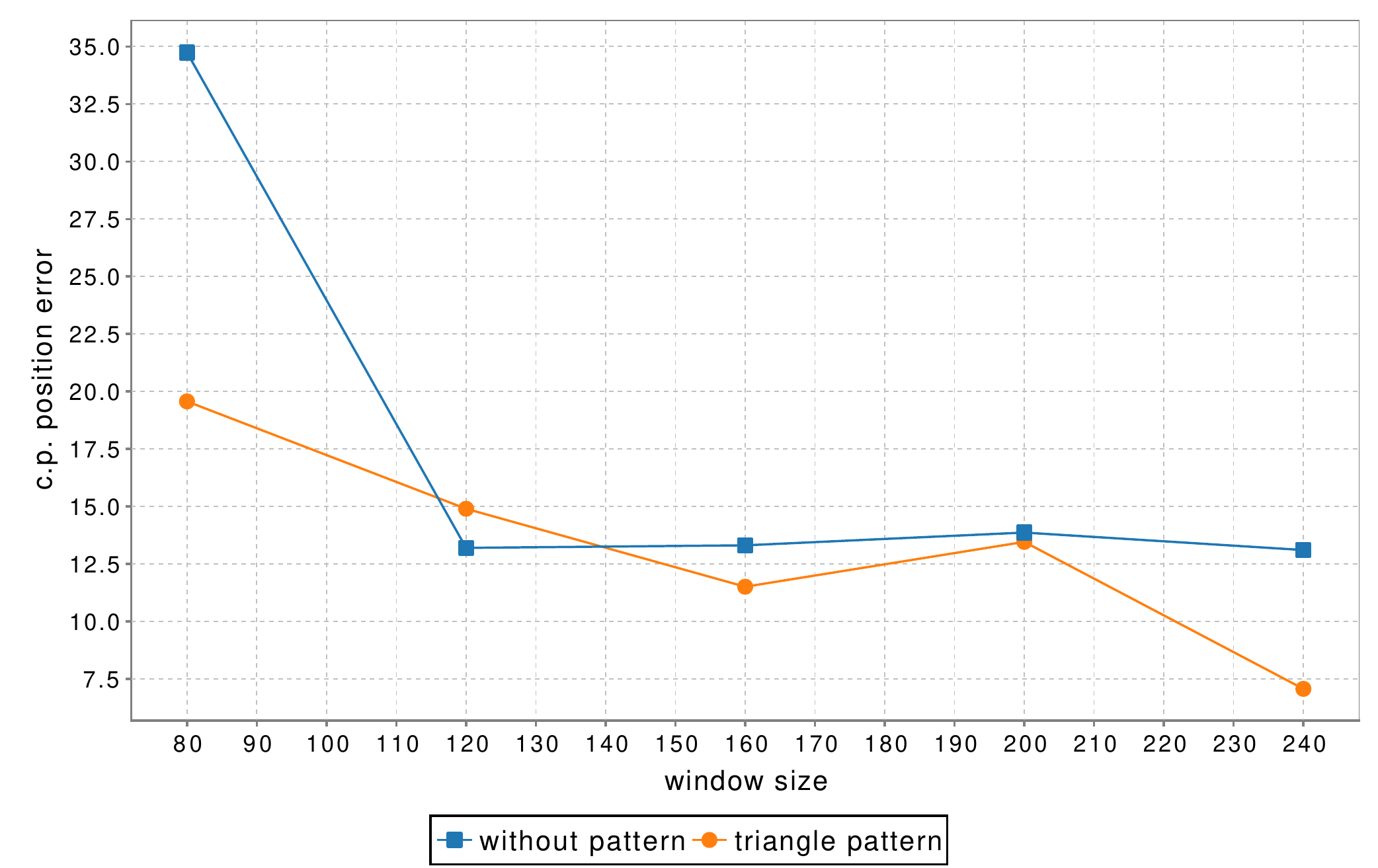}
    \includegraphics[width=0.45\textwidth]{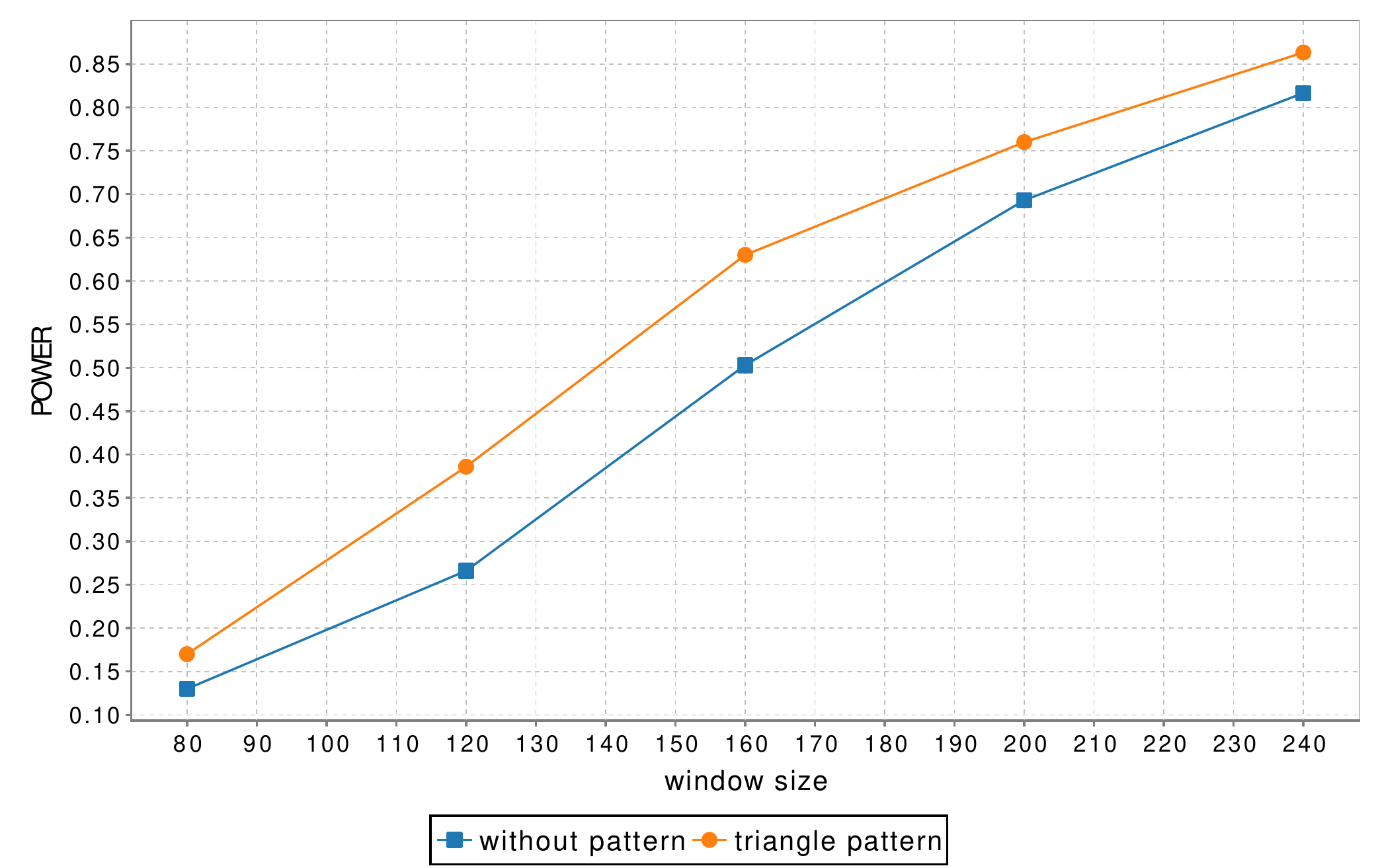}
    \caption{Change point localisation test and power test for the procedure from Section \ref{sec:procedure}. One case with triangle pattern and the other case without pattern.}
    \label{fig:errpower}
\end{figure}

The second experiment describes bootstrap convergence depending on window size ($2h$).  We set bootstrap confidence level equal to $0.1$ and compute p-value from real distribution with bootstrap quantile $z^{\flat}$. 
\[
\Pb \left( \max_{1 \leq \tau \leq n} \dLhconvb  > z^{\flat} \right)   = 0.1,
\]
\[
\left | \P \left( \max_{1 \leq \tau \leq n} \dLhconv > z^{\flat}  \right) - 0.1
\right |  = O\left(\frac{1}{h^{\beta}} \right).
\]  
From the plot below (ref. Figure \ref{fig:bootvalid}) one can observe that 
\[
\beta > \frac{1}{2},
\]
which suppose better convergence in comparison with the theoretical study (ref. Theorem \ref{boot_approx_main}), where $\beta  = 1/6$.
\begin{figure}[!h]
    \centering
    \includegraphics[width=0.8\textwidth]{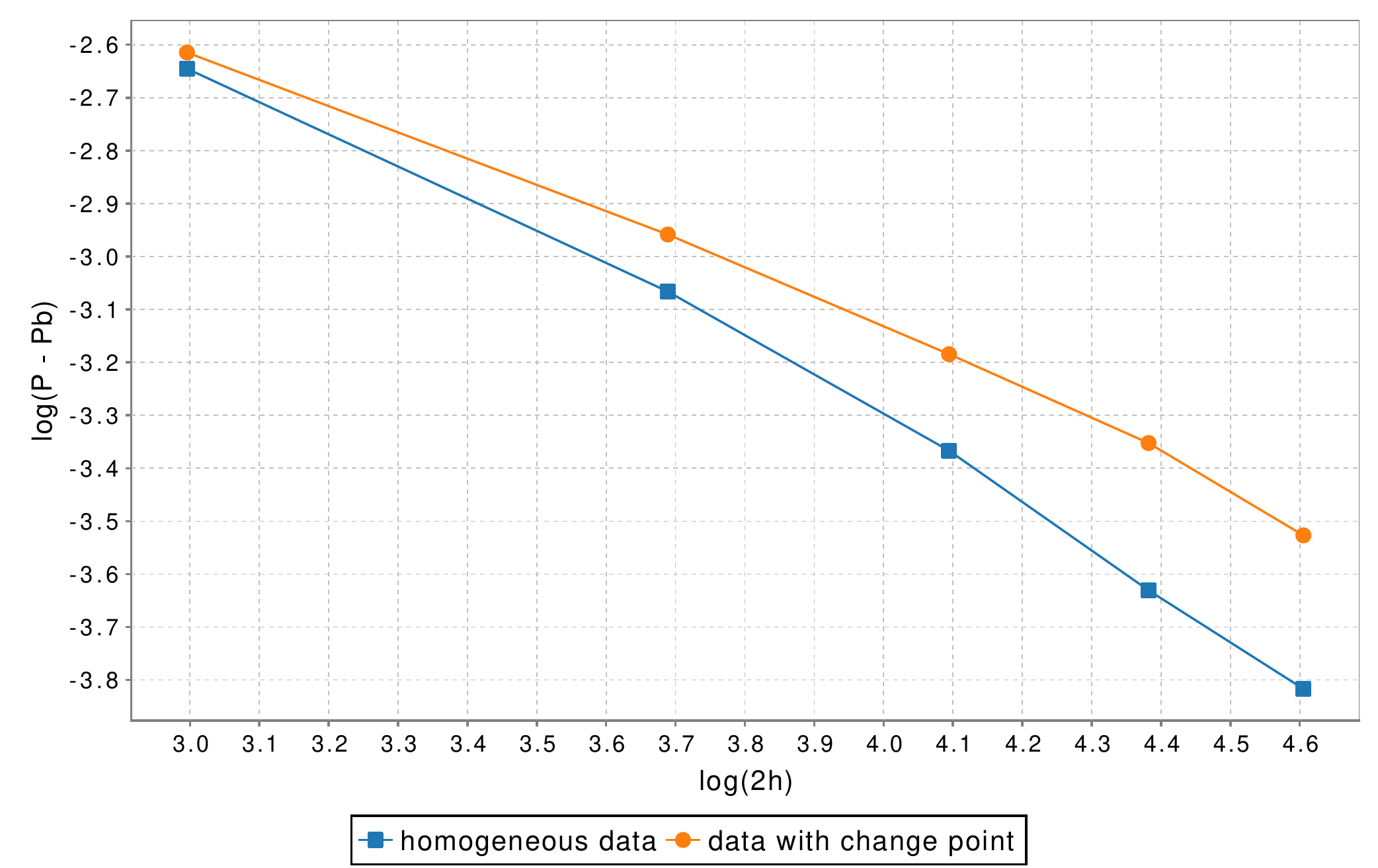}
    \caption{Bootstrap convergence.  Homogeneous data:  $Y_{i} \in \ND(0, I_{p}),  0 \leq i < 6h$, data with change point: $Y_{i} \in \ND(0, I_{p}),  0 \leq i < 3h$, $ $  $Y_{i} \in \ND(0.3, I_{p}),  3h \leq i < 6h$. The parameters are $p = 30$, $h \in \{ 10, 20, 30, 40, 50 \}$. }
    \label{fig:bootvalid}
\end{figure}

The last experimental part presents results of the comparison of the proposed algorithm of change point detection (LRTOffline) \ref{sec:procedure} with two other methods: Bayesian online changepoint detection (BOCPD) from \cite{BayesOnlineWeb} and cpt.meanvar(PELT,$\ldots$) (RMeanVar) from  R package. 
The first method is constructed for online inference, but so far as it returns CP location with each CP signal, it is also applicable for offline testing scenario. The idea of this method is predictive filtering: its forecasts a new data point using only the information have been observed already, where the distribution family is fixed (Normal for the tests in this paper). Bayesian inference calculates the length of the observed data (from the last CP).
The second algorithm also uses preliminary specified model. Its design focuses into finding multiple changes in mean and variance in Normally (another distributions also supported) distributed data. The returned set of change points is the result of sequential testing $H_0$ (existing number of change points) against $H_1$ (one extra change point) applying the likelihood ratio statistic of the whole data coupled with the penalty for CP count.  RMeanVar performs better than well known method CUSUM  due to synchronous changes in both  data parameters mean and variance.

Quality of measurements  uses  Normalised Mutual Information (NMI). The next equation defines 
NMI measure of two partitions ($X$, $Y$) of time range by change points
\[
\text{NMI}(X,Y) = 2 \frac{H(X) + H(Y) - H(X,Y)}{H(X) + H(Y)}.
\]
$H(X)$ and $H(X,Y)$  and entropy functions.  Higher NMI values (they are in $[0,1]$) correspond to better quality. 

Synthetic test data have been generated with different values of the distribution parameter transition ($\Delta$).  Each $\Delta$ value corresponds to 10 sampled data sequences over which one compute measure average.  Each data sequence has two, one or none change points.   The data has two distributions: normal $(\ND(\theta(1), \theta(2)))$ and Poisson $(\Po(\theta))$.  Parametric assumption for all  methods is $\ND(\theta(1), \theta(2))$, so Poisson data corresponds to misspecification  scenario.

\begin{figure}[ht!]
     \begin{center}
        \subfigure{
            \label{fig:Precision_delta_N}
            \includegraphics[width=0.45\textwidth]{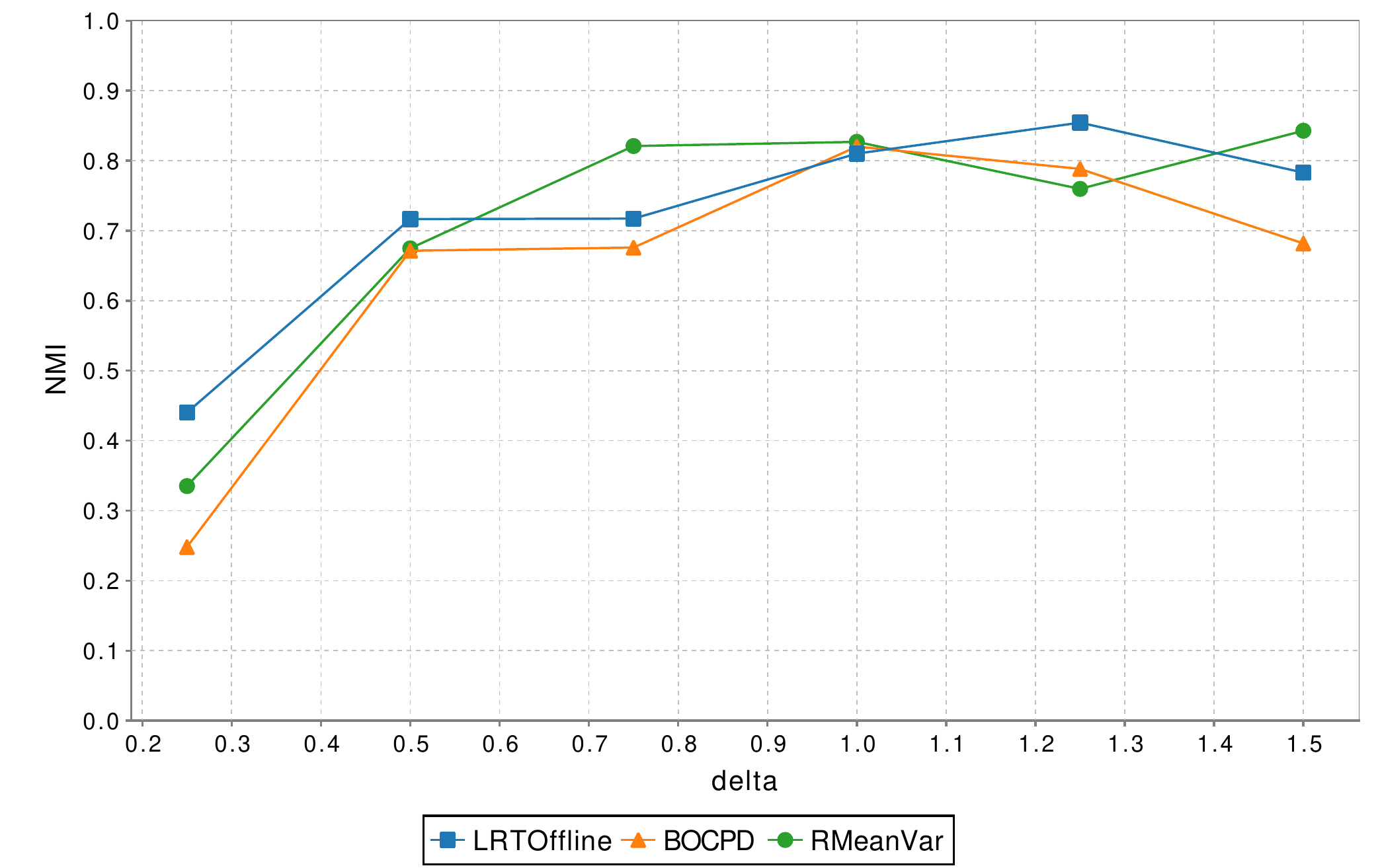}
        }
        \subfigure{
           \label{fig:Recall_delta_N}
           \includegraphics[width=0.45\textwidth]{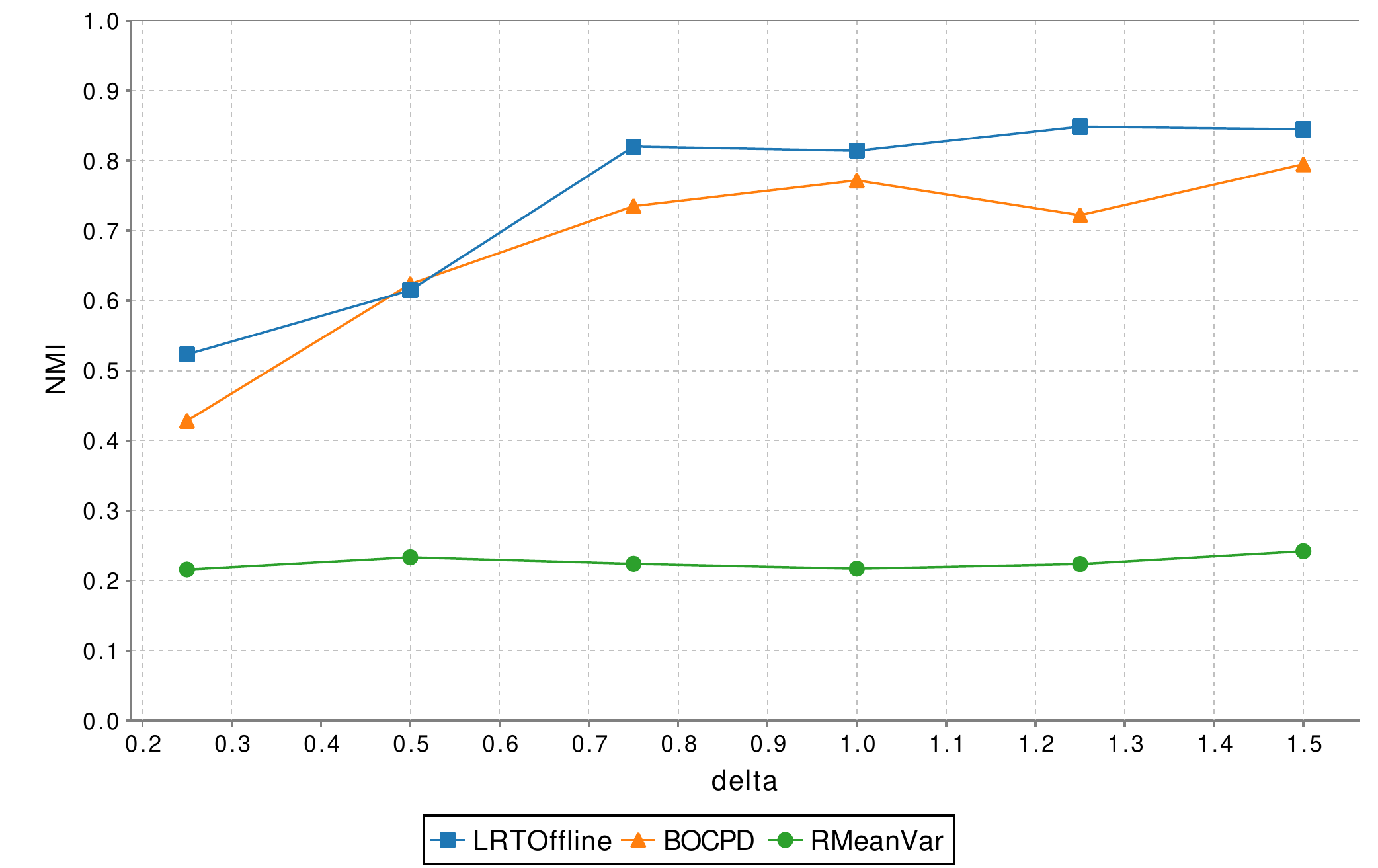}
        }

    \end{center}
    \caption{
         First data: $\ND(\theta(1), \theta(2))$, second data: $\Po(\theta)$,
         data size = 340, parametric assumption for all methods is $\ND(\theta(1), \theta(2))$, NMI -- Normalized Mutual Information between predicted and reference partitions of time interval with change points, change points count per test  $\{0,1,2\}$.
     }
   \label{fig:subfigures}
\end{figure}

In the tests with normal data all the methods achieves similar NMI scores. In the tests with Poisson data (misspecification) RMeanVar has relatively low quality and LRTOffline outperforms slightly BOCPD method. 

Change point detection algorithm (LRTOffline) implementation link: 

https://github.com/nazarblch/cpd

  \bibliography{references}{}
\section{Appendix}
  
  \subsection{Likelihood function restrictions}

\label{likelihood}

Assume that for each data subset $\Ybb = \{Y_i\}_{i=t}^{t + h}$ the likelihood function  $L(\theta) = L(\theta, \Ybb)$  has rather precise approximation by its quadratic Tailor expansion in a local region $\localr$
with  central point $\theta^*$: 
\[
L(\theta) \approx L(\theta^*) + \xiv^T D(\theta - \theta^*) - \frac{1}{2} \normp{D(\theta - \theta^*)}^2,
\]
where
\[
\xiv = D^{-1} \nabla L(\theta^*),
\quad
D^2(\theta) = - \nabla^2 \E L (\theta),
\quad
D = D(\theta^*).
\]
The following functions characterize the error of quadratic approximation:
\[
  \alpha(\theta, \theta_0) = L(\theta) - L(\theta_0)   - (\theta - \theta_0)^T \gradL( \theta_0) +  \frac{1}{2} \Vert D (\theta - \theta_0) \Vert^2, 
\]
\[
\chi(\theta, \theta_0) = D^{-1} \nabla \alpha(\theta, \theta_0) 
= D^{-1} (\gradL(\theta) - \gradL( \theta_0) ) +  D (\theta - \theta_0). 
\]
Let with probability $1 - e^{-\xx}$ their  upper bounds   in region $\localr$  satisfy  conditions 
\begin{equation}\label{cond_A}\tag{A}
\frac{| \alpha(\theta, \theta^*)  |}{\Vert D(\theta - \theta^*) \Vert} \leq \romb,  
\quad
  \Vert \chi(\theta, \theta^*) \Vert \leq  \romb,
\end{equation}
where \[\romb = \{\delta (\rr) + 6 \nunu \omega \zz(\xx, p)  \} \rr,\]
\[
\zz(\xx, p) = \sqrt{6p} + \sqrt{2\xx} + \frac{12p}{g} , 
\]
\[
\omega \sim \frac{1}{\sqrt{h}}, 
\quad
\delta(\rr) \sim \frac{r}{\sqrt{h}},
\quad
g \sim \sqrt{h}.
\]
The stochastic part of the likelihood for independent data has denotation  
\[\tag{Zeta}
\zeta(\theta) = L(\theta) - \E L(\theta)  = \sum_{i = 1}^n \zeta_i(\theta)  
\]
The next restriction for the Fisher matrix $-D^2(\theta)$ and components $\zeta_i(\theta, Y_i)$ deviations ensure condition (\ref{cond_A}).
\begin{equation}\label{cond_dD}\tag{dD}
\Vert D^{-1} D^2(\theta) D^{-1} - I_p\Vert \leq \delta(\rr),
\end{equation}
For all $\normp{\gamma_1} = \normp{\gamma_2} = 1$ and $|\lambda | < g$ and $ i$
\begin{equation*}\label{EDi}\tag{ED2i}
\log \E \text{exp} \left\{ 
\frac{\lambda}{\omega} \gamma^T_1 D^{-1}  \nabla^2 \zeta_i (\theta) D^{-1} \gamma_2 
\right\}
\leq \frac{\lambda^2 \nu_i^2}{2}, 
\quad \sum_i  \nu_i^2  = \nu_0^2.
\end{equation*} 
Condition (\ref{cond_dD}) is responsible for the quadratic approximation of $\E L(\theta)$, so 
\[
\| \nabla \E L(\theta) - \nabla \E L(\thetas)  - D(\theta - \thetas) \| 
\leq \delta(\rr) \rr.
\]
In its turn (\ref{EDi}) enables linear approximation of  $\zeta(\theta)$.   
\begin{lemma}[Deviations of empirical process norm] 
\label{ddZeta}
Let condition \ref{EDi} is fulfilled, then 
in the local region $\localr$ with probability $1- e^{-\xx}$ the next statement holds  for all~$\theta, \theta_0 \in \localr$ 
\[
\Vert D^{-1} (\nabla \zeta (\theta) -  \nabla \zeta (\theta_0))  \Vert
\leq 6 \nunu \omega \zz(\xx,p) \rr.
\]
\end{lemma}
Paper \cite{wilks2013} contains proof for this statement. 
If the considered point is MLE ($\theta = \opttheta$) then its concentration in the region $\localr$ follows from condition  (\ref{cond_L}) and Theorem 2.1 from \cite{wilks2013}.
\[\tag{L}\label{cond_L}
\E L(\theta^*) - \E L(\theta_r) =  \Omega (\rr^2),
\quad
\normp{D(\theta^* -\theta_r)} = \rr.
\]

 \subsection{LRT theorem}
  
\label{lrt_main}

Further consider a fixed window position $t$ and window size $2h$. We are going to derive explicit dependence between statistic $\dLh(t)$ and parameter difference from left and  right part of the window $(\thetas_r - \thetas_l)$.  Approximation of $\dLh(t)$ by its quadratic form  splits  noise and deterministic parts, such that  $2 \dLh(t)\approx\normp{ D(\thetas_r - \thetas_l) + \xiv_{lr}}^2$. In the fixed window position the likelihood function  has view 
\[
L(\theta) = L_l(\theta) + L_r(\theta) = L(\theta, \Ybb_l) + L(\theta, \Ybb_r),
\]   
\[
D^2  = - \nabla^2 \E L(\theta^*), 
\quad
D_k^2  = - \nabla^2 \E L_k(\theta_k^*), 
\quad
\xiv_k = D_k^{-1} L_k(\theta_k^*),
\quad
i = \{l,r\}.
\]
Assume that exist a local region where parameter $\theta$ concentrates
\[
\localr = \{ \theta: \| D(\theta - \thetas) \| < \rr \}.
\]
From condition (\ref{cond_A}) for  function $L(\theta)$  it holds with probability $1 - e^{-\xx}$ in local region ($\opttheta, \opttheta_l, \opttheta_r \in \localr$) (ref. Theorem 3.2 in \cite{wilks2013} with $\dLh(t) = L(\theta, \theta_0)$) that
\[
\left |
\sqrt{2 \dLh(t)} - 
\normp{
\begin{matrix}
D_l(\opttheta - \opttheta_l) \\
D_r(\opttheta - \opttheta_r) \\
\end{matrix}}
\right |
\leq 2 \diamondsuit (\sqrt{2}\rr, \xx).
\]
Find relation between $\opttheta, \opttheta_l, \opttheta_r$ using Theorem 3.1 from  \cite{wilks2013} with notation $\xiv_i(\theta) = D_k^{-1}  \nabla L_k(\theta)$
\[
\normp{D (\opttheta - \theta)} \leq 
\normp{
D^{-1}\{ D_l \xiv_l(\theta) + D_r \xiv_r(\theta) \}
} + 2 \diamondsuit (\rr, \xx).
\]
\[
\normp{
\begin{matrix}
D^{-1} D_l \{ \xiv_l(\theta) -  D_l(\theta - \opttheta_l) \}   \\
D^{-1} D_r \{ \xiv_r(\theta) -  D_r(\theta - \opttheta_r) \}   \\
\end{matrix}
}
\leq  2 \diamondsuit (\sqrt{2} \rr, \xx).
\]
Define vector $\widetilde{\theta}$ that is close to $\opttheta$
\[
\widetilde{\theta} = \argmin_{\theta} \left\{ \normp{D_l(\theta - \opttheta_l)}^2 + \normp{D_r(\theta - \opttheta_r)}^2  \right\},
\]
\[
\widetilde{\theta} = (D_l^2 + D_r^2)^{-1} (D_l^2 \opttheta_l + D_r^2 \opttheta_r).
\]
\[
\left |
\normp{
\begin{matrix}
D_l(\opttheta - \opttheta_l) \\
D_r(\opttheta - \opttheta_r) \\
\end{matrix}}  - 
\normp{
\begin{matrix}
D_l(\widetilde{\theta} - \opttheta_l) \\
D_r(\widetilde{\theta}  - \opttheta_r) \\
\end{matrix}}
\right | 
\leq 
\normp{D (\opttheta - \widetilde{\theta})} 
\leq 
 2 \diamondsuit (\rr, \xx) +  2 \diamondsuit (\sqrt{2}\rr, \xx).
\]
\[
\normp{
\begin{matrix}
D_l(\widetilde{\theta} - \opttheta_l) \\
D_r(\widetilde{\theta}  - \opttheta_r) \\
\end{matrix}}  = 
\normp{D_{lr} (\opttheta_r - \opttheta_l)},
\quad D_{lr} =  D_l D^{-1} D_r.
\]
The temporary result is (with probability $1 - 3 e^{-\xx}$) 
\[
\left |
\sqrt{2 \dLh(t)} - 
\normp{D_{lr} (\opttheta_r - \opttheta_l)}
\right |
\leq 4 \diamondsuit (\sqrt{2}\rr, \xx)  +  2 \diamondsuit (\rr, \xx).
\]
Involve $\xiv_l$ and $\xiv_r$  by means of Fisher expansion (equation 3.7 in   \cite{wilks2013}) for the model with two independent components 
\[
\normp{
\begin{matrix}
D_r D^{-1} \{ D_l(\opttheta_l - \theta^*_l ) - \xiv_l \} \\
D_l D^{-1} \{ D_r(\opttheta_r - \theta^*_r ) - \xiv_r \} \\
\end{matrix}} 
\leq
\diamondsuit (\sqrt{2}\rr, \xx).
\]
The final result is Theorem \ref{lrt_th},
which enables to describe  $\dLh(t) $ function depending on change point type and subsequently choose appropriate pattern $P_h(t)$ (ref. Section \ref{sec:procedure}).

\begin{theorem}
\label{lrt_th}
Assume that MLE parameters belong to the local region $\opttheta, \opttheta_l, \opttheta_r \in \localr$ and the likelihood has a fit quadratic expansion (\ref{cond_A}), then with probability $1 - 4e^{-\xx}$ for each $t$ 
\[
\left |
\sqrt{2 \dLh(t)} - 
\normp{D_{lr} (\theta^*_r - \theta^*_l)(t) + \xiv_{lr}(t)}
\right |
\leq 7 \diamondsuit (\sqrt{2}\rr, \xx),
\]
where
\[
\xiv_{lr}(t) =  D_{lr}  \{ D_l^{-2} \nabla L(\thetas_l, \Ybb_l) +  D_r^{-2} \nabla L(\thetas_r, \Ybb_r) \}, 
\quad  
D_{lr} =  D_l D^{-1} D_r.
\]
\end{theorem}

  \subsection{Bootstrap Wilks and Fisher expansions}

\label{boot_WF}

As it was mentioned in Section \ref{sec:procedure} the bootstrap procedure allows to yield likelihood function with two options: each likelihood component is multiplied by weight (weighted bootstrap) or new data is resampled (empirical bootstrap).    
The likelihood function in weighted bootstrap case is a zipped sum with i.i.d weights $(\ub_1,\ldots,\ub_n)$ and independent $\{l_i (\theta)\}^{n}_{i=1}$:
\[
\Lbf = \sum_{i=1}^n \ub_i l_i(\theta),
\]  
\[
\zeta^{\flat}(\theta) = \Lbf - \Lf.
\]
Each weight element has $\Varb \ub_i = 1$ and $\Eb \ub_i = 1$, which is made with a view
\begin{EQA}
\Eb \Lbf & = & \Lf, \\
\Varb \nabla \Lbf  &=&  \sum_{i=1}^n \nabla l_i(\theta) \nabla l_i(\theta)^{T}. 
\end{EQA}
It is expected that $\Varb \nabla \Lbf$ is close to $\Var \nabla \Lf$, which depends on $\E \nabla l_i(\theta)$ values. For example in i.i.d models $\E \nabla l_i(\thetas) = 0$ and $\E \nabla l_i(\theta) $ is close to zero in $\localr$.

Here variable $\xiv$  has a bootstrap duplicate 
\[\tag{xib}\label{xi_boot}
\xib = D^{-1} \nabla \zeta^{\flat}(\opttheta) = \sum_i \xiv_i \epsb_i, 
\quad
\epsb_i = \ub_i - 1.
\] 
Let function $\alphab12$ denotes quadratic approximation error for the weighted likelihood function. Assume further then $\theta, \theta_0 \in \localr$.
\[
\alphab12 = \Lbf - \Lb0 - (\theta - \theta_0)^{T} \nabla \Lb0 + \frac{1}{2} \Vert  D(\theta - \theta_0) \Vert^2.
\]
The mean and deviation of the approximation error are (ref. Theorem 3.2 in \cite{wilks2013})
$$
\Vert D^{-1} \nabla \Eb \alphab12 \Vert =   \Vert D^{-1} \nabla  \alpha(\theta, \theta_0) \Vert 
\leq  \diamondsuit(\rr,\xx),
$$
\begin{equation*}
\label{S}\tag{Sb}
S^{\flat}(\theta, \theta_0) = D^{-1} \{\nabla \alphab12 -  \Eb \nabla \alphab12 \}
= \sum_{i=1}^{n} D^{-1} \{\nabla l_i(\theta) - \nabla l_i(\theta_0) \}  \epsb_i.
\end{equation*}
Define function 
\begin{EQA}
{\overset{o}{S}}_{\flat} &=& S^{\flat}(\theta, \theta_0) - \E_{Y} S^{\flat}(\theta, \theta_0) \\
& = & 
\sum_{i=1}^{n} D^{-1} \{\nabla \zeta_i(\theta) - \nabla \zeta_i(\theta_0) \}  \epsb_i.
\end{EQA}
The function $ {\overset{o}{S}}_{\flat}$ (taking into account Lemma \ref{ddZeta}) with probability $1- e^{-\xx}$ fulfils 
\[
\Vert {\overset{o}{S}}_{\flat} \Vert \leq 
\romb \sqrt{\sum_{i=1}^{n} (\nu_i^2 / \nu_0^2) (\epsb_i)^2 }, 
\quad \sum_i \nu_i^2 / \nu_0^2 = 1.
\] 
Let vector $\epsb = (\epsb_1,\ldots, \epsb_n)^T$ has a restricted exponential moment:
\begin{equation*}
\label{Eu1}\tag{Eu}
\log \Eb e^{\gamma^T \epsb} \leq \frac{\normp{\gamma}^2}{2} ,
\end{equation*} 
then by means of the quadratic form deviation (Lemma \ref{LLbrevelocroB}) with variable $\normp{B^{1/2}\epsb}$ where $B = \diag(\nu_1^2 / \nu_0^2,\ldots,\nu_n^2 / \nu_0^2)$  one get with probability $1 - 2 e^{-\xx}$
\[
\sqrt{\sum_{i=1}^{n} (\nu_i^2 / \nu_0^2) (\epsb_i)^2 } \leq  \sqrt{\tr(B)} + \sqrt{2 \xx} \max_i \frac{\nu_i}{\nunu}  = 1 + O \left( \sqrt{\frac{2\xx}{n}}  \right).
\]
Consequently with high probability 
\[
\Vert {\overset{o}{S}}_{\flat} \Vert \leq  2 \romb. 
\]
In order to find a bound for expectation  $\E_Y S(\theta, \theta_0)$ one may employ the  deviation bound for matrix  sub-Gaussian sums (ref. Lemma \ref{sum_eA}).  Restrict the norms of $\E_Y S(\theta, \theta_0)$ components and expand it by Tailor.
\begin{EQA}
\E_Y  S^{\flat}(\theta, \theta_0) 
&=& \sum_{i=1}^{n} D^{-1} \E_Y \{\nabla l_i(\theta) - \nabla l_i(\theta_0) \}  \epsb_i  \\
&=& \left \{ \sum_{i=1}^{n} D^{-1} \E_Y \{\nabla^2 l_i(\theta_1)  \}  D^{-1}  \epsb_i \right \} D (\theta -  \theta_0)  ,
\end{EQA}
\begin{equation*}\label{Li}
\tag{dDi}
\normop{ D^{-1} \nabla^2 \E l_i (\theta_1)  D^{-1} } \leq C_i(\rr),
\end{equation*}
which leads with probability $1 - 2e^{-\xx}$ to
\[
\Vert \E_Y S^{\flat}(\theta, \theta_0) \Vert \leq \sqrt{2 (\xx + \log p) \sum_i C_i^2(\rr)} \,  \rr.
\]

\begin{theorem}[Weighted bootstrap Wilks]
\label{wilks_boot} 
Under conditions (\ref{EDi}), (\ref{cond_A}), (\ref{Li})  with sub-Gaussian bootstrap weights  (condition \ref{Eu1}) in local region $\localr$ it holds with probability $(1 - 2 e^{-\xx})(1 - 4e^{-\xx^{\flat}})$ ($\xx$~relates to $\Ybb$ generation and $\xx^{\flat}$ relates to bootstrap weights generation)
\[\tag{Ab}
| \alphab12 |
\leq 
 \rombb \Vert D (\theta - \theta_0)  \Vert ,
\]
where
\[
 \rombb = 
\left (2 +   O \left( \sqrt{\frac{2\xx^{\flat}}{n}}  \right)  \right ) \romb +  
\sqrt{2(\xx^{\flat} + \log p) \sum_i C^2_i(\rr) } \, \rr.
\] 
\end{theorem}
A modification of Fisher expansion  (Theorem 2.2 in \cite{wilks2013}) for the weighted likelihood could be proved using the following property
\[
\chib12 =  D^{-1} ( \nabla L^{\flat}(\theta) -   \nabla L^{\flat}(\theta_0) ) + D (\theta -  \theta_0),
\] 
\[
\chib12 = D^{-1} \nabla \alphab12. 
\]
\begin{theorem}[Weighted bootstrap Fisher]
\label{fisher_boot}
Under conditions from Theorem \ref{wilks_boot} it holds with probability $(1 - 2 e^{-\xx})(1 - 4e^{-\xx^{\flat}})$ that
\[
\Vert \chb(\theta^{\flat}, \widehat{\theta}) \Vert
= 
\Vert  D (\theta^{\flat} -  \widehat{\theta}) -  D^{-1} \nabla L^{\flat}(\widehat{\theta}) \Vert
\leq 
  \rombb, 
\] 
where $\theta^{\flat}$, $\widehat{\theta}$ are MLE parameters of the weighted and non-weighted likelihood functions.
\end{theorem}

Theorem \ref{wilks_boot}  enables to prove a theorem similar to Theorem \ref{lrt_th}  for the  bootstrap LRT statistic~$\dLhb$.  The proof steps are the same as in Theorem \ref{lrt_th}.

\begin{theorem}[Weighted bootstrap LRT]
\label{lrt_boot}
Assume that MLE parameters belong to the local region $\theta^{\flat}, \theta_l^{\flat}, \theta_r^{\flat} \in \localr$, the likelihood has a good quadratic expansion (conditions from Theorem \ref{wilks_boot}), then with probability $1 - 8 e^{-\xx} - 16 e^{-\xx^{\flat}} $ 
for each window position $t$
\[
\left |
\sqrt{2 \dLhb(t)} - 
\normp{D_{lr} (\opttheta_r - \opttheta_l)(t) + \xiv_{lr}^{\flat}(t)}
\right |
\leq 7 \diamondsuit^{\flat} (\sqrt{2}\rr, \xx),
\]
where
\[
\xiv_{lr}^{\flat}(t) = D_{lr} \{ D_l^{-2} \nabla L^{\flat}(\opttheta_l, \Ybb_l) +  D_r^{-2} \nabla L^{\flat}(\opttheta_r, \Ybb_r) \}.
\]
\end{theorem}  

\noindent \textbf{In empirical bootstrap} case (\ref{Le}) with sample size $h$ and dataset size $n$ the function corresponding to quadratic approximation error is
\[
S^{\epsilon}(\theta, \theta_0) 
= \sum_{i=1}^{h} \sqrt{\frac{n}{h}} D^{-1} \{\nabla l_{k(i)}(\theta) - \nabla l_{k(i)}(\theta_0) \},
\]
where random indexes $k(i) \in \{1,\ldots,n\}$ and independent. Define 
\[
u_i = (0 , \ldots , \underset{k(i)}{1}, \ldots, 0)^T,  
\]
and 
\[
\nabla l(\theta) = (\nabla l_1(\theta),\ldots, \nabla l_n(\theta)).
\]
Rewrite $S^{\epsilon}(\theta, \theta_0) $ as
\[
S^{\epsilon}(\theta, \theta_0) 
=  \sqrt{\frac{n}{h}}  D^{-1} \{\nabla l(\theta) - \nabla l(\theta_0) \} \sum_{i=1}^{h} (u_i - \E u_i). 
\]
Define function ${\overset{o}{S}}_{\epsilon}$ with $\overset{o}{u} = \sum_{i=1}^{h} (u_i - \E u_i)$
\begin{EQA}
{\overset{o}{S}}_{\epsilon}  &= & S^{\epsilon}(\theta, \theta_0) - \E_Y S^{\epsilon}(\theta, \theta_0) \\
& = & \left [ \sqrt{\frac{n}{h}} D^{-1} ( \nabla^2 \zeta(\theta_1) \overset{o}{u} ) 
D^{-1} \sqrt{ \frac{n}{h} } \right ] \sqrt{ \frac{h}{n} } D (\theta - \theta_0).
\end{EQA}
With $B = \diag(\nu_1^2/\nunu^2 ,\ldots,\nu_n^2 /\nunu^2)$ using Lemma \ref{ddZeta} one get 
\[
\Vert {\overset{o}{S}}_{\epsilon} \Vert \leq 
 \romb \sqrt{ (\overset{o}{u})^T  B \overset{o}{u}  } \leq \romb \,  \| \overset{o}{u} \|_{\infty} ,
\]
where $\| \overset{o}{u} \|_{\infty}$ is the  maximal count of recurrences in bootstrap sample. 
\[
\E e^{\| \overset{o}{u} \|_{\infty}} \leq h \E e^{u_{11}} = h \left( \left( 1 - \frac{1}{n} \right) + \frac{e}{n} \right)^h \leq h (1 + e).
\]
Therefore by means of Chernoff bound with probability $1 - e^{-\xx^{\epsilon}}$
\[
\| {\overset{o}{S}}_{\epsilon} \| \leq (\log(4h) + \xx^{\epsilon}) \romb.
\]
In order to find a bound for expectation  $\E_Y S^{\epsilon}(\theta, \theta_0)$ one may employ the matrix Bernstein bound (ref. Lemma \ref{kolch_matrix}).  Restrict the norms of $\E_Y S^{\epsilon}(\theta, \theta_0)$ components and expand it by Tailor.
\begin{EQA}
\E_Y  S^{\epsilon}(\theta, \theta_0) 
&=& \sum_{i=1}^{h} \sqrt{\frac{n}{h}} D^{-1} \E_Y \{\nabla l_{k(i)}(\theta) - \nabla l_{k(i)}(\theta_0) \}    \\
&=& \left \{ \sum_{i=1}^{h} \frac{n}{h}   D^{-1} \E_Y \{\nabla^2 l_{k(i)}(\theta_1)  \}  D^{-1}    \right \} \sqrt{\frac{h}{n}}  D (\theta -  \theta_0)  ,
\end{EQA}
\begin{equation*}\label{Li}
\tag{dDi}
\normop{  \frac{n}{h} D^{-1} \nabla^2 \E l_i (\theta_1)  D^{-1} } \leq C_i(\rr),
\end{equation*}
which leads with probability $1 - 2e^{-\xx^{\epsilon}}$ to
\[
\Vert \E_Y S^{\flat}(\theta, \theta_0) \Vert \leq   \left( \frac{2}{3} R \xx^{\epsilon}_p +  \vp \sqrt{5 \xx^{\epsilon}_p} \right ) \,  \rr,
\] 
where $R =  \| C(\rr) \|_{\infty}$ and $\vp = \| C(\rr) \|$.
Summarize bounds for $S^{\epsilon}(\theta, \theta_0) $.

\begin{theorem}[Empirical bootstrap Wilks]
\label{wilks_boot_e} 
Under conditions (\ref{EDi}), (\ref{cond_A}), (\ref{Li})  in local region $\localr$ it holds with probability $(1 -  e^{-\xx})(1 - 3e^{-\xx^{\epsilon}})$ ($\xx$~relates to $\Ybb$ generation and $\xx^{\epsilon}$ relates to bootstrap indexes $\{k(i) \}$ generation)
\[\tag{Ae}
| \alpha^{\epsilon}(\theta, \theta_0) |
\leq 
 \rombe \Vert D (\theta - \theta_0)  \Vert ,
\]
where
\[
 \rombe = 
\left (    \log(4h) + \xx^{\epsilon} \right ) \romb +  
\left( \frac{2}{3} \| C(\rr) \|_{\infty} \xx^{\epsilon}_p +  \| C(\rr) \|  \sqrt{5 \xx^{\epsilon}_p} \right )  \, \rr.
\] 
\end{theorem}

\begin{theorem}[Empirical bootstrap LRT]
\label{lrt_boot_e}
Assume that MLE parameters belong to the local region $\theta^{\epsilon}, \theta_l^{\epsilon}, \theta_r^{\epsilon} \in \localr$, the likelihood has a good quadratic expansion (conditions from Theorem \ref{wilks_boot_e}), then with probability $1 - 4 e^{-\xx} - 12 e^{-\xx^{\epsilon}} $ 
for each $t$
\[
\left |
\sqrt{2 \dLh^{\epsilon}(t)} - 
\normp{ \xiv_{lr}^{\epsilon}(t)}
\right |
\leq 7 \diamondsuit^{\epsilon} (\sqrt{2}\rr, \xx),
\]
where
\[
\xiv_{lr}^{\epsilon}(t) = D_{lr} \{ D_l^{-2} \nabla L^{\epsilon}(\opttheta, \Ybb_l) +  D_r^{-2} \nabla L^{\epsilon}(\opttheta, \Ybb_r) \}.
\]
\end{theorem}

\subsection{Bootstrap measure approximation scheme}  
  
Approximation  of  $\P$ measure and variable $f(X)$ ($f$ is non-random)  by corresponded bootstrap (empirical) measure $\Pb$ and variable $f(X^{\flat})$ could be done in three steps: 1) approximate $f(\cdot)$ by smooth function $f_{\triangle}(\cdot)$ and $\Ind$ by smooth indicator; 2) make Gaussian approximation and Gaussian comparison for variables $f_{\triangle}(X)$ and $f_{\triangle}(X^{\flat})$; 3) find anti-concentration bound for $\P(f(\tilde{X}) < x)$.

\begin{figure}[!h]
\begin{center}

\begin{tikzpicture}[>=stealth, thick]

\node (A) at (0,0) [draw, process, text width=12cm, minimum height=0.5cm, align=flush center] 
{
$ |\P \{f(X) < x \} - \Pb \{f(X^{\flat}) <  x \} | \leq \, ???
$ \\
$
X^{\flat} = \sum_i X_i \eps^{\flat}_i
$ 
};

\node (B) at (0,-2) [draw, process, text width=5cm, minimum height=0.5cm, align=flush center] 
{
$X \overset{w}{\longrightarrow} \tilde{X} \in \ND(0,\;\Sigma\;)$ \\
$\quad\quad\quad\quad\quad\quad\quad\;\updownarrow$ \\
$X^{\flat} \overset{w}{\longrightarrow} \tilde{X}^b \in \ND(0,\Sigma^{\flat})$
};

\node (C) at (0,-4) [draw, process, text width=12cm, minimum height=0.5cm, align=flush center] 
{
smoothing: $ \Ind [ f(X) < x - \triangle ] \leq f_{\triangle} (X) \leq \Ind [ f(X) < x + \triangle ] $ \\
anti-concentration: $| \P \{f(\tilde{X}) < x \} - \P \{f(\tilde{X}) <  x + \triangle  \}  | \leq C_A \triangle $
};

\node (D) at (0,-6) [draw, process, text width=12cm, minimum height=0.5cm, align=flush center] 
{
$| \E f_{\triangle}(X) - \E f_{\triangle}(\tilde{X}) | \leq \frac{C_\mu}{\triangle^2} \sum_i \E \| X_i \|^3 $\\
\smallskip
$| \E f_{\triangle}(\tilde{X}) - \Eb f_{\triangle}(\tilde{X}^{\flat}) | \leq  \frac{ C_\Sigma}{ \triangle }\Vert \Sigma - \Sigma^{\flat} \Vert_{\infty} $
};

\draw[->] (A) -- (B);
\draw[->] (B) -- (C);
\draw[->] (C) -- (D);

\end{tikzpicture}
\end{center}
\end{figure}

  \subsection{Anti-concentration}
  \label{ac_sec}

Anti-concentration property can be interpreted as an asymptotic of probability measure depending on event size, tending to zero. Denote by $A_\varepsilon \setminus A$ a region of size $\varepsilon$ around event $A$.    Then anti-concentration is higher when probability of $A_{\varepsilon} \setminus A$ is lower. The next  lemma deals with anti-concentration in one dimensional case where the random variable is maximum of a Gaussian vector.
\begin{lemma}
\label{anti_conc}
Let $\tX \in \ND (m, \Sigma) \in \R^p$,  
\[\sigma_1 \leq  \sqrt{\Sigma_{ii}} \leq  \sigma_2,\] 
\[a_p = \E \max_i (\tX_i - m_i) /\sqrt{ \Sigma_{ii}}, \] 
then $\forall c$
 \[
 \P( \max_{i} \tX_i  \in  [c, c + \varepsilon] )
 \leq \varepsilon C_A (\sigma_1, \sigma_2, a_p),
 \]
 \[
 C_A (\sigma_1, \sigma_2, a_p) =  \frac{4 }{\sigma_1} \left( \frac{\sigma_2}{\sigma_1} a_p +  \left(\frac{\sigma_2}{\sigma_1} - 1\right) \sqrt{2 \log \left(\frac{\sigma_1}{\varepsilon} \right)}+2  - \frac{\sigma_1}{\sigma_2}\right) \approx 4 \frac{\sigma_2}{\sigma_1^2} a_p.
 \]
\end{lemma}
Proof may be found by reference \cite{ChCh2013}.

\begin{remark}
Consider the case when $m = 0$.  If vector $\tX$ has a low parameter $\sigma_1$ then low-variance components of $\tX$ could be dropped from consideration. Find the bound $\sigma$ for such  low-variance components from condition that with high probability
\[
 \max_{i: \, \sigma_i < \sigma} \tX_i  < \max_{i} \tX_i 
\] 
The upper probabilistic bound for  $ \max_{i: \, \sigma_i < \sigma} \tX_i $  is 
$\sigma (\sqrt{2 \log(p)} + 3)$, since $\E \max_{i: \, \sigma_i < \sigma} \tX_i \leq \sigma \sqrt{2 \log(p)}$ and  maximum of a Gaussian vector deviated from its expectation such as one dimensional Gaussian variable:
\[
\P \left( \max_{i: \, \sigma_i < \sigma} \tX_i > \E \max_{i: \, \sigma_i < \sigma} \tX_i + r  \right) \leq e^{- r^2 / \sigma^2}.
\]
\end{remark}

As for extension of this lemma  one can employ it  for a Composite maximum function of type $\max_t Q_t(X)$  in case 
\[
\max_t Q_t(X)  = \max_{\gamma_t , t} \{ \gamma_t^T X  \},
\quad \gamma_t \in \Upsilon.
\]
Apply Lemma \ref{anti_conc} with replacement $X_i = \gamma_t^T X$, $i = (\gamma_t, t)$.
\begin{lemma}
\label{anti_conc_ext}
For parameters 
\[
 \sigma_1 \leq  \sqrt{ \gamma_t^T \Sigma \gamma_t } \leq  \sigma_2, 
\quad \gamma_t \in \Upsilon,
\]
\[
a_p = \E \max_{t, \gamma_t} \frac{\gamma_t^T (\tX - m)} {\sqrt{ \gamma_t^T \Sigma \gamma_t }} \leq \frac{ \E \max_t Q_t( \tX - m) }{\sigma_1} ,
\]
it holds that 
 \[
 \P( \max_t Q_t( \tX)    \in  [c, c + \varepsilon] )
 \leq  \varepsilon C_A (\sigma_1, \sigma_2, a_p).
 \]
\end{lemma}

  \subsection{Gaussian approximation}
  \label{gar}
    
The goal of  Gaussian approximation is to measure the difference between a distribution of independent random vectors sum and distribution of the closest Gaussian vector.  It makes sequential replacement of random arguments in function  $\E f(X_1 + \ldots + X_n)$ by its normal duplicates with the same moments:
  \[
  \E f(X_1 + \ldots + X_n), \E f(\tX_1 + \ldots + X_n), \ldots, \E f(\tX_1 + \ldots + \tX_n),
  \] 
  \[
  \E X_i  = m_i, 
  \quad
  \Var X_i = S_i,
  \quad
  \tX_i \in \ND(m_i, S_i).
  \]
  Each replacement yields difference element $\text{err}_i$   according to Tailor expansion with $t \in [0,1]$
\begin{EQA}
&& \E f(X + X_i) - \E f(X + \tX_i) \\
& \quad & = \frac{1}{6} \E \tr \left\{ \nabla^3 f(X + t X_i) \,   X_i \otimes X_i \otimes X_i \right\} - \frac{1}{6} \E \tr \left\{ \nabla^3 f(X + t \tX_i) \,  \tX_i \otimes \tX_i \otimes \tX_i \right\}. 
\end{EQA}
The first and the second elements of the expansion are substituted. Use property for a tensor $T$ and any vectors $a$, $b$, $c$ 
\[
\tr\{T \, a \otimes b \otimes c \} = \sum_{i,j,k} T_{ijk} a_ib_jc_k \leq 
\| T \|_1 \|a\|_{\infty} \|b\|_{\infty} \|c\|_{\infty}. 
\]
Then
  \[\label{err_gar_i}
  \text{err}_i \leq  \frac{1}{6}  \E \normp{\nabla^3  f(X + t X_i)   }_1 \normp{X_i}_{\infty}^3  + \frac{1}{6}  \E \| \nabla^3  f(X + t \tX_i)   \|_1 \| \tX_i \|_{\infty}^3. 
  \] 
Consider the case when function $f$ is a composition of a smooth indicator $g_\triangle$ and some other differentiable function $h$ (for example smooth max). The third gradient of the composition is
 \[
\nabla^3 (\gtr \circ h) = \gtr''' \nabla h \otimes \nabla h \otimes \nabla h  + 2 \gtr'' \nabla^2 h \otimes \nabla h +  \gtr'' \nabla h \otimes \nabla^2 h + \gtr' \nabla^3 h,  
 \]
 \[
\normp{ \nabla^3 (\gtr \circ h) }_1 \leq   | \gtr''' |  \normp{ \nabla h }_1^3  + 3 | \gtr'' | \normp{ \nabla^2 h }_1 \normp{ \nabla h }_1  + | \gtr' | \normp{ \nabla^3 h }_1.  
 \]
Assume that  $g_\triangle$ grows from 0 to 1 in interval $[z, z + \triangle]$. So $g'_\triangle = 0 $ outsize $[z, z + \triangle]$. Furthermore in this case 
\[
\nabla^3 f =  \nabla^3 (\gtr \circ h)   = \nabla^3 (\gtr \circ h)  \Ind [ z \leq h(x) \leq  z + \triangle  ].
\]
 Move to approximation of a distribution function.  The  aim is to find upper bound for 
\[
 E_G = \left | \P \left ( H \left( \sum_i X_i \right ) \leq  z \right ) - \P \left ( H \left( \sum_i \tX_i \right ) \leq  z \right )  \right |. 
\]
Let $h$ be a smooth approximation for function $H$ such that 
 \[
  H(x) - \triangle \leq h(x) \leq H(x) + \triangle,
 \]
 \[\label{sm_err}\tag{SmAprx}
\Ind [ H(x)  \geq z + \triangle ] \leq  \gtr (h(x)) \leq \Ind [ H(x)  \geq z - \triangle ].
 \]  
Since $\E \Ind[H(X) \leq  z] = \P(H(X) \leq  z)$, one get with $X = \sum_i X_i$
\begin{EQA}
&& \P (H(X) \leq  z - \triangle )   \\
& \quad &  \leq  \E f(X) \\
& \quad & \leq \E f(\tX) + \sum_i \text{err}_i \\ 
 & \quad & \leq  \P (H(\tX) \leq  z + \triangle ) + \sum_i \text{err}_i.
\end{EQA}
Subsequently
 \[
\left | \P ( H(X) \leq  z) - \P ( H(\tX) \leq  z \pm 2 \triangle) \right | \leq \sum_i \text{err}_i.
 \]
Suppose an anti-concentration property for random variable $H(\tX)$ (ref. Section \ref{ac_sec}):
\[
 \P ( H(\tX) \in  [z,  z + \triangle]) \leq C_A \triangle.
\]
Then 
\[
E_G \leq 2 C_A \triangle +  \sum_i \text{err}_i
\]
and 
\begin{EQA}
&&  E_X \left(  \Ind [ z \leq h(X + tX_i) \leq  z + \triangle  ] \right )  \\
&\quad & \leq E_{X} \left(  \Ind [ z \leq H(X + tX_i) \leq  z + 2 \triangle  ] \right)    \\
&\quad & \leq  E_{\tX} \left(  \Ind [ z \leq H(\tX + tX_i) \leq  z + 2 \triangle  ]  \right) + E_G \\
&\quad &   \leq  C_A  (2\triangle + \|  X_i \|_{\infty}) + E_G. 
\end{EQA}
Demote the sum of the third moments 
 \[
 \mu_X^3 = \sum_i \left( \E \normp{X_i}_{\infty}^3 + \E \normp{\tX_i}_{\infty}^3 \right).
 \]
 Assume restriction for the third derivatives and define  constant $C_{\mu}$:
 \[
\frac{1}{6} \normp{ \nabla^3 (\gtr \circ h) }_1 \leq \frac{1}{\triangle^3} C_{\mu}. 
 \]
Then
 \[
\sum_i \text{err}_i  \leq \frac{1}{\triangle^3} C_{\mu} (  2 \triangle C_A + E_G  ) \,   \mu_X^3 + O \left (\sum_i  \E \normp{X_i}_{\infty}^4 \right ) .
 \]
 Furthermore, neglecting the sum of the fourth moments
 \[
 E_G  \lesssim  \frac{1}{\triangle^3}  C_{\mu} ( 2  \triangle C_A + E_G ) \mu_X^3 + 2\triangle C_A,
 \]
 \[\label{gar_err}\tag{ErrG}
 \left | \P ( H(X) \leq  z) - \P ( H(\tX) \leq  z ) \right |  \leq  5 C_{\mu}^{1/3} C_A  \mu_X.
 \]  
 
Consider linear forms ($A X$) with sparse-row matrices and random vector  $X = (X_1,\ldots,X_n)$ with independent elements (sub-vectors). 
  In order to make  approximation by Gaussian vector $A \tX$ one have  to group   ($X_1,\ldots,X_n$) such that elements in one group has no common non-zero coefficients in each row of matrix $A$. This allows following representation 
\[
A X = Z = Z_1 + \ldots + Z_h, 
\quad Z_i =  A F_i X, 
\] 
vectors $\{Z_i\}$ are independent and $F_i$ is a filter for the $i$-th group  that sets matrix columns to $0$ related to the  other groups. 
In case each row of the matrix $A$ has no more than $h$ non-zero elements then minimal groups count equals to $h$.  The next statement confirms it.  
\begin{lemma}
Let  each element in a set of subsets $\{ \mathcal{ M }_s \}$ has size $h$. Then  subsets $\{ \mathcal{ Z }_1, \ldots, \mathcal{ Z }_h \}$ exist with  properties 
\[
\bigcup_s  \mathcal{ Z }_s = \bigcup_s  \mathcal{ M }_s,
\quad 
\mathcal{ Z }_k \bigcap \mathcal{ M }_s = 1,
\quad 
 \mathcal{ Z }_k \bigcap \mathcal{ Z }_s = 0.
\] 
\end{lemma}

\begin{proof} Build  subsets $\{ \mathcal{ Z }_1, \ldots, \mathcal{ Z }_h \}$  constructively. 
Take one element from $\bigcup_s  \mathcal{ M }_s$. Exclude this element from all $\{ \mathcal{ M }_s \}$ and add it to $ \mathcal{ Z }_1$. Mark subsets which contain this element as  $\{ \mathcal{ M' }_s \}$. Take  another element  from $\bigcup_s  \mathcal{ M }_s  \setminus  \bigcup_s \mathcal{ M' }_s $ and add it to $ \mathcal{ Z }_1$. Repeat this procedure until $ | \bigcup_s  \mathcal{ M }_s  \setminus  \bigcup_s \mathcal{ M' }_s |  = 0 $. Then do the same steps for $\bigcup_s  \mathcal{ M }_s \setminus \mathcal{ Z }_1$ and obtain $\mathcal{ Z }_2$ with the required properties by construction. 
\end{proof}
Summarise results of this Section. Apply equation (\ref{gar_err}) to $Z_i$ instead of $X_i$. 

\begin{lemma}
\label{gar_lemma}
 Let matrix $A$ has at most $h$ non-zero elements in each row and non-zero elements correspond to independent elements in vector $X = (X_1,\ldots,X_n)$. Then Gaussian approximation with vector $\tX$ has following upper bound 

 \[\label{gar_err_A}\tag{ErrGA}
 \left | \P \left(H(A X) \leq  z \right) - \P \left(H (A \tX) \leq  z \right) \right |  \leq  5 C_{\mu}^{1/3} C_A  \mu_Z,
 \]  
 where 
\[
\mu_Z^{3} = \E \normp{X}^3_{\infty} \sum_{i = 1}^h  \vertiii{ A F_i }^3_1 \leq    \E \normp{X}^3_{\infty} h \| A \|^3_{\infty},
\]
and $C_\mu$, $C_A$ correspond to 
  \[
\frac{1}{6} \normp{ \nabla^3 (\gtr \circ h (AX)) }_1 \leq \frac{1}{\triangle^3} C_{\mu},
 \]
 \[
 \P ( H(A\tX) \in  [z,  z + \triangle]) \leq C_A \triangle.
\]
\end{lemma}

   \subsection{Gaussian comparison}
   \label{gcmp}

This Section deals with extension of the Gaussian approximation ( ref. Section\ref{err_gar_i}) for the case when the second moments of $\Var X$ and $\Var \tX$ are slightly different. Let as previously $X = \sum_{i = 1}^n X_i$ and 
\[
\normp{\Var X_i - \Var \tX_i}_{\infty} \sim \frac{1}{n}.
\]	  
Then after sequential replacement $X_i \to \tX_i$ the approximation  bound term will be
\[
\sum_{i=1}^n \text{err}_i  \sim  1,
\]  
while 
\[
\normp{\Var X - \Var \tX}_{\infty} \sim \frac{1}{\sqrt{n}}.
\]
The next lemma resolves this problem.  At first one is able to make Gaussian approximation with equal variances ($\Var X = \Var \tX$), and then compare two Gaussian vectors with different variances.  
 
\begin{lemma}
\label{LGauscomp}
Let \( X \) and \( Y \) be two zero mean Gaussian vectors  with 
\( \Sigma_{X} = \Var(X) \) and \( \Sigma_{Y} = \Var(Y) \).
Let also \( \fm(X) \) be a smooth function.
Then 
\[\tag{EErrGC}
	\reps
	 =
	\bigl| \E \fm(X) - \E \fm(Y) \bigr| 
	\leq 
	\frac{1}{2} \, \|\Sigma_{X} - \Sigma_{Y} \|_{\infty} \,  \| \E \nabla^{2} \fm \|_{1} \, ,
\label{mXvYvdef}
\]
where \(  \| \nabla^{2} \fm \|_{1} = \sup_{X} \| \nabla^{2} \fm(X) \|_{1} \).

\end{lemma}

\begin{proof}
Without loss of generality assume that \( X \) and \( Y \) are given on 
the same probability space and independent.
For each \( t \in [0,1] \), define 
\begin{EQA}
	Z(t)
	& = &
	\sqrt{t} \, X + \sqrt{1-t} \, Y,
	\\
	\Psi(t)
	& = &
	\E \fm(Z(t)) 
	=
	\E \fm\bigl( \sqrt{t} \, X + \sqrt{1-t} \, Y \bigr).
\label{ZvtPsimZvt}
\end{EQA}
\begin{EQA}
	\reps
	& = &
	|\Psi(1) - \Psi(0)|
	=
	\biggl| \int_{0}^{1} \Psi'(t) dt \biggr| .
\label{epsintPsid01}
\end{EQA}
\begin{EQA}
	\Psi'(t)
	& = &
	\E [\nabla \fm(Z(t))^{\T} Z'(t)]
	=
	\frac{1}{2} \E \bigl[ 
		\bigl\{ t^{-1/2} X - (1-t)^{-1/2} Y \bigr\}^{\T} \, \nabla \fm(Z(t)) 
	\bigr] .
\label{Psidt}
\end{EQA}
To compute this expectation, we apply the Stein identity.
Let \( W \) be a zero mean Gaussian vector. 
Then for any \( C^{1} \) vector function \( s \) it holds
\begin{EQA}
	\E [W \, s(W)^{\T}]
	& = &
	\Var(W) \, \E [\nabla s(W)] .
\label{ESteinW}
\end{EQA}
\begin{EQA}
	\E \bigl[ \nabla \fm(Z(t)) X^{\T} \bigr]
	& = &
	t^{1/2} \Sigma_{X} \E \bigl[ \nabla^{2} \fm\bigl( Z(t) \bigr) \bigr] 
	\\
	\E \bigl[ \nabla \fm(Z(t)) Y^{\T} \bigr]
	& = &
	(1-t)^{1/2} \Sigma_{Y} \E \bigl[ \nabla^{2} \fm\bigl( Z(t) \bigr) \bigr],
\label{EmZvWvT}
\end{EQA}
\begin{EQA}
	\bigl| \Psi'(t) \bigr|
	& \leq &
	\frac{1}{2} \Bigl| 
		\tr\bigl\{ \bigl( \Sigma_{X} - \Sigma_{Y} \bigr) 
		\E \bigl[ \nabla^{2} \fm\bigl( Z(t) \bigr) \bigr]  \bigr\}
	\Bigr|
	\\
	& \leq &
	\frac{1}{2} \, \|\Sigma_{X} - \Sigma_{Y} \|_{\infty} \, 
	\bigl\| \E \bigl[ \nabla^{2} \fm\bigl( Z(t) \bigr) \bigr] \bigr\|_{1}  
	\leq 
	\frac{1}{2} \, \|\Sigma_{X} - \Sigma_{Y} \|_{\infty} \, \E \| \nabla^{2} \fm \|_{1} \, .
\label{Psiptup}
\end{EQA}
\end{proof}

Find an upper bound for distribution difference
 \[
 E_{GC} = \left | \P ( H(X) \leq  z) - \P ( H(Y)   \leq  z ) \right | 
\]
Identically to Section \ref{err_gar_i}  consider a composition of a smooth indicator $g_\triangle$ and a differentiable smooth approximation $h$ for function $H$. The approximation satisfies condition (\ref{sm_err}).
Assume that  $g_\triangle$ grows from 0 to 1 in interval $[z, z + \triangle]$. So $g'_\triangle = 0 $ outsize $[z, z + \triangle]$. 
Assume following  restriction for the second derivative of the function $f = g_\triangle \circ h $ and define $C_{\Sigma}$:
 \[
\frac{1}{2} \| \nabla^2 g_\triangle \circ h \|_{1} \leq \frac{1}{\triangle^2} C_{\Sigma}.
 \]
Then by means of Lemma \ref{LGauscomp} and taking into account $\P = \E \Ind $
 \begin{EQA}
 && \left | \P ( H(X) \leq  z) - \P ( H(Y) \leq  z \pm 2 \triangle) \right | \\
&\quad & \leq  \bigl|  \E \fm(X) - \E \fm(Y) \bigr|  \\  
&\quad & \leq  \| \nabla^2 g_\triangle \circ h \|_{1}   \E \Ind \bigg [ h(X) \in [z, z + \triangle] \bigg ] \|\Sigma_{X} - \Sigma_{Y} \|_{\infty}, \\
 &\quad & \leq  \frac{1}{\triangle^2}  C_{\Sigma}  \E \Ind \bigg [ H(X) \in [z, z + 2\triangle] \bigg ] \|\Sigma_{X} - \Sigma_{Y} \|_{\infty}. 
 \end{EQA}
 Following the logic from Section \ref{err_gar_i} the anti-concentration (ref. Section \ref{ac_sec}) property  allows $ 2 \triangle$-shift elimination and provides an upper bound for $\E \Ind  [ H(X) \in [z, z + 2\triangle]  ]$.
 \[
  \E \Ind [ H(X) \in z \pm  2\triangle  ] \leq  2 \triangle C_A ,
 \]
 \[
  E_{GC}  \leq   \frac{2}{\triangle^2}  C_{\Sigma}  \triangle C_A \|\Sigma_{X} - \Sigma_{Y} \|_{\infty} + 2\triangle C_A,
 \]
 Optimize over $\triangle$ values
 \[
 E_{GC}  \leq   4 C^{1/2}_{\Sigma} C_{A}  \|\Sigma_{X} - \Sigma_{Y} \|^{1/2}_{\infty}.
 \]
Summarise the discussion related to $ E_{GC} $ bound.
\begin{lemma} 
\label{gc_lemma}
Let function  $H(X)$ with Gaussian argument satisfies anti-concentration with $C_A$ constant, \( X \) and \( Y \) be two zero mean Gaussian vectors  with 
\( \Sigma_{X} = \Var(X) \) and \( \Sigma_{Y} = \Var(Y) \). Then under condition  (\ref{sm_err})   
\[\label{gc_err}\tag{ErrGC}
 \left | \P ( H(X) \leq  z) - \P ( H(Y) \leq  z) \right | \leq  4 C^{1/2}_{\Sigma} C_{A}  \|\Sigma_{X} - \Sigma_{Y} \|^{1/2}_{\infty}.
 \]
 \end{lemma}
Consider a random vector $X = (X_1,\ldots,X_n)$  with independent sub-vectors $\{X_i\}$ and matrix $A$ which consists of blocks of size $[p \times p]$, $p = \dim(X_i)$. Define also a bootstrap vector with independent random weights $\{ \epsb_i \}$, such that  $X^{\flat} = (\epsb_1 X_1,\ldots, \epsb_n X_n)$.
\[
A = \left(\begin{matrix}
a_{11} I_{p} & \ldots &  a_{1n} I_{p} \\
 & \ldots &  \\
a_{n1} I_{p} & \ldots &  a_{nn} I_{p}
\end{matrix} \right).
\]
The next comparison of   $H (A X)$ and $H (A \bX)$  my means of (\ref{gc_err})  require upper bound for the maximal element of the covariance matrices difference. 
\begin{EQA}
\| \Sigma_{AX} - \Sigma^{\flat}_{AX} \|_{\infty} &=&   \max_{a^T, b^T \in \text{rows} A}  |  a^T  ( \Sigma  -  \Sigma^{\flat}) b | .  
\end{EQA}
Let for a fixed rows $a$, $b$  with probability $1 - e^{-\xx}$  
\[
|  a^T  ( \Sigma  -  \Sigma^{\flat}) b |  =  \left |  \sum_{ij} a_i b_j \Sigma_{ij}  - \sum_{ij} a_i b_j X_i X_j  \right |  \leq \square(\xx).
\]
Note that elements in sum ($a_i b_i X_i X_j $) are independent due to the specific block structure of matrix $A$. Then the joint bound with probability $1 - e^{-\xx}$ is 
\[
\| \Sigma_{AX} - \Sigma^{\flat}_{AX} \|_{\infty}  \leq  \square( \xx + 2 \log(np)).
\]
Involve the upper bound for covariance matrix deviations (\ref{var_diff_err}) with $\varepsilon_i =  X_i $ and $\UV_i = a_i / V$
\begin{EQA}
\square(\xx) &=& V^2 \left( \frac{2}{3} R_{\varepsilon \varepsilon} \xx + 2 \vp_{\varepsilon \varepsilon} \sqrt{5 \xx} + \td^2 \normp{b}^2 \right), 
\end{EQA}
where  $\vp_{\varepsilon \varepsilon} = \delta     \sqrt{ \| \Sigma \|_{\infty}} ( 3 + \| b \|) $, 
\[
V  \td =  \| A \|_{\infty},
\quad 
V^2 = \sum_{ij}  a_i b_j \Sigma_{ij} \leq \vertiii{ A }^2 \| \Sigma \|_{\infty}, 
\quad
\| b \|^2 = \sum_{i: \, a_i > 0} \E X_i^2. 
\]
Finally under assumption $\frac{2}{3} R_{\varepsilon \varepsilon} \xx <  \vp_{\varepsilon \varepsilon} \sqrt{5 \xx} $ with probability $1 - 1/n$
\[
\label{dSA}\tag{dSA}
\| \Sigma_{AX} - \Sigma^{\flat}_{AX} \|_{\infty}  \leq 10 \sqrt{\log(np)} \vertiii{ A } \| A \|_{\infty}  \| \Sigma \|_{\infty}  ( 3 + \| b \|) + \| A \|_{\infty}^2 \|b \|^2.
\]

   \subsection{Smooth-max properties}
   \label{sm_max_sec}
  
 Denote by $h(x)$ a smooth maximum function which converges to  $\max_i x_i$ when $\beta \to \infty$.
  
\[
h(x) = \beta^{-1} \log u(x),
\quad
u(x) = \sum_i e^{\beta x_i}, 
\quad 
x = (x_1, x_2, \ldots).
\]
Characterize its derivatives which will be used further in Tailor expansions. 
\begin{lemma} All derivatives of $h(x)$ of order $m = \{1,2,3\}$ have following upper bounds $\forall x$
\label{sm_max_grad}
\[
\| \nabla^{(m)} h(x) \|_1  \leq \beta^{m-1}. 
\]
\end{lemma}

\begin{proof}
\begin{EQA}
\nabla h(x) &=&   \beta^{-1}  \frac{\nabla u}{u}, \\
\nabla^2 h(x) &=&   \beta^{-1}  \left( 
\frac{\nabla \otimes \nabla u}{u}  -  \frac{\nabla u \otimes \nabla u}{u^2} 
\right),\\
\nabla^3 h(x) &=&   \beta^{-1} \left(
\frac{\nabla \otimes \nabla \otimes \nabla u}{u}
 - \frac{\nabla \otimes \nabla u \otimes \nabla u}{u^2}
 - \frac{\nabla \otimes (\nabla u \otimes \nabla u)}{u^2} 
 + 2 \frac{\nabla u \otimes \nabla u \otimes \nabla u}{u^3} 
\right).
\end{EQA}
Define $p_i = \frac{\nabla u}{u}(i)$ that satisfies condition $\sum_i p_i = 1$. The first tensor norm equals to the convolution maximum with vectors $\alpha$, $\phi$, $\gamma$ under restriction $\normp{\alpha}_{\infty} = 1, $ $ \normp{\phi}_{\infty} = 1, $ $ \normp{\gamma}_{\infty} = 1$.
\begin{EQA}
\alpha^T \nabla^2 h(x) \gamma  &=& \beta \left(
\sum p_i \alpha_i \gamma_i  - \sum p_i \alpha_i \sum p_j \gamma_j 
\right) \\
&=& \beta \left( \E \alpha \gamma - \E \alpha \E \beta \gamma \right) = \beta \E \overset{o}{\alpha}  \overset{o}{\gamma} \leq  \beta \| \alpha \|_{\infty} \| \gamma \|_{\infty}, \\
 \sum_{ijk}  \nabla^3_{ijk} h(x) \alpha_i \phi_j \gamma_k  &=& 
 \beta^2 \left(
\E \alpha  \phi \gamma 
- \E \alpha  \E \phi \gamma
- \E \alpha  \phi \E \gamma
- \E \alpha \gamma \E \phi
+ 2 \E \alpha \E  \phi \E \gamma 
\right)\\
&=&
\beta^2 \left( \E \overset{o}{\alpha} \overset{o}{\phi} \overset{o}{\gamma}
\right) \leq \beta^2  \| \alpha \|_{\infty} \| \phi \|_{\infty} \| \gamma \|_{\infty}.
\end{EQA}
Taking maximum  provides  the required restriction for L1 tensor norms.

\end{proof}
The next property of $h(x)$ with $x \in \R^p$ characterise the error of smooth maximum approximation
\[
 \max_i(x_i)  \leq  h(x) \leq \max_i(x_i)  +  \beta^{-1} \log(p).
\] 
 Applying indicator for  both  parts of this inequality yields statement similar to \ref{sm_err} which is used in Gaussian approximation (ref. Section \ref{gar}) and Gaussian comparison (ref. Section \ref{gcmp}). 
\begin{lemma}
\label{max_approx}
For a smooth indicator function $\gtr$ ($\gtr$ grows from 0 to 1 inside interval $[z, z + \triangle]$)  it holds  with $\triangle = \beta^{-1} \log(p)$ that
\[
\Ind \left[\max_{0 \leq i \leq p} x_i > z + \triangle \right]  \leq  \gtr h (x) \leq \Ind \left[\max_{0 \leq i \leq p} x_i > z - \triangle \right] .
\]
\end{lemma}

 This result allows to use inequalities (\ref{gar_err}) and  (\ref{gc_err}) with following restrictions for the derivatives. Let $g_{\triangle} (x) =  g(x / \triangle)$ and $H(x) = \max(x)$ then
 \[
C_{\mu} =  \frac{\max (g', g'', g''')}{6} \left( 1 + 3 \log p + \log^2 p  \right),
 \]
 \[
 C_{\Sigma} = \frac{\max (g', g'')}{2} \left( 1 + \log p \right).
 \]

  \subsection{Composite smooth-max properties}
  \label{complex_max} 
  
Consider a composite maximum function with its smooth approximation $h(q(X))$
\[
\max_{1 \leq  t \leq T} Q_t(x)  \approx h(q(x)), 
\quad q(x) = (q_1(x), \ldots, q_T(x)).
\]
where $h$ in the smooth approximation for maximum from Section \ref{sm_max_sec} and  $q_t$ is a smooth approximation for $Q_t$ with property: 
\[
\forall t: \, Q_t(x) \leq q_t(x) \leq  Q_t(x) + \beta^{-1}_q \log(p).
\]
Combination with maximum approximation property $h(x) \leq \max_t(x_t)  +  \beta^{-1} \log(p)$ leads to statement 
\[
\max_t(Q_t(x))  \leq  h(q(x))  \leq  \max_t(Q_t(x))  +  (\beta_q^{-1} + \beta^{-1}) \log(p).
\]
This allows to extend Lemma \ref{max_approx}.
\begin{lemma}
\label{cmax_approx}
For a smooth indicator function $\gtr$ ($\gtr$ grows from 0 to 1 inside interval $[z, z + \triangle]$)  it holds  with $\triangle = (\beta_q^{-1} + \beta^{-1} ) \log(p)$  for all $x$ and $z$ that
\[
\Ind \left[\max_{1 \leq  t \leq T} Q_t(x) > z + \triangle \right]  \leq  \gtr( h( q (x)) \leq \Ind \left[\max_{1 \leq  t \leq T} Q_t(x) > z - \triangle \right] .
\]
\end{lemma}

Assume also restriction for derivatives of the functions $q_t$: 
\[
\forall t: \, \| \nabla^{(m)} q_t(X) \|_1 \leq   \frac{C_q^{m-1}}{\triangle^{m-1}}, 
\]
from which follows 
\[
\| \nabla h (q(X)) \|_1 \leq \| \nabla h  \|_1 \| \nabla q_t \|_1 \leq 1,
\]
\[
\| \nabla^2 h (q(X)) \|_1 \leq \| \nabla^2 h  \|_1 \| \nabla q_t \|^2_1 + \| \nabla h  \|_1 \| \nabla^2 q_t \|_1 \leq  \frac{\log T}{\triangle} + \frac{C_q}{\triangle},
\]
\begin{EQA}
\| \nabla^3 h (q(X)) \|_1 & \leq &  \| \nabla^3 h  \|_1 \| \nabla q_t \|^3_1 + 3 \| \nabla^2 h  \|_1 \| \nabla^2 q_t \|_1  \| \nabla q_t \|_1 +  \| \nabla h  \|_1 \| \nabla^3 q_t \|_1    \\
 & \leq & \frac{1}{\triangle^2} ( \log^2 T + 3 C_q  \log T + C_q^2 ).  
\end{EQA}
So one can override the constants used in  (\ref{gar_err}) and (\ref{gc_err}) for the case $H(X) = \max_{1 \leq  t \leq T} Q_t(X)$:
\[
C_{\mu} =  \frac{\max (g', g'', g''')}{6} \left( 1 + 3 \log T +  3 C_q  + \log^2 T + 3 C_q  \log T + C_q^2  \right),
 \]
 \[
 C_{\Sigma} = \frac{\max (g', g'')}{2} \left( 1 + \log T + C_q  \right),
 \]
 where $ g(x /\triangle) = g_{\triangle}(x) $.

   \subsection{LRT bootstrap precision}
  
\label{lrt_boot_prec}

Consider statistic $\dLhconv $ and the corresponded bootstrap analogy $\dLhconvb$ defined in Section~\ref{sec:procedure}. One could also repeat the further logic for empirical bootstrap $\dLhconve$. Describe the bootstrap approximation for the quadratic form of the statistic $\dLhconv $ on the grounds of Theorems~\ref{lrt_th} and~\ref{lrt_boot}. The quadratic form of  $\dLhconv $  is 
\[
\tag{TPQ}
\label{conv_stat}
 \max_{1 \leq \tau \leq n}  \left \{ \sum_{t} P_{\tau}(t)  \Vert  \xiv_{lr} (t) \Vert \right \}.
\]
The corresponded bootstrap quadratic form is
\[
 \max_{1 \leq \tau \leq n}  \left \{ \sum_{t} P_{\tau}(t)  \Vert  \xiv^{\flat}_{lr} (t) \Vert \right \}.
\]
Our aim is to show that these two forms are close by distribution.  
For simplification  assume that for all window positions the true model parameter is fixed $\theta_l^* = \theta_r^*  = \theta^*$. Then $\xiv_{lr} (t)$ doesn't pay attention on parameter changes and 
\[
\xiv_{lr} (t) = \sum_{i = t}^{t + h} \xiv_i - \sum_{i = t + h}^{t + 2h} \xiv_i,
\] 
where $\xiv_i = D^{-1} \nabla l_i(\thetas)$ and $D^{-1} =  D_{lr}  D_l^{-2} =  D_{lr}  D_r^{-2}$. Use smooth-max approximation for the composite maximum function (ref. Section \ref{complex_max})  with argument $ \xiv_{lr} = Z = A \xiv$, where  $\xiv= (\xiv_1,\ldots, \xiv_n)$. 
\[
\label{conv_max_stat}
\frac{1}{\sqrt{p}} \max_{\tau} \left \{ \sum_{t }   P_\tau (t)  \| Z(t) \|  \right \} \approx   h(q(Z)),
\]
where 
\[
q_{\tau}(Z) = \sum_{t}   P_\tau (t) \frac{\| W_t Z \| } {\sqrt{p}}, 
\quad 
W_t = \diag (0, \ldots,  1_{tp}, \ldots, 1_{tp + p}, \ldots,  0).
\]
By means of Lemmas  \ref{gar_lemma}, \ref{gc_lemma} and \ref{anti_conc_ext} we are to prove the next Theorem.  

\begin{theorem}[Bootstrap Approx] 
\label{boot_approx}
Let window size be equal to $h$, the dataset size -- $n$,  the model dimension -- $p$, pattern functions $P_{\tau}(t)$ be independent from $\| \xiv_{lr}(t) \|$ and normalized $\sum_t | P_\tau (t) | = 1$. Then the  distribution difference  between $\dLhconv$ and   $\dLhconvb$ is
\begin{EQA}
&& \left | \P \left( \max_{1 \leq \tau \leq n} \sum_{t} P_\tau (t) \|  \xiv_{lr}(t) \|  > z \right) - \Pb \left(\max_{1 \leq \tau \leq n} \sum_{t} P_\tau (t) \|  \xiv^{\flat}_{lr}(t) \|  > z \right) 
\right |  \\
&\quad & \leq 5 C_{\mu}^{1/3} C_A  \mu_Z +  4 C^{1/2}_{\Sigma} C_{A}  \|\Sigma_{A\xi} - \Sigma_{A \xi}^{\flat} \|^{1/2}_{\infty},
\end{EQA}
where the constants $C_{\mu}$ and $C_{\Sigma}$ correspond to  composite maximum function (ref. Section \ref{complex_max}) and 
\[
\|\Sigma_{A\xi} - \Sigma_{A \xi}^{\flat} \|_{\infty}  \leq 10 \sqrt{\log(np)} \sqrt{2h}   \| \Sigma_{\xi} \|_{\infty}  ( 3 + \| b \|) +  \|b \|^2 ,  
\]
\[
 \| b \|^2 = \max_{t} \sum_{i = t}^{t + 2h} \| \E \xiv_i \|^2_{\infty}	, 
\]
\[
\mu_Z^{3} \leq \E \normp{\xiv}^3_{\infty} h \| A \|^3_{\infty} = \E \normp{\xiv}^3_{\infty} h.
\]
\end{theorem}

\begin{proof}

One should estimate distribution difference from replacement of the random argument in statistic (\ref{conv_stat}):  $\xiv \to \txiv \to \bxiv$. Note that $\txiv \in \ND(\E \xiv, \Var(\xiv))$ and $\bxiv \in \ND(0, \diag(\xiv_i \xiv_i^T))$. Taking into account $\sum_t | P_\tau(t)  | = 1$ estimate $q$'s derivatives required for (\ref{gar_err}) and (\ref{gc_err}).
\[
\| \nabla q_{\tau}(Z) \|_1  = \sum_{t }  | P_\tau (t) | \frac{ \| W_t Z \|_1}{\sqrt{p} \| W_t Z \|} \leq 1,
\]
\[
\| \nabla^2 q_{\tau}(Z) \|_1  = \sum_{t}   | P_\tau (t) | \frac{ \| W_t - Z W_t Z^T / \| W_t Z \|^2   \|_1}{\sqrt{p} \| W_t Z \|} \leq  \frac{2 \sqrt{p}}{  \| W_t Z \| },
\]
\[
\| \nabla^3 q_{\tau}(Z) \|_1  \leq \sum_{t }  3  | P_\tau (t) |  \frac{ \| W_t \otimes  W_t Z    \|_1}{\sqrt{p} \| W_t Z \|^3} + \sum_{t } 3  | P_\tau (t) |  \frac{ \| ZW_tZ^T	 \otimes  W_t Z    \|_1}{\sqrt{p} \| W_t Z \|^5} \leq  \frac{6 p}{  \| W_t Z \|^2 }.
\]
Then the constant $C_q$ figuring in Section \ref{complex_max} has bound 
\[
C_q \leq \frac{\sqrt{6 p} \triangle }{  \| W_t Z \| } \lesssim 1,
\]
and 	therefore with $T = n$ one get
\[
C_{\mu} =  \frac{\max (g', g'', g''')}{6} \left( 5 + 6 \log n +   \log^2 n   \right),
 \]
 \[
 C_{\Sigma} = \frac{\max (g', g'')}{2} \left( 2 + \log n  \right).
 \]
Finally the bootstrap approximation error for statistic (\ref{conv_stat}) is 
\begin{EQA}
&& \left | \P \left( \max_{1 \leq \tau \leq n} \sum_{t} P_\tau (t) \|  \xiv_{lr}(t) \|   > z \right) - \Pb \left( \max_{1 \leq \tau \leq n} \sum_{t} P_\tau (t) \|  \xivb_{lr}(t) \|    > z \right) 
\right |  \\
&\quad & \leq \left | \P \left( \max_{1 \leq \tau \leq n} \sum_{t} P_\tau (t) \|  \xiv_{lr}(t) \|    > z \right) - \Pb \left( \max_{1 \leq \tau \leq n} \sum_{t} P_\tau (t) \|  \txiv_{lr}(t) \|   > z \right) 
\right | \\
&\quad & +
\left | \P \left( \max_{1 \leq \tau \leq n} \sum_{t} P_\tau (t) \|  \txiv_{lr}(t) \|   > z \right) - \Pb \left( \max_{1 \leq \tau \leq n} \sum_{t} P_\tau (t) \|  \xivb_{lr}(t) \|    > z \right) 
\right |  \\
&\quad & \leq 5 C_{\mu}^{1/3} C_A  \mu_Z +  4 C^{1/2}_{\Sigma} C_{A}  \|\Sigma_{A\xi} - \Sigma_{A \xi}^{\flat} \|^{1/2}_{\infty},
\end{EQA}
From (\ref{dSA}) follows
\begin{EQA}
\|\Sigma_{A\xi} - \Sigma_{A \xi}^{\flat} \|_{\infty} &\leq &   10 \sqrt{\log(np)} \vertiii{ A } \| A \|_{\infty}  \| \Sigma \|_{\infty}  ( 3 + \| b \|) + \| A \|_{\infty}^2 \| b \|^2 \\ 
& = & 10 \sqrt{\log(np)} \sqrt{2h}   \| \Sigma_{\xi} \|_{\infty}  ( 3 + \| b \|) +  \| b \|^2 .
\end{EQA}
For parameter $C_A$ from Lemma \ref{anti_conc_ext} one have to estimate $\sigma_1$, $\sigma_2$ and $a_p$ for the case $\tX = A \txiv$
\[
 p \sigma^2_1 \leq   \sum_{t,s}   P_\tau(t) P_\tau(s)  \gamma_t^T W_t \Sigma  W_s \gamma_s  \leq p \sigma^2_2 ,
 \]
\[
 \gamma = (\gamma_\tau, \ldots, \gamma_{\tau+h}),
 \quad 
 \gamma_t \in  \argmax_{\| \gamma_t \| = 1} \gamma_t^T W_t\tX.
\]
The domain of $\gamma$ can be restricted with high probability by condition
\[
p \sigma^2_1 \leq  \sum_{t,s} P_\tau(t) P_\tau(s) \frac{\tX^T}{ \| W_t \tX \| }   W_t \Sigma  W_s  \frac{\tX}{ \| W_s \tX \| }  \leq p  \sigma^2_2,
\] 
Define random variables $w_1^2$, $w_2^2$ with the same distribution as $\| W_t \tX \| \| W_s \tX \|$ and $P_\tau = \sum_t P_\tau(t) W_t$, then 
\[
p \sigma^2_1 w_1^2 \leq  \tX^T    P_\tau \Sigma  P_\tau \tX  \leq p  \sigma^2_2 w_2^2.
\] 
In the middle is a chi-square random variable with degrees of freedom dependent from matrix  $B = \Sigma^{1/2} P_\tau \Sigma  P_\tau \Sigma^{1/2}$. With regarding  chi-square deviations (Theorem \ref{TexpbLGA}) it holds with high probability that
\[
\tX^T    P_\tau \Sigma  P_\tau \tX \geq \tr(B) -  \sqrt{\tr(B^2) \log n},
\]
and correspondingly one may set bounds $\sigma_1$, $\sigma_2$ as follows
\[
 \sigma^2_1   = \min_t  \frac{ \tr(B) -  \sqrt{\tr(B^2) \log n} }{p \, \zq^2(W_t \Sigma W_t, \log n) }, 
 \quad  
 \sigma^2_2 = \max_t  \normop { W_t \Sigma W_t }/ p . 
\] 
For parameter $a_p$ one have to find bound for  $ \E \max_t Q_t(A \txiv) $. From conditions $\sum_t |P_\tau (t)| = 1$ and $   \| W_t Z \|  \leq   \zq (W_t \Sigma W_t, x)  $ (Theorem \ref{LLbrevelocroB}) follows that 
\[
\E \max_t Q_t(A \txiv)  \leq \zq (W_t \Sigma W_t, \log n),
\]
and furthermore 
\[
a_p \leq \frac{ \E \max_t Q_t(A \txiv) }{\sigma_1} \leq \frac{ \sqrt{p} \zq^2 (W_t \Sigma W_t, \log n)}{ (\tr(B) -  \sqrt{\tr(B^2) \log n} )^{1/2} }.
\]
The final estimation for $C_A$ is
\[
C_A \approx 4 \frac{\sigma_2}{\sigma_1^2} a_p \leq  \max_t  \normop { W_t \Sigma W_t } \frac{ p \zq^4 (W_t \Sigma W_t, \log n)}{ (\tr(B) -  \sqrt{\tr(B^2) \log n} )^{3/2} }.
\]

\end{proof}

   \subsection{Generalized linear model case}
   \label{glm_sec}

Generalized linear models (GLM) are frequently used for modeling 
the data with special structure:
categorical data, binary data, 
Poisson and exponential data, volatility models, etc.
All these examples can be treated in a unified way by a GLM approach.
This section specifies the results and conditions to this case.
Let \( \Ybb = (Y_{1},\ldots, Y_{\nsize}) \sim \P \) be a sample of independent r.v.'s. The parametric GLM is given by $Y_i \sim \P_{\psi_i^T\theta}$, where $\psi_i$ are given factors in $\R^p$.
Generalised linear model may be presented in form
\[
L(\theta) = S^T \theta - A(\theta)
\]
where
\[
A(\theta) = \sum_{i=1}^n g(\psi_i^T \theta), 
\quad
S = \sum_{i=1}^n Y_i \psi_i.
\]
This model has following properties (ref. Section \ref{likelihood})
\[
- \nabla^2 \E L(\theta)  = D^2(\theta) = \sum_{i=1}^n g''(\psi_i^T \theta) \psi_i \psi_i^T,
\]
\[
\delta(r) = a_g \delta_{\psi} r,
\]
where
\[  a_g = \max_x \frac{g'''(x)}{g''(x)},
\quad \delta_{\psi} = \max_i \normp{D^{-1} \psi_i}.
\]
As \( \GLMlink(\cdot) \) is convex, it holds \( \GLMlink''(u) \geq 0 \) for any \( u \) and thus \( D^2(\theta) \geq 0 \).
An important feature of a GLM is that the stochastic component \( \zeta(\theta) \) of 
\( L(\theta) \) is \emph{linear in} \( \theta \): 
with \( \varepsilon_{i} = Y_{i} - \E Y_{i} \)
\[
	\zeta(\theta)=
	\sum_{i=1}^{\nsize} \varepsilon_{i} \psi_{i}^{\T} \theta ,
	\quad
	\xiv = D^{-1} \sum_{i=1}^{\nsize} \varepsilon_{i} \psi_{i}.
\label{nablaGLMdef}
\]
Linearity in \( \theta \) of the stochastic component \( \zeta(\theta) \) and 
concavity of the deterministic part \( \E L(\theta) \) allow for
a simple and straightforward proof of the result
about localisation of the MLE \( \opttheta \) in the region $\Thetas(\rups)$ (ref. condition \ref{cond_L} and Theorem 2.1 in \cite{wilks2013}).

\begin{theorem}
\label{TGLMsolution}
If for some \( \rups > 0 \), \( D(\theta) \)  fulfils  \ref{cond_dD} condition
with \( \rddelta(\rups) < 1 \), 
and if \( \xiv \)  is a sub-Gaussian vector (\ref{expgamgm}) with variance matrix $B$
then 
\begin{EQA}
	\P\bigl( \opttheta \not\in \Thetas(\rups) \bigr)
	& \leq &
	2 \ex^{-\xx} 
\label{PttsTsGLM}
\end{EQA}
where lover bound for $\rups$ is defined by $\zq(\BB,\xx)$ from Lemma \ref{LLbrevelocroB}  
\begin{EQA}
	\rups \{ 1 - \rddelta(\rups) \}
	& \geq &
	2 \zq(\BB,\xx) .
\label{rups1mrdrups}
\end{EQA}
\end{theorem}

\begin{proof}
The function \( L(\theta) \) is concave in \( \theta \)
because 
\begin{EQA}
	- \nabla^{2} L(\theta)
	&=&
	D(\theta) 
	\geq 
	0 .
\label{nab2LthIF}
\end{EQA}
Therefore, it suffices to check that for each \( \theta \) with 
\( \| D (\theta - \thetas) \| = \rups \) that with probability \( 1 - 2 \ex^{-\xx} \)
\begin{EQA}
	L(\thetas) - L(\theta)
	& > &
	0.
\label{Ltr0sLtGLM}
\end{EQA}
\begin{EQA}
	L(\thetas) - L(\theta)
	&=&
	(\theta - \thetas)^{\T} \nabla L(\thetas)
	+ \frac{1}{2} \bigl\| D(\thetad) (\theta - \thetas) \bigr\|^{2}
	\\
	& \geq &
	(\nablaGLM - \E \nablaGLM)^{\T} (\theta - \thetas)
	+ \frac{1 - \rddelta(\rups)}{2} \| \DPc (\theta - \thetas) \|^{2}
	\\
	&=&
	\xiv^{\T} \DPc (\theta - \thetas)
	+ \frac{1 - \rddelta(\rups)}{2} \rups^{2} 
	\geq 
	- \| \xiv \| \, \rups
	+ \frac{1 - \rddelta(\rups)}{2} \rups^{2}.
\label{LtsLtTaGLM}
\end{EQA}
If \( \| \xiv \| \leq \rups \bigl\{ 1 - \rddelta(\rups) \bigr\}/2 \), then this implies
\( L(\thetas) - L(\theta) > 0 \), and the result follows.  Theorem \ref{LLbrevelocroB} gives a probabilistic restriction for $\| \xiv \|$:
\begin{EQA}
	\P\Bigl( 
		\| \xiv \| > \zq(\BB,\xx)
	\Bigr)
	& \leq &
	2 \ex^{-\xx} .
\end{EQA}
\end{proof}
As a corollary, we obtain Fisher and Wilks expansions (condition \ref{cond_A} holds for the case $\theta = \opttheta$)
for  generalized linear models with 
\[
\diamondsuit (\rups, \xx) = \delta (\rups)  \rups = a_g \delta_{\psi} \rups^2.
\]
\begin{theorem}
\label{TFWGLM}
Suppose the conditions of Theorem~\ref{TGLMsolution} for some \( \rups \).
Then with probability
\( 1 - 2 \ex^{-\xx} \)
\begin{EQA}
	\bigl\| D \bigl( \opttheta - \thetas \bigr) - \xiv \bigr\|
	& \leq &
	\rups \, \rddelta(\rups),
	\\
    \bigl| 2 L(\opttheta,\thetas) - \| \xiv \|^{2} \bigr|
    & \le &
    2 \rups^{2} \, \rddelta(\rups) +  \rups^{2} \, \rddelta^{2}(\rups) .
    \\
    \Bigl| 
    	\sqrt{ 2L(\opttheta,\thetas) } 
		- \| \xiv \| 
	\Bigr|
    & \le &
    3 \rups \, \rddelta(\rups) .
\label{CorolFWGLM}
\end{EQA}
\end{theorem}

\begin{proof}

\begin{EQA}
	\bigl| \alpha(\theta, \theta^*) \bigr|
	&=&
	\bigl| 
		\GLMLINK(\theta) - \GLMLINK(\thetas) 
		- (\theta - \thetas)^{\T} \nabla \GLMLINK(\thetas) 
		- \| D (\theta - \thetas) \|^{2}/2 
	\bigr|
	\\
	&=&
	\frac{1}{2} \bigl| (\theta - \thetas)^{\T} \bigl\{ D(\thetas) - D(\thetad) \bigr\} 
		(\theta - \thetas) \bigr| ,
\label{ElElaeGLM}
\end{EQA}
where \( \thetad \) is a point on the interval between \( \theta \) and \( \thetas \). Condition \eqref{cond_dD} implies 
\begin{EQA}
	\bigl| \alpha(\theta, \theta^*) \bigr|
	& \leq &
	\frac{\rddelta(\rups)}{2} \bigl\| \DPc (\theta - \thetas) \bigr\|^{2}
	\leq 
	\frac{\rddelta(\rups)}{2} \rups^{2}.
\label{ELELa22GLM}
\end{EQA}
\end{proof}

Assume that for independent $\varepsilon_i = Y_i - \E Y_i$
\[\tag{SGeps}
	\log \E \exp\bigl( \lambda \, \expzeta_{i}^{-1} \varepsilon_{i} \bigr) 
	\leq 
	\frac{1}{2} \nunu^{2} \lambda^{2} ,
	\qquad 
	i=1,\ldots,\nsize,
	\quad
	|\lambda| \leq \gmiid .
\label{EexpleiGLM}
\]
\[
	\VPc^{2}
	= 
	\sum_{i=1}^{\nsize} \expzeta_{i}^{2} \, \psi_{i} \psi_{i}^{\T} .
\label{VPc2GLM}
\]

The squared norm \( \| \xiv \|^{2} \) is a quadratic form of \( \varepsilon_{i} \)
and one can employ upper bound from  
Section~\ref{sub_gaus_exp}.

\begin{theorem}
\label{TGLMcond}
Suppose \eqref{EexpleiGLM},  \( \zq(\dimp,\xx) \leq \sqrt{\dimp} + \sqrt{2 \xx} \) according to Lemma    \ref{LLbrevelocro}, 
fix
\[
	\rups
	=
	4 \nunu \zq(\dimp,\xx) ,
\label{r04xGLM}
\]
Then the conditions of Theorem~\ref{TGLMsolution} are fulfilled with 
\( \rddelta(\rups) \leq \aGLMlink(\rups) \, \dPsi \, \rups \)
and 
\[
	\log \E \exp\bigl\{\gamma^{\T} \xiv \bigr\}
	 \leq 
	\frac{\nunu^{2}}{2} \| D \VPc^{-1} \gamma \|^{2} \, ,
	\qquad
	\|   \gamma \| \leq \gm = \frac{\gmiid}{\dPsi \,  \expzeta_{i}}.
\label{dPsiGLM}
\]

\end{theorem}

\begin{proof}

For each \( \theta \in \Thetas(\rups) \) and \( i \leq \nsize \), it holds 
\[
	\bigl| \psi_{i}^{\T} \theta - \psi_{i}^{\T} \thetas \bigr| 
	=
	\bigl| \bigl( D^{-1} \psi_{i} \bigr)^{\T} D (\theta - \thetas) \bigr|
	\leq 
	\| D^{-1} \psi_{i} \| \, \rups 
	\leq 
	\dPsi \, \rups .
\label{PsiiTtsGLM}
\]
\[
	D(\theta) - D(\thetas)
	=
	\sum_{i=1}^{\nsize} \bigl\{ 
		\GLMlink''(\psi_{i}^{\T} \theta) - \GLMlink''(\psi_{i}^{\T} \thetas) 
	\bigr\} \, \psi_{i} \psi_{i}^{\T} .
\label{IFtIfGLM}
\]
\[
	\GLMlink''(\psi_{i}^{\T} \theta) - \GLMlink''(\psi_{i}^{\T} \thetas) 
	=
	\frac{\GLMlink'''(\psi_{i}^{\T} \thetad)}{\GLMlink''(\psi_{i}^{\T} \thetas)} \,
		\bigl( \psi_{i}^{\T} \theta - \psi_{i}^{\T} \theta \bigr) \,
		\GLMlink''(\psi_{i}^{\T} \thetas) .
\label{GLMliddPts}
\]
\[
	\max_{i \leq \nsize} \,\, 
	\biggl| \frac{\GLMlink'''(\psi_{i}^{\T} \thetad)}{\GLMlink''(\psi_{i}^{\T} \thetas)} \,
		\bigl( \psi_{i}^{\T} \theta - \psi_{i}^{\T} \theta \bigr)
	\biggr|
	 \leq 
	\aGLMlink(\rups) \, \dPsi \, \rups 
\]

The last statement of theorem follows from
\[
	\log \E \exp\bigl\{\gamma^{\T} \VPc^{-1} (\nablaGLM - \E \nablaGLM) \bigr\}
	=
	\sum_{i=1}^{\nsize} \log \E \exp\bigl( \lambda_{i} \, \expzeta_{i}^{-1} \varepsilon_{i} \bigr) ,
\label{logEexpGLM1n}
\]
\[
	|\lambda_{i}|
	 = 
	|\gamma^{\T} \VPc^{-1} \psi_{i}| \, \expzeta_{i} 
	\leq 
	\gm \,\, \| \VPc^{-1} \psi_{i} \| \, \expzeta_{i}  
	\leq 
	\gmiid .
\label{lamiGLM}
\]
\[
	\log \E \exp\bigl\{\gamma^{\T} \VPc^{-1} (\nablaGLM - \E \nablaGLM) \bigr\}
	 \leq 
	\frac{\nunu^{2}}{2} \sum_{i=1}^{\nsize} \lambda_{i}^{2}
	 = 
	\frac{\nunu^{2}}{2} \sum_{i=1}^{\nsize} 
		\gamma^{\T} \VPc^{-1} \bigl( \psi_{i} \psi_{i}^{\T} \, \expzeta_{i}^{2} \bigr) 
		\, \VPc^{-1} \gamma 
	=
	\frac{\nunu^{2}}{2} \| \gamma \|^{2} \, .
\]
\end{proof}

Finally, condition \ref{EexpleiGLM} ensure all the required likelihood restrictions (condition \ref{cond_A} and $\opttheta \in \localr$) from Section \ref{likelihood}.

\subsubsection{Bootstrap Wilks and Fisher expansions}

\label{BFWe}

This Section projects statements from Section \ref{boot_WF} with overriding some aspects for GLM. In this case function \ref{S} doesn't depend on $\Ybb$

\begin{EQA}
S^{\flat}(\theta, \theta_0) = \E_{Y} S^{\flat}(\theta, \theta_0) 
&=& \sum_{i=1}^{n} D^{-1} \{\nabla l_i(\theta) - \nabla l_i(\theta_0) \}  \epsb  \\
&=& \sum_{i = 1}^n D^{-1} \{ \nabla g(\psi_i^{T} \theta) - \nabla g(\psi_i^{T} \theta_0) \} \epsb.
\end{EQA}
The gradients difference has view  with $\theta_1 \in (\theta, \theta_0)$
\[
\nabla g(\psi_i^{T} \theta) - \nabla g(\psi_i^{T} \theta_0) =  g''(\psi_i^{T} \theta_1) \psi_i \psi_i^{T} D^{-1} D (\theta - \theta_0).
\]
Then $S^{\flat}(\theta, \theta_0) $ may be presented as a sum of independent matrices 
\[
S^{\flat}(\theta, \theta_0)  = \left( \sum_{i=1}^n S_i \right) D (\theta - \theta_0),
\]
where
\[
S_i  = d_i d_i^{T} \epsb_i, 
\quad
d_i = \sqrt{g''(\psi_i^{T} \theta_1)}  D^{-1} \psi_i ,
\]
\[
\normp{ d_i }^2 \leq (g''(\psi_i^{T} \theta^*) + \delta(\rups) ) \; \delta_{\psi}^2 = \delta_{d}^2. 
\]
From deviation bound for matrix Gaussian sums (Lemma \ref{CUvepsB}) with probability $1 - 2e^{-\xx}$
\[
 \normop{ \sum_{i=1}^n S_i} \leq  \td_{d}^2 \, \nunu   \| \expzeta  \|  \sqrt{2 \xx} .
\]
\begin{theorem}
\label{glm_boot_WFL}
Theorems \ref{fisher_boot}, \ref{wilks_boot} and \ref{lrt_boot} have  in GLM case following upper bound  
\begin{EQA}
\diamondsuit^{\flat}(\rups, \xx)
& = & \rups \td_{d}^2 \, \nunu   \| \expzeta  \|  \sqrt{2 \xx}  + 2  \diamondsuit(\rups).
\end{EQA}
\end{theorem}
An important variable in Theorem \ref{glm_boot_WFL} is  $\xivb =   D^{-1}L^{\flat}(\opttheta)$.  Conditionally on $\Ybb$ from Theorem \ref{TGLMcond} 
\[
	\log \Eb \exp\bigl\{\gamma^{\T} \VPc^{-1} (\nablaGLM^{\flat} - \E \nablaGLM^{\flat}) \bigr\}
	 \leq 
	\frac{1}{2} \sum_{i=1}^{\nsize} \lambda_{i}^{2}
	=
	\frac{1}{2} \sum_{i=1}^{\nsize} 
		\gamma^{\T} \VPc^{-1} \bigl( \psi_{i} \psi_{i}^{\T} \, \eps_{i}^{2} \bigr) 
		\, \VPc^{-1} \gamma 
\]
By definition of  $\VPc$ 
\[
\E \left\{ \sum_{i=1}^{\nsize} 
		 \VPc^{-1} \bigl( \psi_{i} \psi_{i}^{\T} \, \eps_{i}^{2} \bigr) 
		\, \VPc^{-1} \right \}  =  I
\]
Applying consequence of Theorem \ref{kolch_matrix} it holds with probability $1- e^{-\xx}$ for $\eta_i = \eps_{i}^{2} - \E \eps_{i}^{2}$
\[
\normop{ \sum_{i=1}^{\nsize} 
		 \VPc^{-1} \bigl( \psi_{i} \psi_{i}^{\T} \, \eta_i \bigr) 
		\, \VPc^{-1} } \leq  \frac{2}{3} R_{\psi} \xx_p +  \vp_{\psi} \sqrt{5 \xx_p},
\]
where 
\[
 \vp_{\psi} \leq \td_{\psi} \max_i \sqrt{\frac{ E \eta_i^2}{\expzeta_i^2} }, 
 \quad
  R_{\psi} \lesssim \td_{\psi}^2 \normp{ \expzeta }_{\infty} \nunu.
\]
\begin{theorem}
\label{dPsiGLMboot}
Assume $ (2/3) R_{\psi} \xx_p +  \vp_{\psi} \sqrt{5 \xx_p}  \leq 1  $ then with probability $1- e^{-\xx}$
\[
	\log \Eb \exp\bigl\{\gamma^{\T} \xivb \bigr\}
	 \leq 
	\frac{\nu_\flat^2}{2} \|  \gamma \|^{2} \, ,
	\quad
	\nu_\flat =  \sqrt{2} \normop{D \VPc^{-1}},
	\quad
	\|   \gamma \| \leq \gm.
\]
\end{theorem}

\subsubsection{Bootstrap likelihood argmax concentration}

An important property of GLM likelihood function is convexity: $ -\nabla^2 L = -\nabla^2 \E L \geq 0$. This property is useful for MLE concentration proof (Theorem \ref{TGLMsolution}).  Bootstrap likelihood is convex with high probability  under an additional condition described in the following statement.
\begin{theorem}
\label{GLMLbconv}
Assume that $\forall i$
\[
\sqrt{\gpsiidev} \normp{D^{-1}(\theta) \psi_i } \sqrt{2 \xx}  < 1,
\]
then with probability $1 - e^{-\xx}$   \[
-\nabla^2 \Lbf > 0.
\]
\end{theorem}
 
\begin{proof}
\[
 - \nabla^2 \Lbf = \sum_{i=1}^n \gpsiidev \psi_i \psi_i^{T} u_i = D(\theta) \left( I + \sum_i d_i d_i^{T} \varepsilon_i^{\flat} \right) D(\theta).
\]
\[
 - \nabla^2 \Lbf  \geq 0  \Leftrightarrow 1 - \lambda_{\max} \left( - \sum_i d_i d_i^{T} \varepsilon_i^{\flat} \right) \geq 0,
\]
where $d_i = \sqrt{\gpsiidev} D^{-1}(\theta) \psi_i$, $\sum_i d_i d_i^{T} = I$. Use Lemma \ref{CUvepsB} to get matrix deviation bound which states that with probability $1 - 2 e^{-\xx}$
\[
\lambda_{\max} \left( - \sum_i d_i d_i^{T} \varepsilon_i^{\flat} \right) \leq
\max_i  \normp{d_i} \sqrt{2 \xx}. 
\]

\end{proof}

Using previous result one is able to prove $\opttheta^{\flat}$ concentration analogically to  Theorem \ref{TGLMsolution}.

\begin{theorem} Under condition from lemma \ref{GLMLbconv} 
\[
\P \left( \normp{D(\opttheta^{\flat} - \opttheta)} > \rups^{\flat} \right) \leq 4 e^{-x},
\quad \rups^{\flat} = 8 \nunu^{\flat} z(p, x).
\]

\end{theorem}

\begin{proof}

Detect sufficiently large radius $r$ of local region  $\Theta(r)$ with $\Lbf < \Lbopt$ on the border. With $\zeta^{\flat} = L^{\flat} - L$: 
\[
\Lbf - \Lbopt = \zeta^{\flat} (\theta) - \zeta^{\flat}(\widehat{\theta}) + \Lf - \Lfopt. 
\]  
\[
\Lf - \Lfopt \leq -  \frac{1 - \delta(2r)}{2} r^{2}. 
\]
Use variable \ref{S}, from its definition
\[
\zeta^{\flat} (\theta) - \zeta^{\flat}(\widehat{\theta}) - \nabla \zeta^{\flat}(\widehat{\theta})^{T} (\theta - \opttheta)  = S^{\flat}(\theta, \opttheta)^T D(\theta - \opttheta).
\]
According to Theorem \ref{glm_boot_WFL} 
$ \quad
\normp{S^{\flat}(\theta, \opttheta)} \leq \diamondsuit^{\flat}(r,x)
$.
Gradient $\zeta^{\flat}(\widehat{\theta})$ may be estimated by means of sub exponential property 
\ref{dPsiGLMboot} and bound for $\normp{\xivb}$ provided by Theorem \ref{LLbrevelocro}.
\[
D^{-1} \nabla \zeta^{\flat}(\widehat{\theta}) = \xivb
\]
and
\[
\Pb \left(\normp{\xivb} > \nunu^{\flat} z(p, x)  \right) \leq 2 e^{-x}.
\]
Finally, the proving statement  requires condition (assume $1 - \delta(2r) >  1/2$ and $\diamondsuit^{\flat}(r,x)/r < 1$)
\[
\frac{1}{4} r^2 > r^2 \normp{S^{\flat}(\theta, \opttheta)} + r \nunu^{\flat} z(p, x) \leq 2 r \nunu^{\flat} z(p, x).
\]

\end{proof}

\subsection{Matrix Bernstein inequality}

\begin{lemma}[Master bound]
\label{MastBound}
Assume that  \(S_{1},\dots,S_{n}\) are independent Hermitian matrices of the same size and  
\(Z = \sum_{i=1}^{n} S_{i} \). 
Then 
\begin{EQA}
	\E \supA_{\max}(Z)
	& \leq &
	\inf_{\theta>0}\frac{1}{\theta} 
	\log \tr\exp\left(\sum_{i=1}^{n}\log\E\ex^{\theta S_{i}}\right),
\label{MaCheEx}
	\\
	\P\{\supA_{\max}(Z) \geq \zq\} 
	& \leq & 
	\inf_{\theta>0} \ex^{-\theta \zq} \, 
		\tr\exp\left(\sum_{i=1}^{n}\log\E \ex^{\theta S_{i}}\right).
\label{MaCheProb}
\end{EQA}
\end{lemma}

\begin{proof} 
By the Markov inequality 
\begin{EQA}
	\P\{\supA_{\max}(Z) \geq \zq\} 
	&\leq&
 	\inf_{\theta} \ex^{-\theta \zq} \E \exp(\theta \supA_{\max}(Z)).
\end{EQA}
Recall the spectral mapping theorem: for any function \( f\colon \R \to\R \) 
and Hermitian matrix \( A \) eigenvalues of \(f(A)\) are equal to eigenvalues of \(A\). 
Thus 
\begin{EQA}
	\exp(\theta \supA_{\max}(Z))
 	&= &
 	\exp( \supA_{\max}(\theta Z) ) 
 	= 
 	\supA_{\max}\bigl( \exp(\theta Z ) \bigr) 
	\leq 
	\tr \, \ex^{\theta Z} .
\end{EQA}
Therefore,
\begin{EQA}
	\P\{\supA_{\max}(Z) \geq \zq\}
	&\leq &
	\inf_{\theta} \ex^{-\theta \zq} \E\tr\exp(\theta Z),
\label{Tr3.2.1}
\end{EQA}
and \eqref{MaCheProb} follows.

To prove \eqref{MaCheEx} fix \(\theta\). 
Using the spectral mapping theorem one can get that 
\begin{EQA}
	\E \supA_{\max}(Z) 
	& = &
	\frac{1}{\theta}\E \supA_{\max}(\theta Z)  
	= 
	\frac{1}{\theta} \log\E\exp \bigl( \supA_{\max}(\theta Z) \bigr)
	=
	\frac{1}{\theta} \log\E \supA_{\max}\bigl( \exp (\theta Z) \bigr) \, . 
\label{Tr3.2.2}
\end{EQA}
Thus we get 
\begin{EQA}
	\E\supA_{\max}(Z)
	& \leq &
	\frac{1}{\theta} \log \tr \E \exp(\theta Z).
\label{Tr3.2.2.1}
\end{EQA}
The final step in proving the master inequalities is to bound from above 
\( \E\tr\exp\left(\sum_{i=1}^{n} S_{i} \right)\). 
To do this we use Jensen's inequality for the convex function  
\( \tr \exp(H + \log(X)) \) (in matrix \(X\)), where \(H\) is deterministic Hermitian  matrix.  For a  random Hermitian matrix \( X \) one can write
\begin{EQA}
	\E \tr\exp(H + X) 
	&=& 
	\E \tr \exp(H + \log \ex^{X})
	\leq   
	\tr \exp(H + \log \E \ex^{X}).
\label{Tr3.4.1}
\end{EQA}
Convexity of function $( \tr \exp(H + \log(X)) )$ is followed from   
\[
\tr \exp(H + \log(X)) = \max_{Y \succ 0} [\tr(YH) - (D(Y; X) - \tr X)],
\]
where $D(Y; X)$ is relative entropy 
\[
D(Y; X) = \phi(X) - [\phi(Y) + \langle \nabla\phi(Y), X - Y \rangle],
\quad \phi(X) = \tr(X \log X)
\]
due to the partial maximum and $D(Y; X)$ are concave functions.

Denote by \( \E_{i} \) the conditional expectation with respect to  random matrix \( X_{i} \).  To bound \( \E \tr\exp\left(\sum_{i=1}^{n} S_{i}\right) \)  we use \eqref{Tr3.4.1} for the sum of independent Hermitian matrices by taking the conditional expectations with respect to \( i \)-th matrix:
\begin{EQA}
	\E\tr\exp\left(\sum_{i=1}^{n} S_{i}\right) 
 	&= &
	\E\E_{n}\tr\exp\left(\sum_{i=1}^{n-1} S_{i} + S_{n}\right) 
	\\ 
 	& \leq & 
	\E\tr\exp\left(\sum_{i=1}^{n-1} S_{i} + \log(\E_{n}\exp(S_{n}))\right) 
	\\
	& \leq &  
	\tr\exp\left(\sum_{i=1}^{n}\log\E \ex^{\theta S_{i}} \right).
\label{Tr3.4.1}
\end{EQA}
To complete the prove of the Master's theorem combine  \eqref{Tr3.2.1} and \eqref{Tr3.2.2.1} with \eqref{Tr3.4.1}.
\end{proof}

The next statement was taken from \cite{koltchinskii} with the proof sketch.

\begin{lemma}[Bernstein inequality for moment restricted matrices]
\label{kolch_matrix}
Suppose that 
\[
\forall i: \; \E \psi^2 \left( \frac{\normop{S_i}}{M} \right) \leq 1, 
\]
\[
	\vp^{2} 
	=
	\left\|\sum_{i=1}^{n}\E S_{i}^{2}\right\|_{\oper}  ,
\]
\[
 R = M \psi^{-1} \left( \frac{2}{\delta} \frac{n M^2}{ \vp^{2}  } \right ),
 \quad
 \delta \in (0, 2/\psi(1)).
\]
Then under condition $\zq R \leq (e - 1)(1 + \delta) \vp^{2}$ 
\[
\P\{\| Z \|_{\oper} \geq \zq \} \leq 2p \exp \left \{ - \frac{\zq^2}{ 2(1 + \delta) \vp^2 + 2R \zq/3 } \right \}. 
\]
If $\psi(u) = e^{u^{\alpha}} - 1$ then $R = M \log^{1/\alpha} ( \frac{2}{\delta} \frac{n M^2}{ \vp^{2}  }  + 1)$.
\end{lemma}

\begin{proof}
According to Master bound one have to estimate $\E e^{\theta S }$ for $S$ in $S_1,\ldots,S_n$. Denote a  function 
\[
f(u) = \frac{e^u - 1 - u}{u^2}
\]
Taylor expansion yields 
\[
\E e^{\theta S } \leq I_p + \theta^2 \E S^2 f(\theta \normp{S})
\] 
\[
\log \E e^{\theta S } \leq \theta^2 \E S^2 f(\theta \normp{S}) \leq 
\theta^2 f(\theta \tau) \E S^2  + I_p \theta^2 \E \normp{S}^2 f(\theta \normp{S}) I(\normp{S} \geq \tau)
\]
\[
\E \normp{S}^2 f(\theta \normp{S}) I(\normp{S} \geq \tau) \leq
M^2 \E \psi^2 \left( \frac{\normp{S}}{M} \right) \left( \psi \left( \frac{\tau}{M} \right) \right)^{-1}
\leq 
M^2 \left( \psi \left( \frac{\tau}{M} \right) \right)^{-1}
 \]
 \[
 M^2 \left( \psi \left( \frac{R}{M} \right) \right)^{-1} = \frac{\delta \vp^2}{2n}
 \]
 \[
 \log \E e^{\theta S } \leq 
\theta^2 f(\theta R) \E S^2  + I_p \theta^2  \frac{\delta \vp^2}{2n}
 \]
 \[
  \tr\exp \left( \sum_{i=1}^{n} \log \E\exp(\theta S_{i})\right) \leq 
  \tr\exp \left( \theta^2 f(\theta R) \E \sum_i S_i^2 \right) \exp \left(    \theta^2  \frac{\delta \vp^2}{2} \right).
 \]

\end{proof}

\begin{cons}
In case  $\psi (u)= e^u - 1$ with probability $1 - 2 e^{-x}$
\[
z \leq \frac{2}{3} R \xx_p +  \vp \sqrt{5 \xx_p}, 
\quad
R \approx M.
\]
\end{cons}

Condition for function $\psi (u)= e^u - 1$ follows from sub-Gaussian moment restriction 
\[
\E \exp  \left( 2 \frac{\normop{S_i}}{M} \right) \leq 2.
\]

\begin{lemma}[Deviation bound for matrix convolution with sub-Gaussian weights]
\label{sum_eA}
Let a set of symmetric $[p \times p]$ matrices $(A_1,\ldots,A_n)$ satisfy 
\[
\left \| \sum_i A_i^2 \right  \| \leq \vp^2.
\]
Let \(\eps_{i}\) be independent sub-Gaussian, \( i=1,\ldots,n \).  
\begin{EQA}
	Z
	& = &
	\sum_{i=1}^{n} \eps_{i} A_i
\end{EQA}
fulfills
\begin{EQA}
	\P\biggl( 
		\| Z \|_{\oper}
		\geq  
		    \sqrt{2 \xx_p \vp^2} 
	\biggr)
	& \leq &
	2  \ex^{-\xx},
\end{EQA}
where $\xx_p  = \xx + \log p$.
\end{lemma}

\begin{proof}
Apply the Master inequalities for the case
\[
A_i = U_i \Lambda_i U_i^T,
\quad
\log \E e^{\eps_i A_i } = \frac{1}{2}  U_i \Lambda^2_i U_i^T = \frac{1}{2} A_i^2,
\]
\[
\tr \exp \left \{ \sum_i  \log \E e^{\theta \eps_i A_i } \right \} \leq p \exp \left \{ \frac{1}{2} \theta^2 \vp^2  \right \}.
\]
\end{proof}

\begin{lemma}[Deviation bound for rank one matrix convolution with sub-Gaussian weights]
\label{CUvepsB}
Let vectors \(u_{1},\dots,u_{n}\) in \( \R^{\dimp} \) satisfy
\begin{EQA}
\label{BoundRowmat}
 	\|u_{i}\|
 	&\leq &
 	\td
\end{EQA}
for a fixed constant \(\td\).
Let \(\eps_{i}\) be independent sub-Gaussian, \( i=1,\ldots,n \).  
Then for each vector \( b = (b_{1},\ldots,b_{n})^{\T} \in \R^{n} \), 
the matrix \( Z_{1} \) with
\begin{EQA}
	Z_{1}
	& = &
	\sum_{i=1}^{n} \eps_{i} b_{i} u_{i} u_{i}^{\T}
\end{EQA}
fulfills
\begin{EQA}
	\P\biggl( 
		\| Z_{1} \|_{\oper}
		\geq  
		\td^{2} \| b \| \sqrt{2 \xx} 
	\biggr)
	& \leq &
	2 \ex^{-\xx}.
\end{EQA}
\end{lemma}

\begin{proof}
As \( \eps_{i} \) are i.i.d. standard sub-Gaussian and \( \E \ex^{a \eps_{i}} \leq \ex^{a^{2}/2} \) for 
\( |a| < 1/2 \),
it follows from the Master inequality and property \[ \exp(\eps_i  u u^{\T})  = \frac{ u_{i} u_{i}^{\T} } { u^2 }  \exp(\eps_i  u^2 )\] that
\begin{EQA}
\label{matrix_chernoff1}
	\P\bigl( 
		\| Z_{1} \|_{\oper} \geq \zq
	\bigr)
	& \leq &
	2 \inf_{\theta > 0} \ex^{-\theta \zq} 
		\tr \exp \biggl\{ 
		\sum_{i=1}^{n} \log \E\exp( \theta \eps_{i} b_{i} u_{i} u_{i}^{\T}) 
	\biggr\} 
	\\
	& \leq &
	2 \inf_{\theta > 0} \ex^{-\theta \zq} 
	\tr \exp \biggl\{ 
		\sum_{i=1}^{n} \frac{\theta^{2} b_{i}^{2} \| u_{i} \|^{4}}{2} \,\,
		\frac{u_{i}u_{i}^{\T}}{\|u_{i}\|^{2}} 
	\biggr\} .
\end{EQA}
Moreover, as \( \| u_{i} \| \leq \td \) and 
\( U_{i} = u_{i} u_{i}^{\T} / \| u_{i} \|^{2} \) is a rank-one projector with 
\( \tr U_{i} = 1 \), it holds
\begin{EQA}
	\tr \exp \biggl\{ 
		\frac{\theta^{2}}{2} \sum_{i=1}^{n} b_{i}^{2} \| u_{i} \|^{4} U_{i}
	\biggr\}
	& \leq &
	\exp \tr\biggl( \frac{\theta^{2} \td^{4}}{2} \sum_{i=1}^{n} b_{i}^{2} U_{i} \biggr)
	=
	\exp \frac{\theta^{2} \td^{4} \| b\|^{2}}{2} \, .
\label{trexdPs422}
\end{EQA}
This implies
for \( \zq = \td^{2} \| b \| \sqrt{2 \xx} \)
\begin{EQA}
	\P\bigl( 
		\| Z_{1} \|_{\oper} \geq \zq
	\bigr)
	& \leq &
	2 \inf_{\theta > 0} \exp\biggl(
		- \theta \zq + \frac{1}{2} \theta^{2} \td^{4} \| b \|^{2} 
	\biggr)
	=
	2 \ex^{-\xx} 
\label{PZott0t12}
\end{EQA}
and the assertion follows.
\end{proof}

  \subsection{Covariance matrices}
    
 Consider a sequence independent random variables $\{\varepsilon_i \varepsilon_j^T\}$, $\text{cor} (\varepsilon_i, \varepsilon_j) = \Sigma_{ij} $, flatted into one vector $\varepsilon$. The subject of interest is upper bound for operator norm of
\[
\UV \bldiag{ (\varepsilon \varepsilon^T) } \UV^T  - I_{q},
\]
 \[
 \UV \bldiag{ (\varepsilon \varepsilon^T) } \UV^T = \sum_{ij} \UV_i \varepsilon_i \varepsilon_j^T \UV_j^T, 
\quad
\UV \Sigma \UV^T = I_{q },
\quad
\normop{  \UV_i^T \UV_j } \leq \delta^2.   
 \]
Analogically divide $\varepsilon$ into mean and stochastic parts
\[
\varepsilon = \E \varepsilon + ( \varepsilon -  \E \varepsilon )  = 
B + \zeta.
\]
Then initial term includes three parts: 
\begin{EQA}
\label{SmatrixDiff}
&\UV & \bldiag(BB^{\T})\UV^{\T}  \\
& + & \UV \bldiag(\zeta B^{\T})\UV^{\T} +  \UV \bldiag( B \zeta^T)\UV^{\T} \\
& + &\UV  \bldiag(\zeta \zeta^{T}   -  \Sigma) \UV^{\T}.
\end{EQA}  
Estimate  successively each component.
\begin{EQA}
\normop{\UV \bldiag(BB^{\T})\UV^{\T}} &=& \sup_{\gamma} \sum_{i=1}^n \gamma^T \UV_i B_i B_j^T \UV_j^T \gamma \\
& \leq  &  \sup_{\gamma} \sum_{i=1}^n B_i^2 \normop{  \UV_i^T \UV_j } \leq \td^2 \| B \|^2.
\end{EQA}  
For the second and the third component one may apply Master bound \ref{MastBound}, in which one have to estimate exponential moments of each element of the  $\sum_i \UV_i A_i \UV_i^T$:
\[
\tr \log \E \exp \{ \UV_i A_i \UV_i^T \}.
\]
With condition $\E A_i = 0$ an intuition hint is
\[
\log \E e^{\UV_i A_i \UV_i^T}  \leq  \frac{1}{2} \E (\UV_i A_i \UV_i^T)^2 + O( \normop{\UV_i A_i \UV_i^T}^3) .
\]
Consequently by means of Bernstein matrix inequality \ref{kolch_matrix} one have to restrict the second moment and tail by
$
\log \E \exp \{ \normop{ \UV_i A_{ij} \UV_j^T } \}.
$
\[
\vp^2 = \normop{ \sum_{ij} \E (\UV_i A_{ij} \UV_j^T)^2 } \leq  \max_i \normop{\UV_i^T \UV_j} \normop{ \sum_{ij} \UV_i \E A_{ij}^2 \UV^T_j },
\]
\begin{EQA}
\vp_{\zeta\zeta}^2 / \delta^2 &= & \normop{ \sum_{ij} \UV_i  \E (\zeta_i \zeta_j^T -  \Sigma_{ij})^2 \UV^T_j  } \\
& \leq & \max_{ij} \normop{\Sigma_{ij}} \normop{ \E   (\zeta_i^T \Sigma_{ij}^{-1} \zeta_j) \zeta_i \zeta_j^T \Sigma_{ij}^{-1} -  I  }  \\
& \leq & \max_{ij} \normop{\Sigma_{ij}} \bigg( (\lambda^2 - 1) + \E  (\zeta_i^T \Sigma_{ij}^{-1} \zeta_j)^2 \Ind( \zeta_i^T \Sigma_{ij}^{-1} \zeta_j > \lambda) \bigg), 
\quad \lambda > 1.
\end{EQA}
Upper bound for  $\E \normp{\xi}^4 \Ind( \normp{\xi} > \sqrt{\lambda})$ from \ref{sub_gaus_int} with $\xi^2 = \zeta_i^T \Sigma_{ij}^{-1} \zeta_j$ leads to asymptotic \\  $\vp_{\zeta\zeta}^2 = 9 \delta^2 \max_{ij} \normop{\Sigma_{ij}} p$.
\begin{EQA}
\vp_{\zeta B}^2 /\delta^2 & = & \normop{ \sum_{ij} \UV_i \E (\zeta_i B_i^T)^2 \UV_j^T } \\
  &\leq &  \max_{ij} \normop{\Sigma_{ij}} \sum_{ij} \normop{\UV_j^T \UV_i} B_i^2 \\ 
  & \leq &  \max_{ij} \normop{\Sigma_{ij}} \delta^2 \| B \|^2  .
\end{EQA}
Es for exponential moments for tails restriction
\[
\log \E \exp \{ \normop{  \UV_i A_{ij} \UV_j^T } \}  = \log \E \exp \left \{ \sup_{u_i, u_j}  u_i^T A_{ij} u_j \right \},
\quad 
\normp{u_i}^2, \normp{u_j}^2  \leq  \delta^2.
\]
Sub-Gaussian properties for exponential moments  (Lemma \ref{Lexpxiv}) lead to $\exists M_{\zeta \zeta},$ $M_{\zeta B}$:
\begin{EQA}
&& \log \E \exp \left\{ \sup_{u_i, u_j} \frac{2}{M_{\zeta \zeta}}  u_i^T (\zeta_i \zeta_j^T - \Sigma_{ij})  u_j   \right\} \leq \log \E \exp \left\{ \frac{2 \delta^2 }{M_{\zeta \zeta}}  ( \| \zeta_i \|^2 -   \lambda_{\min} (\Sigma_{ii}))   \right\}   \\
& \quad &  \leq   SG\{ \|\xi\|^2, 2 \delta^2 \normop{ \Sigma_{ii}}  /M_{\zeta\zeta}\} - \frac{2 \delta^2 }{M_{\zeta \zeta}}  \lambda_{\min} (\Sigma_{ii}) \leq \log(2)  ,
\end{EQA}
\[
\log \E \exp \left\{ \sup_{u_i, u_j}  \frac{2}{M_{\zeta B}}  u_i^T \zeta_i B_j^T  u_j  \right\} \leq    SG\{ \| \xi \| , 2 \delta^2 \sqrt{\normop{ \Sigma_{ii}} }  \normp{B_j}  / M_{\zeta B} \} \leq \log(2),
\]
where $ M_{\zeta \zeta} = 3 \delta^2 p \normop{\Sigma_{ii}} $ and $ M_{\zeta B} = 3 \delta^2 p \sqrt{ \normop{\Sigma_{ii}} }			 \normp{B_i}$.
Finally, as a consequence of Theorem   \ref{kolch_matrix} with probability $1 - 2e^{-x}$ and $\xx_{q} = x + \log (2q)$ and $R_{**} \approx M_{**}$
\[
\normop{ \UV \bldiag(\zeta B^{\T})\UV^{\T} } \leq \frac{2}{3} R_{\zeta B} \xx_{q} + 2 \vp_{\zeta B} \sqrt{5 \xx_{q}},
\]
\[
\normop{ \UV \bldiag(\zeta \zeta^{\T})\UV^{\T} } \leq \frac{2}{3} R_{\zeta \zeta}  \xx_{q} + 2 \vp_{\zeta \zeta} \sqrt{ 5 \xx_{q}}.
\]
So, the summarized error with probability $1 - 2e^{-x}$ is
\[
\tag{ErrVD}\label{var_diff_err}
\normop{\UV \bldiag{ (\varepsilon \varepsilon^T) } \UV^T  - I_{q}} \leq
\frac{2}{3} R_{\varepsilon \varepsilon} \xx_{q} + 2 \vp_{\varepsilon \varepsilon} \sqrt{5 \xx_{q}} + \td^2 \normp{B}^2,
\]
where 
\[
R_{\varepsilon \varepsilon} \approx M_{\zeta B} + M_{\zeta \zeta} = 3 \delta^2 p \max_i  \bigg( \normop{\Sigma_{ii}}  + \sqrt{\normop{\Sigma_{ii}} } \normp{B}_{\infty} \bigg)
\]
 and 
 \[ 
 \vp_{\varepsilon \varepsilon} = \vp_{\zeta B} + \vp_{\zeta \zeta} = \delta  \max_i   \sqrt{\normop{\Sigma_{ii}} } ( 3 \sqrt{p} + \| B \|) .
 \]

\subsection{Sub-Gaussian vectors}
\label{sub_gaus_exp}

Consider vector
\(\xi\) which has restricted exponential or sub-Gaussian moments:  \( \exists \; \gm > 0 \):
\[\tag{SG}
    \log \E \exp\bigl( \gamma^{\T}\xi\bigr)
    \le
    \| \gamma \|^{2}/2,
    \qquad
    \gamma \in \R^{\dimp}, \, \| \gamma \| \le \gm .
\label{expgamgm}
\]
Define short notation for upper bounds of such characteristics 
\[
 \log \E \exp\bigl(\lambda X) \leq SG(X, \lambda).
\] 
For ease of presentation, assume below that \( \gm \) is sufficiently large, namely, 
\( 0.3 \gm \ge \sqrt{\dimp} \).
In typical examples of an i.i.d. sample, \( \gm \asymp \sqrt{n} \).
Define $\xxc = \gm^{2}/4$.

\begin{lemma}
\label{LLbrevelocro}   
Let \eqref{expgamgm} hold and 
\( 0.3 \gm \ge \sqrt{\dimp} \).
Then
for each \( \xx > 0 \)
\begin{EQA}
    \P\bigl( \|\xi\| \ge \zq(\dimp,\xx) \bigr)
    & \le &
    2 \ex^{-\xx} + 8.4 \ex^{-\xxc } \Ind(\xx < \xxc) ,
\label{PxivbzzBBro}
\end{EQA}    
where \( \zq(\dimp,\xx) \) is defined by
\begin{EQA}
\label{PzzxxpBro}
    \zq(\dimp,\xx)
    & \eqdef &
    \begin{cases}
        \bigl( \dimp + 2 \sqrt{\dimp \xx} + 2 \xx\bigr)^{1/2}, &  \xx \le \xxc  , \\
        \gm + 2 \gm^{-1} (\xx - \xxc)   , & \xx > \xxc .
    \end{cases}
\label{zzxxppdBlro}
\end{EQA}    
\end{lemma}

Usually
the second term in \eqref{PxivbzzBBro} can be simply ignored. 
Define 
\begin{EQA}[c]
    \dimB
    \eqdef
    \tr \bigl( \BB \bigr) ,
    \qquad 
    \vpB^{2}
    \eqdef
    \tr(\BB^{2}) ,
    \qquad
    \lambdaB \eqdef \lambda_{\max}\bigl( \BB \bigr). 
\label{BBrddB}
\end{EQA}   

\begin{EQA}	
	\zqc^{2}
	& \eqdef &
	\dimA + \vp \gm + \supA \gmb^{2}/2 ,
	\\
	\gmc
	& \eqdef &
	\frac{ \sqrt{\dimA/\supA + \gm \vp / \supA + \gm^{2}/2}}{1 + \vp / (\supA \gmb)} .
\label{xxcgm24mucyyc2B}
\end{EQA}

\begin{lemma}
\label{LLbrevelocroB}   
Let \eqref{expgamgm} hold and 
\( 0.3 \gm \ge \sqrt{\dimA/\supA} \).
Then
for each \( \xx > 0 \)
\begin{EQA}
    \P\bigl( \| \BB^{1/2}\xi\| \ge \zq(\BB,\xx) \bigr)
    & \le &
    2 \ex^{-\xx} + 8.4 \ex^{-\xxc} \Ind(\xx < \xxc) ,
\label{PxivbzzBBroB}
\end{EQA}    
where \( \zq(\BB,\xx) \) is defined by
\begin{EQA}
\label{PzzxxpBroB}
    \zq(\BB,\xx)
    & \eqdef &
    \begin{cases}
        \sqrt{ \dimA + 2 \vp \xx^{1/2} + 2 \supA \xx }, &  \xx \le \xxc, \\
        \zqc + 2 \lambdaB (\xx - \xxc)/\gmc , & \xx > \xxc.
    \end{cases}
\label{zzxxppdBlroB}
\end{EQA}    
\end{lemma}

The upper quantile 
\( \zq(\BB,\xx) = \sqrt{ \dimA + 2 \vp \xx^{1/2} + 2 \supA \xx } \) can be upper bounded 
by \( \sqrt{\dimA} + \sqrt{2 \supA \xx} \):
\begin{EQA}
\label{PzzxxpBroBu}
    \zq(\BB,\xx)
    & \leq &
    \begin{cases}
        \sqrt{\dimA} + \sqrt{2 \supA \xx}, &  \xx \le \xxc, \\
        \zqc + 2 \lambdaB (\xx - \xxc)/\gmc , & \xx > \xxc.
    \end{cases}
\end{EQA}

\begin{lemma}
\label{Lexpxiv}
Suppose \eqref{expgamgm}.
For any \( \mu < 1 \) with
\( \gm^{2} > \dimp \mu \), it holds 
\begin{EQA}
\label{Eexp2xi}
    \E \exp\Bigl( \frac{\mu \|\xi\|^{2}}{2} \Bigr)
        \Ind\Bigl( \|\xi\| \le \gm/\mu - \sqrt{\dimp/\mu} \Bigr)
    & \le &
    2 (1 - \mu)^{-\dimp/2} .
\end{EQA}
\end{lemma}

If $\gm$ is sufficiently large than approximately
\begin{EQA}
  \E \exp\Bigl( \frac{\mu \|\xi\|^{2}}{2} \Bigr)
    &  \le &
     (1 - \mu )^{-p/2} .
\end{EQA}

\begin{lemma}
\label{Lexpxig} 
Suppose \eqref{expgamgm} and \( \| \BB \|_{\oper} = 1 \).
For any \( \mu < 1 \) with
\( \gm^{2}/\mu \ge \dimA \), it holds 
\begin{EQA}
\label{Eexp2xig}
    \E \exp\bigl( {\mu \| \BB^{1/2}\xi\|^{2}}/{2} \bigr)
        \Ind\bigl( \| \BB\xi\| \le \gm/\mu - \sqrt{\dimA/\mu} \bigr)
    & \le &
    2 {\det(\Id_{\dimp} - \mu \BB)^{-1/2}} .
\end{EQA}    
\end{lemma}

\begin{proof}
With \( c_{\dimp}(\BB) = \bigl( 2 \pi \bigr)^{-\dimp/2} \det(\BB^{-1/2}) \)
\begin{EQA}
    && \nquad
    c_{p}(\BB) \int \exp\Bigl( \gamma^{\T}\xi- \frac{1}{2 \mu} \| \BB^{-1/2} \gamma \|^{2} \Bigr)
        \Ind(\| \gamma \| \le \gm) d\gamma
    \\
    &=&
    c_{p}(\BB) \exp\Bigl( \frac{\mu \| \BB^{1/2}\xi\|^{2}}{2} \Bigr)
    \int \exp\Bigl( 
        - \frac{1}{2}  \bigl\| \mu^{1/2} \BB^{1/2}\xi- \mu^{-1/2} \BB^{-1/2} \gamma \bigr\|^{2} 
    \Bigr) \Ind(\| \gamma \| \le \gm) d\gamma    
    \\
    &=&
    \mu^{\dimp/2} \exp\Bigl( \frac{\mu \| \BB^{1/2}\xi\|^{2}}{2} \Bigr) 
    \P_{\xiv}\bigl( 
        \| \mu^{-1/2} \BB^{1/2} \varepsilonv + \BB^{1/2}\xi\| \le \gm / \mu
    \bigr),
\label{intggvvg}
\end{EQA}
where \( \varepsilonv \) denotes a standard normal vector in \( \R^{\dimp} \)
and \( \P_{\xi} \) means the conditional probability given \(\xi\).
Moreover, for any \( u\in \R^{\dimp} \) and \( \rr \ge \dimA^{1/2} + \| u\| \), 
it holds in view of \( \P \bigl( \| \BB^{1/2} \varepsilonv \|^{2} > \dimA \bigr) \le 1/2 \)
\begin{EQA}
    \P\bigl( \| \BB^{1/2} \varepsilonv - u\| \le \rr \bigr)
    & \ge &
    \P\bigl( \| \BB^{1/2} \varepsilonv \| \le \sqrt{\dimA} \bigr)
    \ge 
    1/2 .
\label{Arepsv}
\end{EQA}    
This implies
\begin{EQA}
    && 
    \nquad
    \exp\Bigl( \mu \| \BB^{1/2}\xi\|^{2} / 2 \Bigr) 
        \Ind\bigl( \| \BB\xi\| \le \gm / \mu - \sqrt{\dimA/\mu} \bigr)
    \\
    & \le &
    2 \mu^{- \dimp/2} c_{p}(\BB)
    \int \exp\Bigl( \gamma^{\T}\xi- \frac{1}{2 \mu} \| \BB^{-1/2} \gamma \|^{2} \Bigr)
        \Ind(\| \gamma \| \le \gm) d\gamma .
\label{expxiv1cpg}
\end{EQA}    
Further, by \eqref{expgamgm}
\begin{EQA}
    && 
    \nquad
    c_{p}(\BB) \E \int \exp\Bigl( 
    	\gamma^{\T}\xi- \frac{1}{2\mu} \| \BB^{-1/2} \gamma \|^{2} 
    \Bigr)
        \Ind(\| \gamma \| \le \gm) d\gamma
    \\
    & \le &
    c_{p}(\BB) \int \exp\Bigl( 
        \frac{\| \gamma \|^{2}}{2} - \frac{1}{2\mu} \| \BB^{-1/2} \gamma \|^{2} 
    \Bigr) d\gamma
    \\
    & \le &
    \det(\BB^{-1/2}) \det(\mu^{-1} \BB^{-1} - \Id_{\dimp})^{-1/2}
    =
    \mu^{p/2} \det( \Id_{\dimp} - \mu \BB)^{-1/2}
\label{nununug}
\end{EQA}    
and \eqref{Eexp2xig} follows.
\end{proof}

The next object of interest in this topic is $\E \|\xi\|^r \Ind (\|\xi\| > t)$. Rather useful  form of it is 
\[
\E \|\xi\|^r \Ind (\|\xi\| > t) = \P(\|\xi\| > t) t^r + r \int_{t}^{+\infty} \P(\|\xi\| > t) t^{r-1} d t.
\] 
With $x_0 = (t - \sqrt{p})^2/2$
\[
\int_{t}^{+\infty} \P(\|\xi\| > t) t^{r-1} d t = \int_{x_0}^{+\infty} 2 e^{-x} (\sqrt{2x} + \sqrt{p})^{r-1}  d \sqrt{2x}  \leq
\frac{ 2 e^{-x_0} t^{r-1} } { 1 - (r-1) \log(x_0) / x_0  }
\]
Consequently, if $(r-1) \log(x_0) / x_0 \leq 1/2$, than
\[\tag{SGI}\label{sub_gaus_int}
\E \|\xi\|^r \Ind (\|\xi\| > t)  \leq 2 e^{-(t - \sqrt{p})^2/2} \left( t^r + 2r t^{r-1} \right).
\]
Analogically one is able to restrict moment with exponent part and $r \log(x_0) / x_0 + \alpha \leq 1/2$
\[\tag{SGIexp}\label{sub_gaus_int_exp}
\E \|\xi\|^r e^{\alpha \|\xi\|} \Ind (\|\xi\| > t)  \leq 4 e^{-(t - \sqrt{p})^2/2 + \alpha t} \left( t^r + r t^{r-1} \right).
\]

\subsection{Gaussian quadratic forms}

The next result explains the concentration effect of 
\( \gaussv^{\T} \BB \gaussv \)
for a standard Gaussian vector \( \gaussv \) and a symmetric matrix \( \BB \).
We use a version from \cite{laurent2000} with a complete proof.

\begin{theorem}
\label{TexpbLGA}
\label{Lxiv2LD}
\label{Cuvepsuv0}
Let \( \gaussv \) be a standard normal Gaussian vector and \( \BB \) be symmetric positively
definite \( \dimp \times \dimp \)-matrix.
Then with \( \dimA = \tr(\BB) \), \( \vA^{2} = \tr(\BB^{2}) \), and 
\( \supA = \| \BB \|_{\oper} \), it holds for each \( \xx \geq 0 \)
\begin{EQA}
\label{Pxiv2dimAvp12}
	\P\Bigl( \gaussv^{\T} \BB \gaussv > \zq^{2}(\BB,\xx) \Bigr)
	& \leq &
	\ex^{-\xx} ,
	\\
	\zq(\BB,\xx)
	& \eqdef &
	\sqrt{\dimA + 2 \vA \xx^{1/2} + 2 \supA \xx} .
\label{zqdefGQF}
\end{EQA}
In particular, it implies 
\begin{EQA}
	\P\bigl( \| \BB^{1/2} \gaussv \| > \dimA^{1/2} + (2 \supA \xx)^{1/2} \bigr)
	& \leq &
	\ex^{-\xx} .
\label{Pxiv2dimAxx12}
\end{EQA}
Also
\begin{EQA}
	\P\bigl( \gaussv^{\T} \BB \gaussv < \dimA - 2 \vA \xx^{1/2} \bigr)
	& \leq &
	\ex^{-\xx} .
\label{Pxiv2dimAvp12m}
\end{EQA}
If \( \BB \) is symmetric but non necessarily positive then
\begin{EQA}
	\P\bigl( \bigl| \gaussv^{\T} \BB \gaussv - \dimA \bigr| > 2 \vA \xx^{1/2} + 2 \supA \xx \bigr)
	& \leq &
	2 \ex^{-\xx} .
\label{PxivTBBdimA2vp}
\end{EQA}
\end{theorem}

{
\begin{proof}
Normalisation by \( \supA \) reduces the statement to the case with \( \supA = 1 \).
Further, the standard rotating arguments allow to reduce the Gaussian quadratic form 
\( \| \gaussv \|^{2} \) to the chi-squared form:
\begin{EQA}
	\gaussv^{\T} \BB \gaussv
	&=&
	\sum_{j=1}^{\dimp} \lambda_{j} \nu_{j}^{2}
\label{xiv2sj1p}
\end{EQA}
with independent standard normal r.v.'s \( \nu_{j} \).
Here \( \lambda_{j} \in [0,1] \) are eigenvalues of \( \BB \), and 
\( \dimA = \lambda_{1} + \ldots + \lambda_{\dimp} \), 
\( \vA^{2} = \lambda_{1}^{2} + \ldots + \lambda_{\dimp}^{2} \).
One can easily 
compute the exponential moment of \( (\gaussv^{\T} \BB \gaussv - \dimA)/2 \):
for each positive \( \mu < 1 \)
\begin{EQA}
	\log \E \exp\bigl\{ \mu (\gaussv^{\T} \BB \gaussv - \dimA)/2 \bigr\}
	&=&
	\frac{1}{2} \sum_{j=1}^{\dimp} \bigl\{ - \mu \lambda_{j} - \log(1 - \mu \lambda_{j}) \bigr\} .
\label{lEemux2p2}
\end{EQA}

\begin{lemma}
Let \( \mu \lambda_{j} < 1 \) and \( \lambda_{j} \leq 1 \).
Then 
\begin{EQA}
	\frac{1}{2} \sum_{j=1}^{\dimp} \bigl\{ - \mu \lambda_{j} - \log(1 - \mu \lambda_{j}) \bigr\}
	& \leq &
	\frac{\mu^{2} \vA^{2}}{4 (1 - \mu)} \, .
\label{jmu2v221mu}
\end{EQA}
\end{lemma}

\begin{proof}
In view of \( \mu \lambda_{j} < 1 \), it holds for every \( j \)
\begin{EQA}
	- \mu \lambda_{j} - \log(1 - \mu \lambda_{j}) 
	&=&
	\sum_{k=2}^{\infty} \frac{(\mu \lambda_{j})^{k}}{k}
	\\
	& \leq &
	\frac{(\mu \lambda_{j})^{2}}{2}
	\sum_{k=0}^{\infty} (\mu \lambda_{j})^{k}
	\leq 
	\frac{(\mu \lambda_{j})^{2}}{2 (1 - \mu \lambda_{j})} 
	\leq 
	\frac{(\mu \lambda_{j})^{2}}{2 (1 - \mu)},
\label{jmu2v221mup}
\end{EQA}
and thus
\begin{EQA}
	\frac{1}{2} \sum_{j=1}^{\dimp} \bigl\{ - \mu \lambda_{j} - \log(1 - \mu \lambda_{j}) \bigr\}
	& \leq &
	\sum_{j=1}^{\dimp} \frac{(\mu \lambda_{j})^{2}}{4 (1 - \mu)} 
	\leq 
	\frac{\mu^{2} \vA^{2}}{4 (1 - \mu)} \, .
\label{sjmu2v221mu}
\end{EQA}
\end{proof}
The next technical lemma is helpful.

\begin{lemma}
\label{Lmuvpxx}
For each \( \vA > 0 \) and \( \xx > 0 \), it holds
\begin{EQA}
	\inf_{\mu > 0} \biggl\{ 
		- \mu \bigl( \vA \xx^{1/2} + \xx \bigr) + \frac{\mu^{2} \vA^{2}}{4 (1 - \mu)} 
	\biggr\}
	& \leq &
	- \xx .
\label{infmuxxvp}
\end{EQA}
\end{lemma}

\begin{proof}
Let pick up 
\begin{EQA}
	\mu 
	&=& 
	1 - \frac{1}{2\xx^{1/2}/\vA + 1} = \frac{\xx^{1/2}}{\xx^{1/2} + \vA/2} , 
\label{mu12xx12vp1m1}
\end{EQA}
so that \( \mu / (1 - \mu) = 2 \xx^{1/2}/\vA \). Then
\begin{EQA}
	&& \nquad
	- \mu \bigl( \vA \xx^{1/2} + \xx \bigr) + \frac{\mu^{2} \vA^{2}}{4 (1 - \mu)}
	\\
	&=&
	- \mu \bigl( \vA \xx^{1/2} + \xx + \vA^{2}/4 \bigr)
	+ \frac{\mu \vA^{2}}{4 (1 - \mu)}
	\\
	&=&
	- \frac{\xx^{1/2}}{\xx^{1/2} + \vA/2} \bigl( \xx^{1/2} + \vA/2 \bigr)^{2} 
	+ \frac{2 \xx^{1/2} \vA }{4}
	=
	- \xx 
\label{mux2xv4x12}
\end{EQA}
and the result follows.
\end{proof}

Now we apply the Markov inequality 
\begin{EQA}
	&& \nquad
	\log \P\bigl( \gaussv^{\T} \BB \gaussv > \dimA + 2 \vA \xx^{1/2} + 2 \xx \bigr)
	=
	\log \P\bigl( (\gaussv^{\T} \BB \gaussv - \dimA) / 2 > \vA \xx^{1/2} + \xx \bigr)
	\\
	& \leq &
	\inf_{\mu > 0} \biggl\{ 
		- \mu \bigl( \vA \xx^{1/2} + \xx \bigr) 
		+ \log\E \exp\bigl\{ \mu (\gaussv^{\T} \BB \gaussv - \dimA)/2 \bigr\}
	\biggr\}
	\\
	& \leq &
	\inf_{\mu > 0} \biggl\{ 
		- \mu \bigl( \vA \xx^{1/2} + \xx \bigr) + \frac{\mu^{2} \vA^{2}}{4 (1 - \mu)}
	\biggr\}
	\leq 
	- \xx
\label{x2xv4x12}
\end{EQA}
and the first assertion \eqref{Pxiv2dimAvp12} follows.
The second statement follows from the first one by 
\( \tr(\BB^{2}) \leq \| \BB \|_{\oper} \tr(\BB) = \supA \, \dimA \).

Similarly for any \( \mu > 0 \)
\begin{EQA}
	\P\bigl( \gaussv^{\T} \BB \gaussv - \dimA < - 2 \vA \sqrt{\xx} \bigr)
	& \leq &
	\exp\bigl( - \mu \vA \sqrt{\xx} \bigr)
	\E \exp\Bigl( - \frac{\mu}{2} (\gaussv^{\T} \BB \gaussv - \dimA) \Bigr) .
\end{EQA}
By \eqref{lEemux2p2}
\begin{EQA}
	\log \E \exp\bigl\{ - \mu (\gaussv^{\T} \BB \gaussv - \dimA)/2 \bigr\}
	&=&
	\frac{1}{2} \sum_{j=1}^{\dimp} 
		\bigl\{ \mu \lambda_{j} - \log(1 + \mu \lambda_{j}) \bigr\} .
\label{lEemux2p2m}
\end{EQA}
and 
\begin{EQA}
	\frac{1}{2} \sum_{j=1}^{\dimp} \bigl\{ \mu \lambda_{j} - \log(1 + \mu \lambda_{j}) \bigr\}
	&=&
	\frac{1}{2} \sum_{j=1}^{\dimp} \sum_{k=2}^{\infty} \frac{(- \mu \lambda_{j})^{k}}{k}
	\leq 
	\sum_{j=1}^{\dimp}\frac{(\mu \lambda_{j})^{2}}{4} 
	=
	\frac{\mu^{2} \vA^{2}}{4} .
\label{jmu2v221mum}
\end{EQA}
Here the choice \( \mu = 2 \sqrt{\xx} / \vA \) yields \eqref{Pxiv2dimAvp12m}.

One can put together the arguments used for obtaining the lower and the upper bound 
for getting a bound for a general 
quadratic form \( \gaussv^{\T} \BB \gaussv \), where \( \BB \) is symmetric but not necessarily 
positive.
\end{proof}
}

\subsection{Sandwich lemma}

\begin{lemma} \label{sandwich}  Let differentiable measure $\Pb$ depends on r.v. from a continuous measure $\P$, 
 $z = (z_1, \ldots, z_K )$ is  a multivariate quantile. Assume following error in total variation distance between measures   

\[
 \bigg | \P \bigg(z_1, \ldots, z_K \bigg)  -
  \Pb \bigg(z_1, \ldots,  z_K \bigg) \bigg | < \delta.
\]
Then each quantile $z_k^{\flat}(\alpha)$, $1 \leq  k \leq K$  from measure $\Pb$ may be bounded by quantile from measure~$\P$: 
\[
z_k(\alpha + \delta) \leq  z_k^{\flat}(\alpha) \leq z_k(\alpha - \delta),  
\]
where
\[
 \P \bigg( z_k(\alpha) \bigg) = 1 - \alpha, 
\quad
 \Pb \bigg(  z^{\flat}_k(\alpha) \bigg) = 1 - \alpha.
\]
And if $q^{\flat}$ is the multiplicity correction parameter such that 
\[
\Pb \bigg(z_1(q^{\flat} \alpha), \ldots,  z_K(q^{\flat} \alpha) \bigg)  = 1 - \alpha,
\] 
then 

\[
 \bigg | \P \left( z^{\flat}( q^{\flat} \alpha) \right) - (1 - \alpha)  \bigg | 
 \leq  (2K + 1) \, \delta.
\]

\end{lemma}

\begin{proof}

Define two sets 
\[
\mathbb{Z}_{+}(\delta) = \{z: \P (z) \leq  1- \alpha + \delta \},
\]
\[
\mathbb{Z}_{-}(\delta)  = \{z: \P (z) \geq  1- \alpha - \delta \}.
\]
For all points $z$ from $\mathbb{Z}_{+} \cap \mathbb{Z}_{-}$  it holds  that $\left | \P \left( z \right) - (1 - \alpha)
\right | \leq \delta$.  If  $\Pb(z^{\flat})  = 1 $  then $ z^{\flat } \in \mathbb{Z}_{+}  $ since for all fixed  $z \in \R^{K} \setminus \mathbb{Z}_{+}$:    $\Pb(z) >  \P(z ) - \delta \geq  1 - \alpha$. Analogically  $ z^{\flat } \in \mathbb{Z}_{-}  $ and  $ z^{\flat } \in \mathbb{Z}_{+} \cap \mathbb{Z}_{-}  $. 

In case $K = 1$ one can choose non-random quantiles in the border of 
$\mathbb{Z}_{+} \cap \mathbb{Z}_{-} $ which will bound $z^{\flat}$. 
So each component of $z^{\flat}$ could be bounded in the same way:
\[
z_k(\alpha + \delta) \leq  z^{\flat}(\alpha) \leq  z_k(\alpha - \delta).
\]
In case $K > 1$ the these bounds become random because of multiplicity correction which involves random multiplier $q^{\flat}$. 
We have to bound $z_k(q^{\flat}\alpha + \delta)$ by a non-random quantile in order to use it as an argument for measure $\P$.  
\begin{EQA}
z(\alpha + \delta) \leq  z^{\flat}(\alpha) &\leq & z(\alpha - \delta) \\
 z^{\flat}( q^{\flat} \alpha) &\leq & z( q^{\flat} \alpha - \delta)  \leq z^{\flat}( q^{\flat} \alpha - 2 \delta) \\ 
   1 - \alpha &\leq & \Pb \bigg( z( q^{\flat} \alpha - \delta)  \bigg)  \leq  \Pb \bigg( z^{\flat}( q^{\flat} \alpha - 2 \delta) \bigg)
\end{EQA}

\begin{lemma}
For differentiable measure  $\P(\xi < x)$ and event $A$:
\[
\frac{\P(\xi < x, A)'_x}{\P(\xi < x)'_x} \leq 1.
\]
\end{lemma}

\[
\Pb \bigg( z^{\flat}( q^{\flat} \alpha - 2 \delta) \bigg) = 
\Pb \bigg( z^{\flat}( q^{\flat} \alpha) \bigg) + \sum_i (\Pb)'_{z^{\flat}_k} (z_k^{\flat})' \,\, 2 \delta,
\]
\[
(\Pb)'_{z_k}  (z_k^{\flat})' = \frac{(\Pb(z_1^{\flat}, \ldots , z^{\flat}_k , \ldots, z^{\flat}_K))'_{z_k^{\flat}}  }{(\Pb(z_k^{\flat}))'_{z_k^{\flat}}} \leq 1,
\]
\[
\Pb \bigg( z^{\flat}( q^{\flat} \alpha - 2 \delta) \bigg)  \leq 
1 - \alpha + 2 K \delta. 
\]
\begin{EQA}
1 - \alpha &\leq & 
 \Pb \bigg( z( q^{\flat} \alpha - \delta)  \bigg)  \leq  
 1 - \alpha + 2 K \delta, \\
1 - \alpha - 2 K \delta & \leq & 
 \Pb \bigg( z( q^{\flat} \alpha + \delta)  \bigg)  \leq  
 1 - \alpha. 
\end{EQA}
According to the arguments from the beginning of the proof $z( q^{\flat} \alpha - \delta) $ and $z( q^{\flat} \alpha - \delta) $ belongs to $\mathbb{Z}_{+} (2K\delta + \delta) \cap \mathbb{Z}_{-}(2K\delta + \delta) $. Due to one dimensional parametrization if $z(q^{\flat} \alpha - \delta)$ there exist two fixed points on the border of $\mathbb{Z}_{+} (2K\delta + \delta) \cap \mathbb{Z}_{-}(2K\delta + \delta) $ such that
\[
z_+ = \max z( q^{\flat} \alpha - \delta), 
\quad
z_- = \min z( q^{\flat} \alpha + \delta). 
\]
Finally, 
\[
z_-\leq  z( q^{\flat} \alpha + \delta) \leq z^{\flat}( q^{\flat} \alpha) \leq z( q^{\flat} \alpha - \delta) \leq  z_+,  
\]
and subsequently
\[
1 - \alpha - (2K + 1) \delta \geq \P(z_-) \leq  \P( z^{\flat}( q^{\flat} \alpha)) \leq \P( z_+) \leq 1 - \alpha + (2K + 1) \delta.
\]

\end{proof}

\subsection{TP variance} 
\label{tpvar}

Assume here that $\sum_t P_\tau(t) = 0$ and $\E \| \xi_{lr}(t) \|  = \sqrt{p}$ then $\sum_t P_\tau (t) \E \| \xi_{lr}(t) \| =0 $ and
\[
\Var \left ( \sum_t P_\tau (t) \| \xi_{lr}(t) \| \right )  \leq  \sum_{t = 1}^{2h}  P^2_\tau (t)  \Var \left( \sqrt{\sum_{t = 1}^{2h} \| \xi_{lr}(t) \| ^2  }  \right).
\] 
\[
\Var \left( \sqrt{\sum_{t = 1}^{2h} \| \xi_{lr}(t) \| ^2  }  \right) = 2h p - \left( \E \sqrt{\sum_{t = 1}^{2h} \| \xi_{lr}(t) \| ^2  }  \right)^2.
\]
Since each $\| \xi_{lr}(t) \|^2$ is close to normal random vector norm, its lower bound may be taken from Theorem \ref{Lxiv2LD}. So with probability $e^{-\xx}$ for all $1 \leq t \leq 2h$
\[
\| \xi_{lr}(t) \|^2 > p  - 2\sqrt{p} (\xx + \log(2h) )^{1/2}.
\]
Then with probability $e^{-\xx}$
\[
\Var \left( \sqrt{\sum_{t = 1}^{2h} \| \xi_{lr}(t) \| ^2  }  \right)  \leq  4h \sqrt{p} (\xx + \log(2h) )^{1/2}.
\]
Consider triangle pattern example 
\[\tag{P}\label{tr_patt}
P_\tau(t) = 
\begin{cases}
0, & t < \tau - h, \\ 
(t - \tau)/h + 1/2, & \tau - h \leq t \leq \tau, \\
(\tau - t) /h + 1/2,  &  \tau \leq t \leq \tau +h, \\
0, & t > \tau + h.
\end{cases} 
\] 
In this case through integral sum one get 
\[
\sum_{t = 1}^{2h}  P^2_\tau (t) \approx \frac{1}{6} h.
\]
Finally 
\[
\Var \left ( \sum_t P_\tau (t) \| \xi_{lr}(t) \| \right )  \leq \frac{2}{3} h^2 \sqrt{p} (\xx + \log(2h) )^{1/2},
\]
and 
\[
z_h(\xx) \leq  \sqrt{ \frac{2}{3} } h p^{1/4} (\xx + \log(2h) )^{1/4} e^{\xx/2}.
\]
The abrupt type change point statistic without noise component has view \[
\sqrt{2 T_{h}(t) } =  (P_{\tau}(t) + 1/2) \Delta \pm 7  \diamondsuit, 
\quad \tau -h \leq t \leq \tau + h.
\] 
The sufficient condition for change point detection of size $\Delta$  in position $\tau$ is  
\[
\sum_{t = \tau -h }^{\tau + h} P_{\tau}(t) \sqrt{2 T_{h}(t) }   > z_h(\xx).
\]
Equivalently
\[
\frac{1}{6} \Delta h + \frac{7}{2} h  \diamondsuit > \sqrt{ \frac{2}{3} } h p^{1/4} (\xx + \log(2h) )^{1/4} e^{\xx/2}.
\]

  \end{document}